\documentclass[twoside,11pt,reqno]{amsart}
\usepackage[T1]{fontenc}
\DeclareSymbolFont{symbols}{OMS}{cmsy}{m}{n}
\usepackage{fullpage}

\newcommand{\commentout}[1]{}

\usepackage[utf8]{inputenc}
\usepackage[T1]{fontenc}

\usepackage{mathrsfs}
\usepackage{mathtools}
\usepackage{amsfonts}
\usepackage{amsmath,amssymb,amsthm}
\numberwithin{equation}{section} %number equations after sections
\usepackage{comment}
\usepackage{enumerate}
\usepackage{graphicx,tikz}
\usepackage{ifthen}
\usetikzlibrary{arrows,scopes}
\usetikzlibrary{calc,cd}
\usepackage{setspace}
\usepackage[hidelinks]{hyperref} %hidelinks to hide red boxes
\usepackage[alphabetic,initials]{amsrefs} %put alphabetic if you want [LR95] 
\newtheorem{thm}{Theorem}[section]
\newtheorem{manualtheoreminner}{Theorem} %manual number theorem
\newenvironment{manualthm}[1]{%
  \renewcommand\themanualtheoreminner{#1}%
  \manualtheoreminner
}{\endmanualtheoreminner}

\newtheorem{lem}[thm]{Lemma}
\newtheorem{cor}[thm]{Corollary}
\newtheorem{prop}[thm]{Proposition}
\theoremstyle{definition}
\newtheorem{defi}[thm]{Definition}
\newtheorem{notation}[thm]{Notation}
\newtheorem{conjecture}[thm]{Conjecture}

\newtheorem{example}[thm]{Example}
\theoremstyle{remark}
\newtheorem{rmk}[thm]{Remark}

\newcommand{\Hil}{\mathcal{H}}
\newcommand{\Kil}{\mathcal{K}}
\newcommand{\slot}{{~\cdot~}}

\DeclareMathOperator{\Spin}{Spin}

\DeclareMathOperator{\SO}{SO}
\DeclareMathOperator{\SU}{SU}

\DeclareMathOperator{\PSU}{PSU}
\DeclareMathOperator{\U}{U}

\DeclareMathOperator{\Mob}{{M\ddot ob}}
\DeclareMathOperator{\Mobt}{{M\ddot ob}_2}
\newcommand{\Sc}[1][]{\mathbb{S}^{1#1}}

\newcommand{\C}{\mathcal{C}}

\newcommand{\cF}{\mathcal{F}}

\newcommand{\cU}{\mathcal{U}}

\newcommand{\K}{\mathcal{K}}

\newcommand{\cP}{\mathcal{P}}
\newcommand{\cC}{\mathcal{C}}

\newcommand{\A}{\mathcal{A}}
\newcommand{\B}{\mathcal{B}}
\newcommand{\cI}{\mathcal{I}}

\newcommand{\M}{\mathcal{M}}

\newcommand{\N}{\mathcal{N}}

\newcommand{\RR}{\mathbb{R}}

\newcommand{\CC}{\mathbb{C}}

\newcommand{\ZZ}{\mathbb{Z}}
\newcommand{\NN}{\mathbb{N}}

\DeclareMathOperator{\End}{End}

\DeclareMathOperator{\Rep}{Rep}

\DeclareMathOperator{\Hom}{Hom}
\DeclareMathOperator{\Ind}{Ind}

\DeclareMathOperator{\Mor}{Mor}
\DeclareMathOperator{\Vir}{Vir}

\DeclareMathOperator{\Ad}{Ad}

\DeclareMathOperator{\id}{id}
\DeclareMathOperator{\Tr}{Tr}

\DeclareMathOperator{\Aut}{Aut}

\newcommand{\dd}{\,\mathrm{d}}

\usepackage{xspace}
\usepackage{xspace}
\DeclareRobustCommand{\eg}{e.g.\@\xspace}
\DeclareRobustCommand{\Eg}{E.g.\@\xspace}
\DeclareRobustCommand{\cf}{cf.\@\xspace}
\DeclareRobustCommand{\Cf}{Cf.\@\xspace}
\DeclareRobustCommand{\ie}{i.e.\@\xspace}
\DeclareRobustCommand{\p}{p.\@\xspace}

\DeclareRobustCommand{\Sec}{Sec.\@\xspace}
\DeclareRobustCommand{\Ch}{Ch.\@\xspace}
\DeclareRobustCommand{\Prop}{Prop.\@\xspace}
\DeclareRobustCommand{\Lem}{Lem.\@\xspace}
\DeclareRobustCommand{\Cor}{Cor.\@\xspace}
\DeclareRobustCommand{\Thm}{Thm.\@\xspace}
\DeclareRobustCommand{\Ch}{Ch.\@\xspace}
\DeclareRobustCommand{\App}{App.\@\xspace}
\DeclareRobustCommand{\Ex}{Ex.\@\xspace}
\DeclareRobustCommand{\Def}{Def.\@\xspace}
\DeclareRobustCommand{\Rmk}{Rmk.\@\xspace}
\DeclareRobustCommand{\Eq}{Eq.\@\xspace}

\makeatletter
\DeclareRobustCommand{\etc}{%
    \@ifnextchar{.}%
        {etc}%
        {etc.\@\xspace}%
}
\makeatother

\newcommand{\Cstar}{$C^\ast$\@\xspace}
\newcommand{\Wstar}{$W^\ast$\@\xspace}

\def\u1net{{\A_\RR}}

\DeclareMathOperator*{\Span}{Span}

\DeclareMathOperator{\Sect}{Sect}

\DeclareMathOperator{\Proj}{Proj}

\DeclareMathOperator{\Trig}{Trig}

\DeclareMathOperator*{\Irr}{Irr}

\DeclareMathOperator*{\UCP}{UCP}
\DeclareMathOperator{\Extr}{Extr}
\def\III{{I\!I\!I}}
\def\II{{I\!I}}

\newcommand{\CS}{/\!\!/} %Double Cosets

\renewcommand{\slot}{\,\cdot\,}

\newcommand{\one}{\mathbf{1}}
\newcommand{\Cred}{{C^\ast_\mathrm{red}}}

\newcommand{\Endd}{\End_\mathrm{d}}

\newcommand{\tikzmatht}[2][0.30]
{\vcenter{\hbox{\begin{tikzpicture}[scale=#1]#2
				 \end{tikzpicture}}}
}
\newcommand{\colM}{black!10}

\newcommand{\colN}{black!5}

\newcommand{\iotabar}{\bar\iota}
\newcommand{\psibar}{\bar\psi}
\newcommand{\rhobar}{\bar\rho}
\newcommand{\sigmabar}{\bar\sigma}
\newcommand{\taubar}{\bar\tau}
\newcommand{\phisharp}{\phi^\sharp}
\newcommand{\rbar}{\bar r}
\newcommand{\idK}{e}

\makeatletter
\def\subsection{\@startsection{subsection}{2}%
  \z@{.5\linespacing\@plus.7\linespacing}{.3\linespacing}%
  {\normalfont\bfseries}}
\makeatother

\usepackage{mathtools}
\mathtoolsset{showonlyrefs=true} %but use \eqref{eq:...} instead of (\ref{eq:...})

\begin{document}

\date{} %\date{\today}
\dateposted{} %\dateposted{\today}

\title{Compact Hypergroups from Discrete Subfactors}
  
\address{Department of Mathematics, Morton Hall 321, 1 Ohio University, Athens, OH 45701, USA}
\author{Marcel Bischoff}
\email{bischoff@ohio.edu}
\address{Institute for Theoretical Physics, Faculty of Physics, University of Leipzig, Br\"uderstra\ss e 16, D-04103 Leipzig, Germany}
\author{Simone Del Vecchio}
\email{simone.del\_vecchio@physik.uni-leipzig.de}
\address{Dipartimento di Matematica, Universit\`a di Roma Tor Vergata, Via della Ricerca Scientifica, 1, I-00133 Roma, Italy 
\emph{and} Department of Mathematics, Vanderbilt University, 1326 Stevenson Center, Nashville, TN 37240, USA}
\author{Luca Giorgetti}
\email{luca.giorgetti@vanderbilt.edu}
\thanks{M.B.\ is supported by NSF DMS grant 1700192/1821162 \emph{Quantum Symmetries and Conformal Nets}. S.D.V.\ is supported by the Deutsche Forschungsgemeinschaft (DFG) within the Emmy Noether grant CA1850/1-1. L.G.\ is supported by the European Union's Horizon 2020 research and innovation programme H2020-MSCA-IF-2017 under Grant Agreement 795151 \emph{Beyond Rationality in Algebraic CFT: mathematical structures and models}. Part of this research was supported by ERC Advanced Grant QUEST 669240 \emph{Quantum Algebraic Structures and Models}}

\begin{abstract}
Conformal inclusions of chiral conformal field theories, or more generally inclusions of quantum field theories, are described in the von Neumann algebraic setting by nets of subfactors, possibly with infinite Jones index if one takes non-rational theories into account. With this situation in mind, we study in a purely subfactor theoretical context a certain class of braided discrete subfactors with an additional commutativity constraint, that we call locality, and which corresponds to the commutation relations between field operators at space-like distance in quantum field theory. Examples of subfactors of this type come from taking a minimal action of a compact group on a factor and considering the fixed point subalgebra.

We show that to every irreducible local discrete subfactor $\mathcal{N}\subset\mathcal{M}$ of type ${I\!I\!I}$ there is an associated canonical compact hypergroup (an invariant for the subfactor) which acts on $\mathcal{M}$ by unital completely positive (ucp) maps and which gives $\mathcal{N}$ as fixed points. To show this, we establish a duality pairing between the set of all $\mathcal{N}$-bimodular ucp maps on $\mathcal{M}$ and a certain commutative unital $C^*$-algebra, whose spectrum we identify with the compact hypergroup.

If the subfactor has depth 2, the compact hypergroup turns out to be a compact group. This rules out the occurrence of compact \emph{quantum} groups acting as global gauge symmetries in local conformal field theory.
\end{abstract}

\maketitle
\tableofcontents

%DRAFTONLY
\setcounter{tocdepth}{3}
%\setcounter{tocdepth}{2}
%\tableofcontents

%%%
\section{Introduction}
%%%

Subfactors coming from inclusions of nets of local algebras, \ie the von Neumann algebraic description of quantum field theories (QFT) \cite{Ha}, have a rich structure and they are highly constrained by physical requirements. We aim to describe them by means of intrinsically defined symmetries that act on the bigger factor and that characterize the smaller factor as the fixed points under the action, extending the results of \cite{Bi2016} to discrete subfactors (possibly with infinite index). 

We restrict for ease of exposition to the case of inclusions of chiral, \ie one-dimensional, conformal field theories $\A\subset\B$ \cite{Lo2} and we keep these in mind as the main source of subfactors with the properties we deal with in this work. The inclusions $\N = \A(I) \subset\M = \B(I)$, for a fixed interval $I$ of the unit circle, are often irreducible type $\III$ subfactors \cite{Ca2004}, \cite{BrGuLo1993}. They are equipped with a normal faithful vacuum preserving conditional expectation which encodes the global gauge symmetries of the inclusion \cite{LoRe1995}. More importantly, the tensor \Cstar-category of $\N$--$\N$ bimodules generated by the standard bimodule ${}_\N{L^2\M}_\N$ is automatically unitarily braided due to the commutation relations between local algebras $\A(I)$ and $\A(J)$ sitting at disjoint intervals $I$ and $J$. These commutation relations, often called locality of the net, describe Einstein's causality principle in algebraic QFT. The braiding is the Doplicher--Haag--Roberts (DHR) braiding, introduced and studied in the algebraic theory of superselection sectors \cite{DoHaRo1971}, \cite{FrReSc1992}. Thus the subfactors $\N\subset\M$ arising from conformal inclusion are automatically braided (Definition \ref{def:braidedsubf}) if the net $\A$ is local.

The locality of the net $\B$ can be encoded as well at the level of a single subfactor, over a single interval $I$. To do this we need at the present stage a further assumption which is not imposed by physics, or at least which is not fulfilled in interesting models of chiral conformal field theory \cite{Re1994}, \cite{Ca2003}.  
This assumption was already present in the literature of subfactors with infinite Jones index \cite{Jo1983}, \cite{HeOc1989}, \cite{IzLoPo1998}. It is equivalent to the existence of sufficiently many\footnote{Precisely, forming a Pimsner--Popa basis for the inclusion \cite{PiPo1986}, \cite{Po1995a}.} charged fields (Section \ref{sec:genQsys}), in the sense of Doplicher and Roberts \cite{DoRo1972}, with whom to describe the extension $\B$ by means of $\A$ \cite{DeGi2018}. The assumption is called discreteness of the subfactor \cite{IzLoPo1998}, or quasi-regularity in the terminology of Popa \cite{Po1999}. We remark that every subfactor with finite Jones index is discrete, and that in ordinary 3+1-dimensional QFT a condition which is equivalent to discreteness is often assumed, namely the absence of irreducible superselection sectors with infinite statistics \cite{DoRo1990}. 

Denote by $\psi_{\rho}$, $\psi_{\sigma}$ the charged fields associated with the superselection sectors $[\rho]$, $[\sigma]$ appearing in the direct sum decomposition of the vacuum sector of $\B$ when restricted to $\A$ \footnote{$\rho$, $\sigma$ are endomorphisms of $\N$ and they correspond to the irreducible sub-bimodules of ${}_\N {L^2\M}_\N$} and denote by $\varepsilon_{\sigma,\rho}$ the DHR braiding.
Then the locality of the net $\B$ is characterized by the following commutation relations $\varepsilon_{\sigma,\rho}\, \psi_{\sigma} \psi_{\rho} = \psi_{\rho} \psi_{\sigma}$, for every pair $\psi_\rho, \psi_\sigma$ \cite{DeGi2018}. With this situation in mind, we call a discrete subfactor local (Definition \ref{def:localsubf}) if it is braided and if its charged fields fulfill the previous commutation relations with respect to the given braiding.  

In this work, under the assumptions of irreducibility, discreteness and locality, we associate to a subfactor $\N\subset\M$ an invariant which describes $\N$ in terms of $\M$ by means of \lq\lq generalized gauge symmetries\rq\rq. The invariant is a compact metrizable space, denoted by $K(\N\subset\M)$, equipped with the structure of a hypergroup (Definition \ref{def:CompactHypergroup}). $K(\N\subset\M)$ acts on $\M$ by unital completely positive maps and it gives $\N$ as the fixed point subalgebra $\M^{K(\N\subset \M)}$. The conditional expectation of the subfactor is given by the average of the action with respect to the Haar measure of $K(\N\subset\M)$. 

In general terms, a compact hypergroup is a compact Hausdorff topological space with an abstract convolution product and an involution defined on its probability Radon measures. The convolution of two Dirac measures concentrated respectively in the points $x$ and $y$ is again a Dirac measure concentrated in a single point $z$ precisely when the hypergroup is a group, in which case $z = xy$. We remark that there are several notions of hypergroup available in the literature, and the one we refer to is intermediate between the most restrictive one of Dunkl, Jewett and Spector \cite{BlHe1995} and the one of hypercomplex system \cite{BeKa1998}. In the case of finite topological spaces they all boil down to the purely algebraic notion of finite hypergroup.

The compact hypergroup of a subfactor is very natural to define. One considers the compact convex set of all $\N$-bimodular ucp maps on $\M$, and defines $K(\N\subset\M)$ to be the subset of extreme points. The non-trivial part is to show that $K(\N\subset\M)$ is a compact metrizable space and to endow its space of probability measures with a convolution and an involution fulfilling the axioms of a hypergroup. To do this we identify as topological spaces the set of all $\N$-bimodular ucp maps on $\M$ with the state space of a canonical commutative unital separable \Cstar-algebra associated with the subfactor, denoted by $\Cred(\N\subset \M)$. This is the main technical result of the paper (Theorem \ref{thm:ucpstatesduality}). In particular, the extreme ucp maps, \ie the elements of $K(\N\subset\M)$, are identified with the Gelfand spectrum of the algebra. The conditional expectation of the subfactor is identified on one side with the Haar measure on $K(\N\subset\M)$, on the other side with the state defining the \Cstar-norm of $\Cred(\N\subset \M)$ via the GNS construction.

As a main application of our analysis, we show that if the subfactor has depth 2, then $K(\N\subset\M)$ is in fact a compact metrizable group acting by automorphisms. This extends from finite to infinite compact groups a result known in the language of conformal nets \cite{Bi2016} and of vertex operator algebras \cite{DoWa2018}, stating the absence of purely Hopf algebra orbifolds in local conformal field theory.
The main results contained in this work have been announced in \cite{BiOWR} and \cite{GiOWR}.

We give below a brief account on the contents of this paper. In Section \ref{sec:discretesubf}, we review some definitions and results on infinite index type $\III$ subfactors, in particular the notion of discreteness (Definition \ref{def:semi-discr}) due to Izumi, Longo and Popa \cite{IzLoPo1998} and its characterization by means of Pimsner--Popa bases of charged fields or equivalently by means of generalized Q-systems of intertwiners (Proposition \ref{prop:discrete} and \ref{prop:Qsysdiscrete}). We introduce the notion of local discrete subfactor (Definition \ref{def:localsubf}) as a special case of braided subfactor (Definition \ref{def:braidedsubf}). We show that a certain invariant introduced in \cite{IzLoPo1998} describing the modular flow on charged fields is trivial for local subfactors (Proposition \ref{prop:arho=1}) as it is for finite index ones, and we draw some conclusions. We introduce the notion of strong relative amenability for irreducible discrete type $\III$ subfactors (Definition \ref{def:strongamenab}), inspired by the type $\II_1$ analogue due to \cite{PoShVa2018}, and we show that braided subfactors are strongly relatively amenable (Corollary \ref{cor:braideddiscreteisstrongamenab}).

In Section \ref{sec:compacthyp}, we specify our notion of compact hypergroup (Definition \ref{def:CompactHypergroup}) and we comment on its collocation in the literature of classical locally compact and compact quantum hypergroups.

In Section \ref{sec:compacthypfromsubf}, which is the main part of this paper, we define the compact hypergroup $K(\N\subset\M)$ and we prove our duality theorem (Theorem \ref{thm:ucpstatesduality}). 
In more detail, in Section \ref{sec:trig} we define a unital associative commutative *-algebra, denoted by $\Trig(\N\subset\M)$, whose multiplication and involution reflect the fusion and the conjugation of the irreducible subsectors $[\rho]$, $[\sigma]$ (Theorem \ref{thm:trigcommutativestaralg}). Apart from the commutativity condition,
this *-algebra is a type $\III$ analogue of a corner of the generalized tube *-algebra associated with an irreducible discrete type $\II_1$ subfactor due to Popa, Shlyakhtenko and Vaes \cite{PoShVa2018}.
See \cite{JoPe2019} for another generalization in this direction with $\N$ of type $\II_1$ and $\M$ possibly of type $\III$. 
We also conjecture that a unified description of our construction can be given in the purely tensor \Cstar/\Wstar-categorical language of \cite{JoPe2019}. 

In Section \ref{sec:dualitypairing}, we define the duality paring \eqref{eq:pairing}, denoted by
$$\langle\slot,\slot\rangle\colon {\UCP}_{\N}(\M)\times \Trig(\N\subset \M)\to \CC,$$ 
between the $\N$-bimodular ucp maps on $\M$ and $\Trig(\N\subset\M)$.
In Section \ref{sec:fourier}, we recall the definition of the subfactor theoretical Fourier transform in the infinite index case, and we study some of its properties for later use. In Section \ref{sec:Cistarred}, we define the commutative unital \Cstar-algebra $\Cred(\N\subset\M)$ as the closure of $\Trig(\N\subset\M)$ in the bounded and faithful GNS representation of the state $\omega_E = \langle E, \slot \rangle$ given by the conditional expectation $E : \M \to \N\subset\M$. In Section \ref{sec:UCP/states}, we show that every $\N$-bimodular ucp map on $\M$ is automatically vacuum preserving and adjointable in the sense of Section \ref{sec:Omegaadjmaps}, \ie it belongs to the set ${\UCP}^{\sharp}_\N(\M,\Omega)$ (Definition \ref{def:UCPbim}). 
We then prove:

\begin{manualthm}{4.34}%\label{thm:ucpstatesduality}
Let $\N\subset \M$ be an irreducible local discrete type $\III$ subfactor. The duality map
\begin{align}
\phi \mapsto \omega_\phi := \langle \phi, \slot \rangle
\end{align}
is a homeomorphism between:
\begin{itemize}
\item the set of normal faithful $\N$-bimodular ucp maps $\phi\colon \M \to \M$ equipped with the pointwise weak operator topology,
\item the set of states $\omega$ on $\Cred(\N\subset \M)$ equipped with the weak$^*$ topology.
\end{itemize}
In particular, $\phi$ is extreme if and only if $\omega_\phi$ is a pure state, \ie a character of $\Cred(\N\subset \M)$.
\end{manualthm}

In Section \ref{sec:UCPcompacthyp}, we define $K(\N\subset\M)$ as the set of all extreme normal faithful $\N$-bimodular ucp maps (Definition \ref{def:CompactHypergroupUCP}). In Theorem \ref{thm:Kishypergroup}, we show that $K(\N\subset\M)$ is a compact hypergroup: the convolution of probability measures is given by the composition of ucp maps, the inverse measure is the adjoint ucp map, the Haar measure is $\omega_E$. 

In Section \ref{sec:actionK}, we study the action of $K(\N\subset\M)$ on $\M$ as an abstract compact hypergroup (Definition \ref{def:action}). In Theorem \ref{thm:genorbi}, we show that $\N\subset\M$ is described as a generalized orbifold:

\begin{manualthm}{5.7}%\label{thm:genorbi}
Let $\N\subset \M$ be an irreducible local discrete type $\III$ subfactor.
The compact hypergroup $K(\N\subset\M)$ acts faithfully on $\M$ and it gives $\N$ as the fixed point subalgebra, $\N = \M^{K(\N\subset\M)}$. 
\end{manualthm}

In Section \ref{sec:repK}, we study the representation theory of $K(\N\subset\M)$ as an abstract compact hypergroup. In particular, we show that $\Trig(\N\subset\M)$ can be identified with the algebra of trigonometric polynomials in the sense of Vrem \cite{Vr1979}. Moreover, we prove the equality between the hyperdimension of a representation \cite{AmMe2014} and the tensor \Cstar-categorical dimension of the associated endomorphism \cite{LoRo1997} (Theorem \ref{thm:hyperdim}).

In Section \ref{sec:depthtwo}, we inspect the special case of depth 2 subfactors, \eg those obtained as the fixed points under the action of a Hopf *-algebra or compact quantum group. In Theorem \ref{thm:noquantum}, not assuming discreteness, we show that $K(\N\subset\M)$ is a classical compact group: 

\begin{manualthm}{7.5}%\label{thm:noquantum}
Let $\N\subset\M$ be an irreducible local semidiscrete type $\III$ subfactor with depth 2. Then $K(\N\subset\M)$ is the compact metrizable group of all *-automorphisms of $\M$ that act trivially on $\N$. In particular, $\N\subset\M$ is a compact group orbifold.
\end{manualthm}

In Section \ref{sec:nonlocal}, we extend some of our results to graded-local subfactors, with the application to Fermionic conformal nets in mind.

Finally, in Section \ref{sec:ex}, we study and compute $K(\N\subset\M)$ in examples coming from group orbifolds, double coset hypergroups and discrete inclusions of non-rational local conformal nets.

%%%
\section{Discrete subfactors}\label{sec:discretesubf}
%%%

%%%
\subsection{The 2-\Cstar-category of morphisms of type $\III$ factors}
%%%

Let $\N,\M$ be type $\III$ factors. Throughout this paper we shall always deal with separable Hilbert spaces 
and von Neumann algebras with separable predual without further mention. Let $\Mor(\N,\M)$ 
be the \Cstar-category of all normal unital *-homomorphisms $\rho\colon \N\to \M$. 
For $\rho,\sigma\in\Mor(\N,\M)$ the morphisms are given by the intertwining operators 
$\Hom(\rho,\sigma) := \{t\in \M: t\rho(n)=\sigma(n)t, n\in \N\}$.
We denote by $\End(\N):=\Mor(\N,\N)$ the tensor \Cstar-category of all normal endomorphisms of $\N$.
The tensor multiplication is given on objects by the composition $\rho\otimes \sigma := \rho\sigma$ for $\rho,\sigma\in \End(\N)$ 
and on morphisms by $r\otimes s := r \rho_1(s) = \rho_2(s) r$ for $r\in\Hom(\rho_1,\rho_2)$, $s\in\Hom (\sigma_1,\sigma_2)$.
The tensor unit is the identity automorphism $\id\in\End(\N)$.
We denote by $\End_0(\N)$ the rigid tensor \Cstar-subcategory of endomorphisms 
with finite intrinsic/tensor \Cstar-categorical \textbf{dimension} \cite{LoRo1997}. 

A \textbf{sector} is an isomorphism class $[\rho]$ where $\rho\in\End(\N)$. We denote by $\Sect(\N)$ the collection of all sectors in $\End(\N)$ and by $\Sect_0(\N)$ the collection of all sectors in $\End_0(\N)$.
Inside $\End(\N)$ we can consider infinite direct sums of endomorphisms with finite dimension. 
Namely, let $\rho_i\in\End_0(\N)$, with $i=1,\ldots, n$ and $n\in \NN\cup\{\infty\}$.
Since $\N$ is of type $\III$, there are isometries $v_i$ with $v_i^\ast v_j =\delta_{i,j}1$
and $\sum_i v_iv_i^\ast = 1$ in the strong operator topology, and we define
\begin{align}
  \bigoplus_i \rho_i := \sum_i v_i\rho_iv_i^\ast\,.
\end{align}
The direct sum is well-defined up to unitary equivalence. Indeed, given another family of isometries $\tilde{v}_i$ such that $\tilde{v}_i^\ast \tilde{v}_j =\delta_{i,j}1$ and $\sum_i \tilde{v}_i \tilde{v}_i^* = 1$ then $\sum_i \tilde{v}_i v_i^*$ is a unitary intertwiner. We denote the obtained category by $\Endd(\N)$.

A morphism $\rho\in\Mor(\N,\M)$ is called \textbf{irreducible} if $\Hom(\rho,\rho) = \CC1$. We say $\rho$ is contained in $\sigma\in\Mor(\N,\M)$, written $\rho\prec\sigma$, if there is an isometry in $\Hom(\rho,\sigma)$.
Let $\rho\in \Endd(\N)$, let $p\in \Hom(\rho,\rho)$ be a projection 
and $v\in \N$ be an isometry with $v v^\ast = p$, which exists since $\N$ is of type $\III$.
Denoted $\rho_p := \Ad(v^\ast )\circ\rho$, we clearly have $v\in\Hom(\rho_p,\rho)$ and $\rho_p \prec \rho$.
Note that the sector $[\rho_p]$ does not depend on the choice of isometry
and we have a map $\Proj(\rho)\to \Sect(\N)$, which restricts to a map $\Proj_0(\rho)\to \Sect_0(\N)$, where
\begin{align}\label{eq:Proj_def}
\Proj(\rho)& := \{p\in\Hom(\rho,\rho): p\text{ is a projection}\}\\
\Proj_0(\rho)& := \{p\in\Proj(\rho) : d(\rho_p) < \infty\}
\end{align}
and $\rho\mapsto d(\rho)$ is the dimension function \cite[\Sec 3]{LoRo1997}. Recall that $\rho\in\End(\N)$ has finite dimension if and only if it admits a \textbf{conjugate} $\rhobar\in\End(\N)$ with a pair of solutions $r_{\rho}\in\Hom(\id,\rhobar\rho)$, $\rbar_{\rho}\in\Hom(\id,\rho\rhobar)$ of the conjugate equations
\begin{align}\label{eq:conjeqns}
\rbar^*_{\rho} \rho(r_\rho) = 1_\rho,\quad r^*_{\rho} \rhobar(\rbar_\rho) = 1_{\rhobar}.
\end{align}
We shall always choose \textbf{standard solutions} $r_\rho, r_{\rhobar}$ of the conjugate equations \cite[\Sec 2]{LoRo1997}, namely those giving the dimension of $\rho$ via $d(\rho) 1_{\id} = r_\rho^* r_\rho = \rbar_{\rho}^* \rbar_{\rho}$. Recall that $d(\rho) = d(\rhobar)$.

%%%
\subsection{Generalized Q-systems}\label{sec:genQsys}
%%%

Let $\N\subset \M$ be a subfactor, \ie a unital inclusion of von Neumann factors in $\B(\Hil)$, the bounded linear operators on the separable Hilbert space $\Hil$. We denote by $\M'$ the commutant of $\M$ in $\B(\Hil)$. 
Let $E:\M\to\M$ be a normal faithful conditional expectation onto $\N$, written as $E:\M \to \N\subset\M$. Denote by $\M_1$ the Jones extension of $\M$ with respect to $E$ \cite{Jo1983}, and by $\hat E: \M_1 \to \M\subset\M_1$ the normal semifinite faithful operator-valued weight dual to $E$ \cite{Ko1986}.

\begin{defi}[{\cite{IzLoPo1998}, \cite{FiIs1999}}]\label{def:semi-discr}
A subfactor $\N\subset \M$ is called \textbf{semidiscrete} if it admits a normal faithful conditional expectation $E:\M\to \N\subset\M$.
It is called \textbf{discrete} if in addition the operator-valued weight $\hat E$ is semifinite on $\N'\cap \M_1$ 
for some (hence for all) $E$.   
\end{defi}

In this terminology, a subfactor has \textbf{finite index} in the sense of \cite{Jo1983}, \cite{Ko1986}, if it is semidiscrete and if $\hat E$ is everywhere defined and bounded. In particular, every finite index subfactor is discrete.

We now recall the definition of Pimsner--Popa basis for a subfactor or inclusion of von Neumann algebras equipped with a normal faithful conditional expectation $E:\M\to\N \subset \M$, with finite or infinite index. In the notation above, let $e_\N$ be the Jones projection for $\N\subset \M$ with respect to $E$. Namely, $e_\N$ is the orthogonal projection onto $\overline{\N\Omega}$, where $\Omega$ is a cyclic and separating vector for $\M$ which induces an $E$-invariant state on $\M$. 

\begin{defi}[{\cite{PiPo1986}, \cite{Po1995a}}]\label{def:PiPobasis}
A \textbf{Pimsner--Popa basis} for $\N\subset \M$ with respect to $E$ is a family of elements $\{M_i\} \subset \M$, where $i$ runs in some set of indices, such that
\begin{itemize}
\item[$(i)$] $M_i^* e_\N M_i$ are mutually orthogonal projections in $\M_1$,
\item[$(ii)$] $\sum_i M_i^* e_\N M_i = 1$, where the sum converges in the strong operator topology.
\end{itemize}
\end{defi}

Condition $(i)$ is equivalent to $E(M_iM_j^*) = 0$ if $i\neq j$ and $E(M_iM_i^*)$ are projections in $\N$, while condition $(ii)$ is equivalent to $\overline{\sum_i M_i^* e_\N \Hil} = \Hil$. Summing up, $M_i^* e_\N$ are partial isometries whose range projections form a partition of unit by mutually orthogonal projections.

Pimsner--Popa bases provide the following right $\N$-module expansion of the elements in $\M$. 

\begin{prop}[{\cite{Po1995a}}]\label{prop:PiPoexp}
Let $\{M_{i}\}$ be a Pimsner--Popa basis for $\N\subset \M$ with respect to $E$ then every $m\in \M$ admits the following expansion:
\begin{align}
m=\sum_i M_i^* E(M_i m),
\end{align}
where the sum converges in the topology generated by the family of seminorms $\{\|\cdot\|_{\varphi}:\varphi\in (\M_{*})_{+},\, \varphi = \varphi \circ E\}$, with $\|m\|_{\varphi}:=\varphi(m^{*}m)^{1/2}$. 

The coefficients $E(M_i m)$ are uniquely determined by $m$ if and only if $E(M_i M_i^*)=1$ for every $i$. 
\end{prop}

\begin{prop}[{\cite{FiIs1999}}] 
Let $\N$, $\M$ be infinite factors. Every semidiscrete subfactor $\N\subset \M$ equipped with a normal faithful conditional expectation $E:\M\to \N\subset\M$ admits a Pimsner--Popa basis with respect to $E$ made of elements in $\M$ (Definition \ref{def:PiPobasis}). 
\end{prop}

Furthermore, discrete subfactors can be characterized among the semidiscrete ones as those admitting Pimsner--Popa bases with an additional \emph{intertwining property}. To explain this we need to introduce Longo's canonical and dual canonical endomorphisms of the subfactor in question \cite{Lo1987}. Denote by $\gamma\in\End(\M)$ a \textbf{canonical endomorphism} of $\N\subset \M$ and by $\theta\in\End(\N)$ the corresponding \textbf{dual canonical endomorphism}. Namely, $\gamma := (\Ad(J_{\N,\Phi}) \circ \Ad(J_{\M,\Phi}))_{\restriction \M}$, where $J_{\N,\Phi}$ and $J_{\M,\Phi}$ are the modular conjugations associated with a jointly cyclic and separating vector $\Phi$ for $\N$ and $\M$, and $\theta := \gamma_{\restriction \N}$. Denote by $\iota:\N\to \M$ the inclusion morphism and by $\iotabar := \iota^{-1}\gamma : \M\to\N$ a weak conjugate in the sense of \cite{Lo1990}. Note that $\iota:\N\to\M$ and $\iotabar:\M\to\N$ are not surjective, unless $\N=\M$, and
\begin{align}\label{eq:iotaandiotabar}
\iota \iotabar = \gamma, \quad \iotabar \iota = \theta
\end{align}
by definition.

\begin{prop}[{\cite{DeGi2018}}]\label{prop:discrete} 
Let $\N$, $\M$ be infinite factors and $\N\subset \M$ a semidiscrete subfactor. Then $\N\subset\M$ is discrete if and only if for some (hence for all) $E:\M\to\N\subset\M$ there is a Pimsner--Popa basis in $\M$ with respect to $E$ whose elements fulfill $M_i \iota(n) = \iota(\theta(n)) M_i$ for every $n\in \N$, \ie
\begin{align}\label{eq:PiPointert}
M_i\in\Hom(\iota, \iota \theta)
\end{align}
for every $i$.
\end{prop}

Pimsner--Popa bases are also part of the data constituting a \emph{generalized Q-system} (Definition \ref{def:genQsys}), which is a way of describing an infinite subfactor $\N\subset\M$ together with an expectation $E:\M\to\N\subset\M$ by means of data only pertaining to $\N$. For infinite factors, the expectation $E$ has finite index \cite{Ko1986} if and only if one can choose a Pimsner--Popa basis made of one element. In this case, the subfactor together with the chosen expectation can be more effectively described by an ordinary Q-system in the sense of Longo \cite{Lo1994, BiKaLoRe2014-2}. 
See \cite[\Sec 2, 3]{DeGi2018} for more comments on this part.

\begin{defi}[{\cite{FiIs1999}}]\label{def:genQsys}
Let $\N$ be an infinite factor, a \textbf{generalized Q-system} in the tensor \Cstar-category $\End(\N)$ is a triple $(\theta, w, \{m_i\})$ consisting of an endomorphism $\theta\in\End(\N)$, an isometry $w\in\Hom(\id,\theta)$, \ie $wn = \theta(n)w$ for every $n\in\N$, and a family $\{m_i\}\subset\N$ such that:
\begin{itemize}
\item[$(i)$] $m_i^* ww^* m_i$ are mutually orthogonal projections in $\N$;
\item[$(ii)$] $\sum_i m_i^* ww^* m_i = 1$, where the sum converges in the strong operator topology;
\item[$(iii)$] $n w = 0$ implies $n = 0$ for $n\in \N_1$, where $\N_1 := \langle\theta(\N), \{m_i\}\rangle \subset \N$.
\end{itemize}
\end{defi}

Conditions $(i)$ and $(ii)$ above clearly resemble those appearing in the definition of Pimsner--Popa basis (Definition \ref{def:PiPobasis}) with $ww^*$ playing the role of a Jones projection. Condition $(iii)$ is a faithfulness condition for the conditional expectation encoded in the datum of the generalized Q-system.

\begin{prop}[{\cite{FiIs1999}}]\label{prop:fiis}
Let $\N$ be an infinite factor and let $\theta\in \End(\N)$, the following conditions are equivalent:
\begin{enumerate}
\item There is a von Neumann algebra $\N_1$ such that $\N_1\subset\N$ with a normal faithful conditional expectation $E':\N_1\to\N_2\subset\N_1$, where $\N_2 := \theta(\N)\subset\N_1$, and $\theta\in\End(\N)$ is a canonical endomorphism of $\N_1\subset\N$;
\item There is a von Neumann algebra $\M$
such that $\N\subset\M$ with a normal faithful conditional expectation $E:\M\to \N\subset\M$, and $\theta$ is a dual canonical endomorphism of $\N\subset\M$, \ie $\theta = \gamma_{\restriction\N}$ where $\gamma\in\End(\M)$ is a canonical endomorphism of $\N\subset\M$;
\item The endomorphism $\theta$ is part of a generalized Q-system $(\theta, w, \{m_i\})$ in the tensor \Cstar-category $\End(\N)$ (Definition \ref{def:genQsys}).
\end{enumerate}
\end{prop}

Note that the subfactors $\N\subset\M$ and $\N_2\subset\N_1$ are isomorphic via $\iotabar:\M\to\N$, where $\gamma = \iota\iotabar$, namely $\iotabar(\M) = \N_1$ and $\iotabar\iota(\N) = \theta(\N) = \N_2$.

\begin{prop}[\cite{DeGi2018}]\label{prop:Qsysdiscrete}
In the assumptions of the previous proposition, the intertwining property of the Pimsner--Popa basis which characterizes the discreteness of $\N\subset\M$ (Proposition \ref{prop:discrete}), or equivalently of $\N_2\subset\N_1$, namely $M_i \in \Hom(\iota, \iota \theta)$, is equivalent to 
\begin{align}\label{eq:genQsysinterts}
m_i \in \Hom(\theta,\theta^2)
\end{align}
for every $i$.
\end{prop}

We refer to \cite[\Sec 3, 5]{DeGi2018} for further explanations. A generalized Q-system $(\theta, w, \{m_i\})$ with the additional property \eqref{eq:genQsysinterts} is called a \textbf{generalized Q-system of intertwiners} \cite[\Def 3.7]{DeGi2018}.

Let $\N\subset \M$ be a discrete subfactor with $\N$, $\M$ infinite factors. It is known from the work of Izumi, Longo and Popa \cite[\Sec 2, 3]{IzLoPo1998} that the dual canonical endomorphism $\theta\in\End(\N)$ can be written as a \emph{direct sum} of irreducible subendomorphisms in $\End(\N)$ with \emph{finite dimension}, and that this condition is another characterization of discreteness. Clearly, the direct sum is finite if and only if the index of the subfactor is finite. If in addition the subfactor is \textbf{irreducible}, namely if $\N'\cap\M = \CC1$ or equivalently if $\Hom(\iota,\iota) = \CC1$, then each irreducible subsector $[\rho]$, $\rho\prec\theta$, $\rho\in\End(\N)$, appears in $[\theta]$ with finite multiplicity $n_\rho$. Namely
\begin{align}\label{eq:thetadiscrete}
\theta = \bigoplus_{[\rho], r} \rho
\end{align}
where the sum runs over inequivalent irreducible subendomorphisms $\rho\prec\theta$, one for each subsector $[\rho]$ of $[\theta]$, and over the finitely many copies of $[\rho]$ in $[\theta]$, labelled by $r=1,\ldots,n_{\rho}$. Moreover, it is shown in \cite{IzLoPo1998} that the multiplicity is bounded above by the square of the dimension
\begin{align}\label{eq:multboundndsquared}
n_\rho \leq d(\rho)^2.
\end{align}
Note that for each sector $[\rho]$ we choose the same endomorphism $\rho$ appearing $n_{\rho}$ times in the direct sum decomposition of $\theta$.
In other words, $\Hom(\theta,\theta) = \theta(\N)' \cap \N$ is a direct sum of finite matrix algebras, indexed by $[\rho]$, each of dimension $n_{\rho}^2$.

In this notation, if $\N\subset\M$ is an irreducible discrete subfactor together with an expectation $E:\M\to\N\subset\M$, which is \emph{unique} by irreducibility, one can choose a Pimsner--Popa basis with respect to $E$ made of \textbf{charged fields}\footnote{The terminology comes from the work of Doplicher and Roberts \cite{DoRo1972} on superselection sectors in quantum field theory.} $\{\psi_{\rho,r}\} \subset\M$, fulfilling
\begin{align}\label{eq:chargedfields}
\psi_{\rho,r}\in\Hom(\iota, \iota \rho)
\end{align}
and normalized such that 
\begin{align}\label{eq:normalizfields}
E(\psi_{\rho,r}\psi_{\rho,r}^*) = 1
\end{align}
for every $\rho, r$. Recall that $E(\psi_{\rho,r}\psi_{\sigma,s}^*) = 0$ if $[\rho]\neq [\sigma]$ or $r\neq s$ by the very definition of Pimsner--Popa basis. We shall always choose 
\begin{align}\label{eq:normalizvacuumfield}
\psi_{\id,1} = 1 
\end{align}
for the field associated with the vacuum sector $[\id]$, which is contained in $[\theta]$ with multiplicity one by irreducibility. 

\begin{notation}
Let $H_\rho := \Hom(\iota, \iota \rho)$ be the space of charged fields $\psi_\rho$ associated with $\rho\prec\theta$. 
\end{notation}

For every $\rho$, the vector space $H_\rho$ has dimension $n_{\rho}$. In this situation, a Pimsner--Popa basis $\{M_{\rho,r}\}$ with the intertwining property \eqref{eq:PiPointert} is obtained by setting
\begin{align}\label{eq:defwM}
w_{\rho,r} := \iotabar(\psi_{\rho,r}^*) w, \quad M_{\rho,r} := \iota(w_{\rho,r})\psi_{\rho,r}
\end{align}
for every $\rho,r$, where the $w_{\rho,r}$ are isometries in $\Hom(\rho,\theta)$ expressing the direct sum decomposition of $\theta$ into irreducibles, 
and $w$ is the isometry in $\Hom(\id,\theta)$, unique by irreducibility, which relates $E$ and $\gamma$ via the Connes--Stinespring representation 
\begin{align}\label{eq:stineE}
E = \iota(w)^* \gamma(\cdot) \iota(w) 
\end{align}
\cite[\Sec 5]{Lo1989}. Note that
$w_{\sigma,s}^*w_{\rho,r} = \delta_{\sigma,\rho} \delta_{s,r} 1$ and $p_{\rho,r} := \iota^{-1}E(M_{\rho,r}M_{\rho,r}^*) = w_{\rho,r} w_{\rho,r}^*$ is a projection in $\Hom(\theta,\theta)$ for every $\rho,r$. Actually $p_{\rho,r}\in\Proj_0(\theta)$ defined in equation \eqref{eq:Proj_def}. Moreover, $w = w_{\id,1}$.

\begin{notation}\label{not:genQsysint}
From now on, we deal with irreducible discrete subfactors of type $\III$. We shall use without further mention Pimsner--Popa bases (Definition \ref{def:PiPobasis}) made of charged fields $\{\psi_{\rho,r}\}$ and the associated generalized Q-systems (Definition \ref{def:genQsys}) of intertwiners
\begin{align}\label{eq:defm}
(\theta, w, \{m_{\rho,r} := \iotabar(M_{\rho,r})\})
\end{align}
with $\{M_{\rho,r}\}$ defined as in \eqref{eq:defwM}. The labels run over inequivalent irreducible subendomorphisms $\rho\prec\theta$ and multiplicity counting indices $r=1,\ldots, n_{\rho}$. 
Moreover, we shall assume $m_{\id,1} = \theta(w)$.
\end{notation}

For later use we state the following properties partially contained in \cite[\Sec 5]{DeGi2018}.

\begin{lem}\label{lem:genQsysprops}
Let $(\theta,w,\{m_{\rho,r}\})$ be a generalized Q-system of intertwiners constructed from charged fields as in Notation \ref{not:genQsysint}. Let $p_{\rho,r} = w_{\rho,r} w_{\rho,r}^*$ as before. Then
\begin{enumerate}
\item $(1_\theta\otimes p_{\sigma,s})m_{\rho,r} = \delta_{\rho,\sigma}\delta_{r,s} m_{\sigma,s}$ for $\rho,\sigma\prec\theta$ and $r,s$ as above. In particular, $(1_\theta\otimes p_{\id,1})m_{\rho,r} = \delta_{\rho,\id}\delta_{r,1} m_{\id,1}$, or equivalently $(1_\theta\otimes w^*)m_{\rho,r} = \delta_{\rho,\id}\delta_{r,1} 1_\theta$;
\item $(w^* \otimes 1_\theta)m_{\rho,r} = p_{\rho,r}$.
\end{enumerate}
Thus 
\begin{align}
\sum_{\rho,r} (1_\theta \otimes q)m_{\rho,r} =: (1_\theta \otimes q)m, \quad \sum_{\rho,r} (p \otimes q)m_{\rho,r} =: (p \otimes q)m 
\end{align}
are finite sums and well-defined in $\N$ for every $p,q\in\Proj_0(\theta)$, despite the fact that $m := \sum_{\rho,r} m_{\rho,r}$ itself is not well-defined as an element of $\N$ beyond the finite index case.
Summing over $\rho,r$ in the last equation of $(1)$ and in the sense of strong convergence in $(2)$, we have the identities 
\begin{align}
(1_\theta\otimes w^*)m = 1_\theta, \quad (w^* \otimes 1_\theta) m = 1_\theta.
\end{align}
\end{lem}

\begin{proof}
The first part of (1) follows from $(1_\theta \otimes w_{\sigma,s}^*) m_{\rho,r} = \iotabar\iota (w_{\sigma,s}^* w_{\rho,r}) \iotabar(\psi_{\rho,r}) = \delta_{\rho,\sigma}\delta_{r,s} \iotabar(\psi_{\rho,r})$, because the $w_{\rho,r}$ are isometries with mutually orthogonal range projections. The rest of (1) and (2) are contained in \cite[\Sec 5]{DeGi2018}.

To show the second statement it suffices to observe that $q\in\Proj_0(\theta)$ has finite support in $\Hom(\theta,\theta)$, namely there exist finitely many projections $p_{\rho,r}$ such that $p = \sum p_{\rho,r}$ fulfills $q \leq p$, \ie $q = pqp$.
\end{proof}

The Pimsner--Popa expansion (Proposition \ref{prop:PiPoexp}) for $\N_2=\theta(\N)\subset \N_1$ with respect to $E' = \theta(w^* \cdot w)$ (Proposition \ref{prop:fiis}) applied to the basis elements $m_{\sigma,s}$ gives:

\begin{lem}\label{lem:PiPoexpbasis}
$m_{\sigma,s} = \sum_{\rho,r} m_{\rho,r}^* \theta(m_{\sigma,s} w)$, where the sum over $\rho,r$ is finite.
\end{lem}

\begin{proof}
$m_{\sigma,s} = \sum_{\rho,r} m_{\rho,r}^\ast E'(m_{\rho,r} m_{\sigma,s}) = \sum_{\rho,r} m_{\rho,r}^\ast \theta(w^\ast m_{\rho,r} m_{\sigma,s} w ) = \sum_{\rho,r} m_{\rho,r}^* \theta(m_{\sigma,s} w)$ by the properties of generalized Q-systems of intertwiners stated in Lemma \ref{lem:genQsysprops}. The sum over ${\rho,r}$ is finite because $E(M_{\rho,r} M_{\sigma,s}) = E(w_{\rho,r} \psi_{\rho,r} w_{\sigma,s} \psi_{\sigma,s}) = w_{\rho,r} \rho(w_{\sigma,s})E(\psi_{\rho,r}\psi_{\sigma,s})$, $E(\psi_{\rho,r}\psi_{\sigma,s})$ belongs to $\Hom(\id,\rho \sigma) \cong \Hom(\sigmabar,\rho)$ and, for every ${\sigma}$, the sector $[\sigmabar]$ appears with finite multiplicity in $[\theta]$ by the irreducibility of the subfactor.
\end{proof}

\begin{lem}\label{lem:truncated}
For every $q\in\Proj_0(\theta)$ we have:
\begin{enumerate}
\item $(1_\theta\otimes q)mm^\ast(1_\theta\otimes q) = (1_\theta\otimes q)(m^\ast \otimes 1_\theta)(1_\theta\otimes m)(1_\theta\otimes q) = (1_\theta\otimes q)(1_\theta\otimes m^\ast)(m \otimes 1_\theta)(1_\theta\otimes q)$ \emph{(}truncated Frobenius property\emph{)};
\item $q = q ((w^*m^*) \otimes 1_\theta) (1_\theta \otimes (mw)) = ((w^*m^*) \otimes 1_\theta) (1_\theta \otimes (mw)) q$ 
\emph{(}truncated conjugate equations for $\theta$ solved by $mw$\emph{)}.
\end{enumerate}
\end{lem}

\begin{proof}
As for Lemma \ref{lem:PiPoexpbasis}, $m_{\sigma,s}m_{\rho,r}^* = \sum_{\tau,t} m_{\tau,t}^\ast E'(m_{\tau,t}m_{\sigma,s}m_{\rho,r}^*) = \sum_{\tau,t} m_{\tau,t}^\ast (1_\theta\otimes m_{\sigma,s})(1_\theta\otimes p_{\rho,r})$. By multiplying on either side by the same projection $(1_\theta \otimes q)$ and summing over $\rho,\sigma,r,s$ we obtain the first equality in (1), and the second one as well because the left hand side is manifestly self-adjoint.

To show (2), by multiplying by $w^*$ on the left of the equality stated in Lemma \ref{lem:PiPoexpbasis}, and again by Lemma \ref{lem:genQsysprops}, we obtain $p_{\sigma,s} = \sum_{\rho,r} w^*m_{\rho,r}^* (1_\theta \otimes (m_{\sigma,s} w))$.
For every $q\in\Proj_0(\theta)$ we can take sufficiently many irreducibles $\sigma\prec\theta$ such that $q=qp=pq$, where $p:= \sum_{\sigma,s} p_{\sigma,s}$, and by plugging in the expression for $p_{\sigma,s}$ we have the claim.
\end{proof}

%%%
\subsection{Braided and local subfactors}\label{sec:localsubf}
%%%

In the type of subfactors we deal with in this paper, we shall be equipped with additional structure motivated by conformal field theory. 
More specifically, we have in mind subfactors coming from inclusions of two \emph{local} conformal nets \cite[\Thm 4.9]{LoRe1995}, \cite[\Thm 6.8]{DeGi2018}.

Let $\theta\in\End(\N)$ be a dual canonical endomorphism of $\N\subset\M$, and assume that $\theta$ belongs to a full \emph{braided} tensor \Cstar-subcategory $\C\subset\End(\N)$, \eg $\C=\langle\theta\rangle$ the tensor \Cstar-category generated by $\theta$ equipped with a braiding $\varepsilon_{\theta^n,\theta^m}\in\Hom(\theta^{n+m},\theta^{n+m})$ for every $n,m\in\NN$.

Recall that a braided tensor \Cstar-category $\cC$ is a tensor \Cstar-category equipped with a family of unitary natural isomorphisms interchanging the order of tensor products $\{\varepsilon_{\rho,\sigma}\in\Hom(\rho\otimes\sigma,\sigma\otimes\rho)\}_{\rho,\sigma\in\cC}$, called braiding, which is compatible the tensor structure, see \eg \cite[\Ch 8]{EGNO15}. Naturality means $\varepsilon_{\rho',\sigma'} (u\otimes v) = (v\otimes u) \varepsilon_{\rho,\sigma}$ for every pair of morphisms $u\in\Hom(\rho,\rho')$ and $v\in\Hom(\sigma,\sigma')$. We denote the braiding and the opposite braiding respectively by $\varepsilon^+_{\rho,\sigma} :=\varepsilon_{\rho,\sigma}$ and $\varepsilon^-_{\rho,\sigma} := \varepsilon_{\sigma,\rho}^\ast$.

\begin{defi}\label{def:braidedsubf}
We call a subfactor \textbf{braided} if the dual canonical endomorphism $\theta$ belongs to a full braided tensor \Cstar-subcategory of $\End(\N)$.
\end{defi}

In Proposition \ref{prop:Qsysdiscrete} we saw that a semidiscrete subfactor admits a generalized Q-system of intertwiners if and only if it is discrete.

\begin{defi}
A generalized Q-system of intertwiners $(\theta,w,\{m_{\rho,r}\})$ in a braided tensor \Cstar-category $\cC\subset \End(\N)$ is called \textbf{commutative} if 
\begin{align}\label{eq:commutgenQsysint}
\theta(\varepsilon^{\pm}_{\theta,\theta}) m_{\sigma,s} m_{\rho,r} = m_{\rho,r} m_{\sigma,s} 
\end{align}
for every $\rho,\sigma,r,s$ as in Notation \ref{not:genQsysint}, where one can equivalently choose $\varepsilon^+_{\theta,\theta}$ or $\varepsilon^-_{\theta,\theta}$.
\end{defi}

\begin{defi}\label{def:localsubf}
Let $\N\subset\M$ be a braided discrete subfactor. It is called \textbf{local} if it admits a commutative generalized Q-system of intertwiners.
\end{defi}

Note that being braided is additional structure for a subfactor, while being local is a constraint on the structure.

\begin{lem}
The commutativity condition \eqref{eq:commutgenQsysint} is equivalent to
\begin{align}\label{eq:commutfields}
\iota(\varepsilon^{\pm}_{\sigma,\rho}) \psi_{\sigma} \psi_{\rho} = \psi_{\rho} \psi_{\sigma} 
\end{align}
for every pair of charged fields $\psi_\rho\in H_\rho$, $\psi_\sigma\in H_\sigma$, where $\rho,\sigma\prec\theta$ are irreducible and $\iota:\N\to\M$ is the inclusion morphism. 
In particular, if a braided discrete subfactor admits a commutative generalized Q-system of intertwiners then every generalized Q-systems of intertwiners is commutative.
\end{lem}

\begin{proof}
By the very definition of the $m_{\rho,r}$, \eqref{eq:defwM}, \eqref{eq:defm}, and by naturality of the braiding, one can show that \eqref{eq:commutgenQsysint} is equivalent to $\iota(\varepsilon^{\pm}_{\sigma,\rho}) \psi_{\sigma,s} \psi_{\rho,r} = \psi_{\rho,r} \psi_{\sigma,s}$. The statement for arbitrary charged fields follows by taking linear combinations of the $\psi_{\rho,r}$.
The second statement also follows by naturality of the braiding and by observing that $\tilde\psi \in H_{\tilde\rho}$ and $\psi \in H_{\rho}$ with $[\rho] = [\tilde\rho]$ are related by $\tilde\psi = u\psi$, where $u$ is a unitary conjugating $\rho$ to $\tilde\rho$.  
\end{proof}

One can consider two different Hilbert space inner products on the space of charged fields $H_\rho$, for $\rho\prec\theta$ \emph{irreducible} with finite dimension. Namely, for $\psi, \psi' \in H_\rho$ one can consider $\iota^{-1}E(\psi'\psi^*) \in \Hom(\rho,\rho) = \CC1$ and $\frac{1}{d(\rho)} \psi^* \psi' \in \Hom(\iota,\iota) \in \CC1$. Note that our choice of normalization \eqref{eq:normalizfields} means orthonormality of the charged fields in the Pimsner--Popa basis $\{\psi_{\rho,r}\}$ with respect to the first inner product. In \cite[\Sec 3]{IzLoPo1998} it is shown that the difference of the two inner products is measured by a positive operator $a_\rho \in \B(H_\rho)$ such that
\begin{align}\label{eq:twoinnerprods}
E(\psi' \psi^*) = \frac{1}{d(\rho)} \psi^* (a_\rho \psi').
\end{align}
In \cite{IzLoPo1998}, it is shown that every $a_\rho$ is invertible and trace class, thus every $H_\rho$ is finite-dimensional, \ie the multiplicity of each $[\rho]$ in $[\theta]$ is finite for irreducible discrete subfactors.
The collection of all $a_\rho$, for $\rho\prec\theta$ irreducible with finite dimension, is an invariant for the subfactor which controls the modular action on charged fields with respect to $E$-invariant states, see the proof of \cite[\Thm 3.3 (iii)]{IzLoPo1998}, \cf \cite[\Sec 2.3]{JoPe2019}. 

\begin{rmk}
A special role is played by the case where $a_\rho = 1_{H_\rho}$ for every irreducible $\rho\prec\theta$. For example, all irreducible finite index subfactors have this property \cite[\Sec 3]{IzLoPo1998}, but not all irreducible discrete subfactors fulfill it \cite[\Rmk 3.4, Appendix]{IzLoPo1998}. In the special case of subfactors coming from discrete/compact quantum group actions, the condition $a_\rho = 1_{H_\rho}$ corresponds to the Kac type case \cite[\Sec 4]{IzLoPo1998}, \cite[\Sec 3]{To2009}, \cite[\Sec 6.3]{JoPe2019}. 
\end{rmk}

We show below that locality (Definition \ref{def:localsubf}) implies the condition too.\footnote{Xu proves in \cite[\Cor 3.9]{Xu2005} that $a_\rho = 1_{H_\rho}$ follows for (subfactors arising from) 
%cannot say local subf here because he doesn't assume discrete does he?
inclusions of local conformal nets, under the \emph{strongly additive pair} assumption, defined in the same section.}

\begin{prop}\label{prop:arho=1}
If $\N\subset \M$ is irreducible discrete and local then every $a_\rho = 1_{H_\rho}$.
\end{prop}

\begin{proof}
For every $\psi, \psi' \in H_\rho$ we have
\begin{align}
\psi^* \psi' = &\;E(\psi^* \psi') = E(\psi^* \iota(\rho(r_\rho^*) \rbar_{\rho}) \psi') = \iota(r_\rho^*) E(\psi^* \iota(\rbar_{\rho}) \psi') = \\
= &\; \iota(r_\rho^*) \iota(\varepsilon^{\pm}_{\rho,\rhobar}) E(\psi' \psi^*) \iota(\rbar_{\rho}) 
= \iota(r_\rho^* \varepsilon^{\pm}_{\rho,\rhobar} \rbar_{\rho}) E(\psi' \psi^*)
\end{align}
by using standard solutions of the conjugate equations \eqref{eq:conjeqns} and observing that $\psi^* \iota(\rbar_{\rho}) \in H_{\rhobar}$ whenever $\psi\in H_\rho$, hence applying the commutativity condition \eqref{eq:commutfields} dictated by locality. Thus 
\begin{align}
a_\rho^{-1} = \frac{1_{H_\rho}}{d(\rho)} \iota(r_\rho^* \varepsilon^{\pm}_{\rho,\rhobar} \rbar_{\rho}) = \frac{1_{H_\rho}}{\kappa(\rho)^{\pm 1} d(\rho)}  \iota(r_\rho^* r_\rho) = \frac{1_{H_\rho}}{\kappa(\rho)^{\pm 1}}
\end{align}
using \cite[\Lem 4.3]{LoRo1997}, where $\kappa(\rho)$ is the phase of $\rho$, and $d(\rho) 1_{\id} = r_\rho^* r_{\rho}$ by definition of dimension of $\rho$. Recall that $\kappa(\rho) = \kappa(\rhobar)$. From the positivity of $a_\rho$ we conclude that $\kappa(\rho) = 1$, thus $a_\rho = 1_{H_\rho}$. 
\end{proof}

We have also shown the following generalization of \cite[\Prop 4]{Re1994coset} to the case of irreducible local discrete inclusions.

\begin{cor}\label{cor:spin=1}
Let $\theta\in\End(\N)$ be the dual canonical endomorphism of an irreducible local discrete subfactor $\N\subset\M$. Then the phase of every irreducible $\rho\prec\theta$ is $\kappa(\rho) = 1$.
\end{cor}

As remarked in \cite[\Sec 3]{IzLoPo1998}, from the condition $a_\rho = 1_{H_\rho}$ we can also conclude:

\begin{cor}\label{cor:multboundnd}
Let $\theta\in\End(\N)$ and $\rho\prec\theta$ be as above, then $n_{\rho} \leq d(\rho)$, where $n_{\rho}$ is the multiplicity of $[\rho]$ in $[\theta]$, instead of $n_{\rho} \leq d(\rho)^2$ as in \eqref{eq:multboundndsquared}.
\end{cor}

%%%
\subsection{Dual fields}\label{sec:dualfields}
%%%

For every irreducible thus finite-dimensional subendomorphism $\rho$ of $\theta$, we have fields $\psi \in \M$, $\psi \in H_\rho = \Hom(\iota, \iota\rho)$, and we want to construct \textbf{dual fields}, for later use, namely
\begin{align}
\bar\psi \in \N,\quad \bar\psi \in \Hom(\bar\iota, \rho\bar\iota).
\end{align}
If $\N\subset\M$ has finite index, they are easily given by $\bar\psi := w^* \bar\iota(\psi) \bar\iota(v)$, where $w\in\Hom(\id,\iotabar\iota)$ and $v\in\Hom(\id,\iota\iotabar)$ solve the conjugate equations for $\iota$ and $\iotabar$, see equation \eqref{eq:iotaandiotabar}. In the infinite index case, where we have $w$ but not $v$, we consider the formal sum $m=\sum_{\rho,r} m_{\rho,r}$ as in Lemma \ref{lem:genQsysprops} and recall that $(1_\theta \otimes q) m = \theta(q) m$ is well-defined and it belongs to $\N$ for every $q\in\Proj_0(\theta)$.

\begin{lem}
Let $\N\subset\M$ be as in Lemma \ref{lem:genQsysprops} and assume in addition that it is local.
Then $qm = (q \otimes 1_\theta) m$
is well-defined, it belongs to $\N$ for every $q\in\Proj_0(\theta)$ and it fulfills
\begin{align}
\varepsilon^{\pm}_{\theta,\theta} (1_\theta \otimes q) m = (q \otimes 1_\theta) m.
\end{align}
Moreover, $qm \in \Hom(\id,\theta_1)$, where $\theta_1 := \theta_{\restriction \N_1}$ and $\N_1 = \iotabar(\M)$.
\end{lem}

\begin{proof}
By locality have $\theta(\varepsilon^{\pm}_{\theta,\theta}) m_{\sigma,s} m_{\rho,r} = m_{\rho,r} m_{\sigma,s}$ and by multiplying on both sides by $w^*$ on the left, we get $\varepsilon^\pm_{\theta,\theta} p_{\sigma,s} m_{\rho,r} = p_{\rho,r} m_{\sigma,s}$. Now multiplying again on both sides by $q$ on the left, using naturality of the braiding and summing over $\rho,r$ (finite sum on the left, strongly convergent sum on the right) we get $\varepsilon^{\pm}_{\theta,\theta} (p_{\sigma,s} \otimes q) m = q m_{\sigma,s}$. Thus the sum over $\sigma,s$ of the $q m_{\sigma,s}$ converges strongly to an element in $\N$, that we denote by $qm$, and we have $\varepsilon^{\pm}_{\theta,\theta} (1_\theta \otimes q) m = q m$.

Concerning the intertwining property of $qm$, we first show that $q m m_{\sigma,s}^* =\theta(m_{\sigma,s}^*) qm$. By the Pimsner--Popa expansion, as in  Lemma \ref{lem:PiPoexpbasis}, we have 
\begin{align}\label{eq:frobeniusismasterfield}
m_{\sigma,s} m_{\rho,r}^* = \sum_{{\tau,t}} m_{\tau,t}^* E'(m_{\tau,t} m_{\sigma,s} m_{\rho,r}^*) = \sum_{\tau,t} m_{\tau,t}^* \theta(m_{\sigma,s} m_{\rho,r}^*w) = \sum_{\tau,t} m_{\tau,t}^* \theta(m_{\sigma,s})\theta(p_{\rho,r})
\end{align}

where the sum over $\tau,t$, a priori convergent in the GNS topology, is actually finite because $E(w_{\tau,t} \psi_{\tau,t} w_{\sigma,s} \psi_{\sigma,s} \psi_{\rho,r}^* w_{\rho,r}^*) = w_{\tau,t} \rho_{\tau,t}(w_{\sigma,s}) E(\psi_{\tau,t} \psi_{\sigma,s} \psi_{\rho,r}^*) w_{\rho,r}^*$, $E(\psi_{\tau,t} \psi_{\sigma,s} \psi_{\rho,r}^*)$ for every fixed $\rho$, $\sigma$, $r$, $s$ belongs to $\Hom(\rho_{\rho,r},\rho_{\tau,t} \rho_{\sigma,s})$ and the multiplicity of $[\rho_{\tau,t}]$ in $[\theta]$ is finite by the irreducibility of the subfactor. Now observe that the sum over $\sigma,s$ of the $m_{\sigma,s}^* q =  m^* (p_{\sigma,s} \otimes q) {\varepsilon^{\pm\, *}_{\theta,\theta}}$ also converges strongly and the limit is $(qm)^*$, that we denote by $m^*q$. Thus we can multiply both members of \eqref{eq:frobeniusismasterfield} by $q$ on the right and sum over $\rho,r$ in the strong operator topology, hence we obtain $m_{\sigma,s} m^*q = m^* q \theta(m_{\sigma,s})$ which is the adjoint of the desired equality. 

Now, every element $x\in \N_1$ can be written as $x = \sum_{\sigma,s} m_{\sigma,s}^* E'(m_{\sigma,s} x)$, $E'(m_{\sigma,s} x) \in\theta(\N)$ and the sum over $\sigma, s$ converges in the GNS topology. By continuity of the left multiplication in the GNS topology we conclude $qm x = \theta(x) qm$ for every $x\in \N_1$.
\end{proof}

We define the dual field of $\psi\in H_\rho$ as 
\begin{align}
\bar\psi := w^* \bar\iota(\psi)m
\end{align}
where recall that $m=\sum_{\rho,r} m_{\rho,r}$. It fulfills $\bar\psi \in \Hom(\bar\iota, \rho\bar\iota)$ and $\bar\psi\in\N$ because $w^* \bar\iota(\psi)$ is a linear combination of $w_{\rho,r}^*$ for some $r$, and $w_{\rho,r}^* = w_{\rho,r}^*p_{\rho,r}$.
 
Summing up the notation in a symmetric fashion, we have that:
\begin{align}
(w_{\rho,r}^* \otimes 1_\theta) m = \bar\psi_{\rho,r},\quad (1_\theta \otimes w_{\rho,r}^*) m = m_{\rho,r}.
\end{align}

%%%
\subsection{Adjointable UCP maps}\label{sec:Omegaadjmaps}
%%%

Let $\M,\Omega$ be a pair consisting of a von Neumann algebra $\M\subset\B(\Hil)$ and 
a standard unit vector $\Omega\in\Hil$, \ie $\Omega$ is cyclic and separating for $\M$ and it has norm one. A linear map $\phi\colon \M\to\M$ is \textbf{unital} if $\phi(1)=1$ and \textbf{completely positive} if it is positive, \ie $m\geq 0$, $m\in\M$, implies $\phi(m) \geq 0$, and if $\phi \otimes \id_n\colon \M\otimes M_n(\CC)\to\M\otimes M_n(\CC)$ is positive for every $n\in\NN$. In the following we abbreviate unital completely positive by \emph{ucp}.
 
\begin{defi}
A ucp map $\phi\colon \M\to \M$ is called \textbf{$\Omega$-adjointable} if
\begin{enumerate}
  \item $\phi$ preserves the state given by $\Omega$, \ie
    $(\Omega,\phi(\slot)\Omega)=(\Omega,\slot\Omega)$.
  \item $\phi$ admits an \textbf{$\Omega$-adjoint}, \ie a ucp map $\phisharp\colon \M\to \M$
such that
    \begin{align}\label{eq:omegaadj}
    (\phisharp(m_1)\Omega,m_2\Omega)=(m_1\Omega,\phi(m_2)\Omega)
    \end{align}
    for every $m_1,m_2\in \M$.
\end{enumerate}
We denote by ${\UCP}^{\sharp}(\M,\Omega)$ the set of $\Omega$-adjointable ucp maps on $\M$.
\end{defi}

A ucp map fulfilling property (1) is called \textbf{$\Omega$-preserving} and we denote by $\UCP(\M,\Omega)$ the set of such maps. Note that an $\Omega$-preserving ucp map is automatically normal and faithful. 
Moreover, if an $\Omega$-adjoint exists, then it is unique and automatically $\Omega$-preserving and $\Omega$-adjointable. Indeed, $(1)$ for $\phi^\sharp$ follows from \eqref{eq:omegaadj} and the unitality of $\phi$. 
Moreover, $\phi^{\sharp}{}^\sharp = \phi$, $(\phi_1 \circ \phi_2)^{\sharp}=\phi_2^\sharp \circ \phi_1^\sharp$, and $(\alpha\phi_1 + \beta\phi_2)^\sharp = \alpha \phi_1^\sharp + \beta \phi_2^\sharp$ for $\alpha,\beta\in [0,1]$. 
More generally, for complex linear combinations of elements in ${\UCP}^{\sharp}(\M,\Omega)$, the adjunction map $\phi\mapsto\phi^\sharp$ is antilinear.

For $\phi\in\UCP(\M,\Omega)$, denote by $V_\phi$ the closure of the operator 
\begin{align}
m\Omega\mapsto \phi(m)\Omega, \quad m\in \M.
\end{align}
By the Kadison--Schwarz inequality, $V_\phi \in \B(\Hil)$ and $\|V_\phi\| \leq 1$. Moreover, $V_\phi\Omega = \Omega$, thus $\|V_\phi\| = 1$. Clearly we have $V_{\phi_1\circ\phi_2} = V_{\phi_1}V_{\phi_2}$. If $\phi$ is in addition $\Omega$-adjointable, then $V_{\phisharp} = V_\phi^*$.
For details we refer to \cite[\App A.1]{Bi2016}, \cite{NiStZs2003} and references therein. The following proposition is due to \cite[\Prop 6.1]{AcCe1982}, see also \cite[\Lem 2.5]{De2006}.

\begin{prop}\label{prop:adjointiffmodular}
Let $\phi$ be an $\Omega$-preserving ucp map. Then the following are equivalent:
\begin{enumerate}
  \item $\phi$ admits an $\Omega$-adjoint $\phisharp$;
  \item $\phi\circ \sigma_t^{\M,\Omega} = \sigma_t^{\M,\Omega}\circ\phi$, $t\in\RR$;
  \item $V_\phi J_{\M,\Omega} = J_{\M,\Omega} V_\phi$,
\end{enumerate}
where $\sigma_t^{\M,\Omega}$, $t\in\RR$, is the modular group and $J_{\M,\Omega}$ is the modular conjugation of $(\M,\Omega)$.
\end{prop}

Clearly, if $\M$ is finite and $(\Omega,\slot\Omega)$ is a tracial state, then there is no distinction between $\Omega$-preserving and $\Omega$-adjointable maps. Note however that there are $\Omega$-preserving maps that are not $\Omega$-adjointable \cite[\Ex 2.7]{De2006} and that the notion of $\Omega$-adjoint does not coincide with the one of \emph{transpose} of a ucp map given \eg in \cite[\Ch 8]{OhPe1993}. The two notions coincide on $\Omega$-adjointable maps, see the proof of Lemma \ref{lem:betasharp}.

%%%
\subsection{Amenability for braided discrete type $\III$ subfactors}\label{sec:amenability}
%%%

In this section we introduce a notion of \emph{strong relative amenability} (Definition \ref{def:strongamenab}) for irreducible discrete type $\III$ subfactors (Definition \ref{def:semi-discr}) $\N\subset\M$ inspired by the type $\II_1$ analogue due to Popa, Shlyakhtenko and Vaes \cite[\Def 8.1]{PoShVa2018}. 

Denote by $\theta\in\End(\N)$ the dual canonical endomorphism of $\N\subset\M$ (Section \ref{sec:genQsys}) and let $\Omega\in\Hil$ be a standard unit vector for $\M\subset \B(\Hil)$ (Section \ref{sec:Omegaadjmaps}) such that 
\begin{align}
(\Omega, E(\slot)\Omega) = (\Omega, \slot\Omega)
\end{align}
where $E:\M\to\N\subset\M$ is the unique normal faithful conditional expectation for the subfactor. 
By definition $E$ is $\Omega$-preserving.
Denote also by $\langle \theta \rangle_0$ the rigid C$^\ast$-tensor subcategory of $\End_0(\N)$ generated by the irreducible components of $\theta$.

We adapt a proof of Jones and Penneys \cite{JoPe2019} to show that if $\langle \theta \rangle_0$ is \emph{amenable} in the sense of Popa and Vaes \cite[\Def 5.1]{PoVa2014}, as it is the case \eg for \emph{braided} subfactors (Theorem \ref{thm:cathalfbraidedisamenable}), then $\N\subset \M$ is strongly relatively amenable (Proposition \ref{prop:braidedsubfisstrongamenab}).

We first recall the notion of \emph{ucp-multiplier} on a rigid tensor \Cstar-category \cite[\Def 3.4, \Prop 3.6]{PoVa2014}, \cite[\Def 7.9]{JoPe2019}.
Throughout this paper, rigid tensor \Cstar-categories will always be implicitly assumed to be strict and with simple tensor unit, finite direct sums and subobjects.

\begin{defi}
Let $\cC$ be a rigid tensor \Cstar-category. 
Denote by $\Irr(\C)$ the set of equivalence classes of irreducible objects in $\C$,
and observe that every $t\in \Hom(\rho\sigma,\rho\sigma)$, for $\rho, \sigma$ objects in $\C$, 
can be written as a finite sum 
$t=\sum_{\tau\in\Irr(\C)} t_\tau$, where $t_\tau = (1_\rho\otimes t_{\tau\sigma,\sigma} )(t_{\rho,\rho\tau}\otimes 1_\sigma)$, 
for $t_{\rho,\rho\tau} \in\Hom(\rho,\rho\tau)$ and $t_{\tau\sigma,\sigma} \in \Hom(\tau\sigma,\sigma)$, 
via the following isomorphism of vector spaces
\begin{align}\label{eq:Hisom}
  \bigoplus_{\tau\in\Irr(\C)}\Hom(\rho,\rho\tau) \otimes\Hom(\tau\sigma,\sigma) &\longrightarrow
  \Hom(\rho\sigma,\rho\sigma)\\
  \bigoplus_{\tau\in \Irr(\C)} t_{\rho,\rho\tau}\otimes t_{\tau\sigma,\sigma} &\longmapsto \sum_\tau (1_\rho\otimes t_{\tau\sigma,\sigma} )(t_{\rho,\rho\tau}\otimes 1_\sigma).
\end{align}

We call a complex-valued bounded function $\mu\in \ell^\infty(\Irr(\C))$ a \textbf{ucp-multiplier} on $\C$ if for every 
$\rho,\sigma$ objects in $\C$, the linear map $\mu_{\rho,\sigma}\colon \Hom(\rho\sigma,\rho\sigma)\to
\Hom(\rho\sigma,\rho\sigma)$ defined by
\begin{align}\label{eq:PVcatucp}
  \mu_{\rho,\sigma}(t):=\sum_{\tau\in\Irr(\C)}\mu(\tau)t_\tau
\end{align}
is a ucp map. 
\end{defi}

\begin{defi}\label{def:AmenableTensorCategory}
A rigid tensor \Cstar-category $\C$ is called \textbf{amenable} if there is a net (or sequence if $\Irr(\C)$ is countable) of ucp-multipliers $\{\mu_\alpha\} \subset \ell^\infty(\Irr(\C))$ with finite support, \ie $\mu_\alpha(\tau) = 0$ for all but finitely many $\tau\in \Irr(\C)$, converging pointwise to the identity ucp-multiplier, \ie $\mu_{\id}(\tau) = 1$ for every $\tau \in \Irr(\C)$. 
\end{defi}

\begin{thm}\label{thm:cathalfbraidedisamenable}
Let $\C$ be a rigid tensor \Cstar-category and assume there is an object $\theta$ which generates $\C$ and which has a unitary half-braiding, then $\C$ is amenable. 

In particular, if $\N\subset \M$ is an irreducible braided discrete subfactor and $\theta\in\End(\N)$ is the dual canonical endomorphism, then $\langle\theta\rangle_0$ is amenable.
\end{thm}

\begin{proof}
The proof follows by combining \cite[\Thm 3.5]{NeYa2016}, see also \cite[\Thm 5.31]{LoRo1997}, with \cite[\Prop 5.3]{PoVa2014}.
\end{proof}

Now we introduce our notion of strong relative amenability for irreducible discrete type $\III$ subfactors $\N\subset\M$. We prove then that braided subfactors of this type are automatically strongly relatively amenable.

\begin{defi}\label{def:Nbimoducp}
We say that a ucp map $\phi\colon\M\to\M$ is \textbf{$\N$-bimodular} if $\phi(\iota(n_1)m\iota(n_2)) = \iota(n_1)\phi(m)\iota(n_2)$ for every $m\in\M$, $n_1,n_2\in\N$, where $\iota:\N\to\M$ is the inclusion morphism.
\end{defi}

\begin{defi}\label{def:finitesuppucp}
We say that a ucp map $\phi\colon\M\to\M$ has \textbf{finite support} if $\phi(\psi) = 0$ for every charged field $\psi\in H_\rho$, for all but finitely many inequivalent irreducible subendomorphisms $\rho\prec\theta$. 
\end{defi}

\begin{rmk}
Clearly $\phi$ has finite support if and only if $\phi(\psi_{\rho,r})=0$ for all but finitely many $\rho\prec\theta$, $r=1,\ldots,n_{\rho}$ as in Notation \ref{not:genQsysint}, where $\{\psi_{\rho,r}\}$ is a Pimsner--Popa basis of charged fields.
\end{rmk}

Let $\Omega$ as in the beginning of this section.

\begin{defi}\label{def:strongamenab}
We say that $\N\subset \M$ is \textbf{strongly relatively amenable} if 
there is a net $\{\phi_\alpha\}$ of 
$\N$-bimodular $\Omega$-preserving ucp maps $\phi_\alpha\colon\M\to\M$ with finite support, such that 
$\|\phi_\alpha(m)\Omega-m\Omega\|$ converges to zero for every $m\in\M$.
\end{defi}

\begin{rmk}\label{rmk:strongamenabrepnindep}
The results of Section \ref{sec:UCP/states} shall imply that this notion of strong relative amenability for irreducible discrete type $\III$ subfactors (not necessarily braided, with separable predual) does \emph{not} depend on the choice of standard vector $\Omega$ inducing an $E$-invariant vector state, nor on the Hilbert space representation. Indeed, Lemma \ref{lem:NfixingisOmegapres} shall imply that $\N$-bimodular ucp maps are necessarily $\Omega$-preserving. Moreover, for $m\in\M$, $\|(\phi_\alpha(m)-m)\Omega\|$ converges to zero if and only if $\|(\phi_\alpha(m)-m)\xi\|$ converges to zero for every $\xi\in\Hil$ because vectors of the form $m'\Omega$, with $m'\in\M'$, are dense in $\Hil$.
\end{rmk}

\begin{rmk}
Definition \ref{def:strongamenab} is a type $\III$ analogue of amenability for irreducible discrete type $\II_1$ subfactors introduced in \cite[\Def 8.1]{PoShVa2018}. See also \cite[\Def 2.10]{De1995}. It implies the original notion of amenability introduced by Popa in \cite[\Def 3.2.1]{Po1986}, see \cite[\Prop 2.11]{De1995}.
\end{rmk}

\begin{lem}\label{lem:CPonPsiStarPsi}
Let $\mu$ be a ucp-multiplier on $\langle \theta \rangle_0$ with finite support.
\begin{enumerate}
  \item By Proposition \ref{prop:PiPoexp}, every $m\in\M$ can be written as 
  $m = \sum_{\rho,r} \psi^\ast_{\rho,r} E(\psi_{\rho,r} m)$. Then 
    \begin{align}\label{eq:JPucp} 
      \phi_\mu (m) := \sum_{\rho,r} \psi^\ast_{\rho,r}\, \mu(\rhobar)E(\psi_{\rho,r}m)
    \end{align}
    defines a linear map $\phi_\mu\colon \M\to\M$ which is norm continuous, normal and continuous in the ultra-strong/weak operator topologies.
  \item For every $t\in\Hom(\rho\sigma,\rho\sigma)$ and $\psi\in H_\rho$, where $\rho, \sigma$ are objects in $\langle\theta\rangle_0$, consider 
  the element $\psi^\ast \iota(t)\psi \in \M$. 
  Then 
    \begin{align}
      \phi_\mu(\psi^\ast \iota(t)\psi)&=\psi^\ast \iota(\mu_{\rho,\sigma}(t))\psi
    \end{align}
  where $\mu_{\rho,\sigma}$ is defined in equation \eqref{eq:PVcatucp}.
\end{enumerate}
\end{lem}

\begin{proof}
The map $\phi_\mu$ is well-defined because of the uniqueness of the coefficients in the Pimsner--Popa expansion (Proposition \ref{prop:PiPoexp}). Moreover, the defining sum in \eqref{eq:JPucp} is finite by assumption, thus $\phi_\mu(m)\in\M$ and the stated continuity properties follow from the respective ones of $E$, see \eg \cite[\Prop III.2.2.2]{BlaBook}, hence we have (1).

By the isomorphism \eqref{eq:Hisom}, see \cite[\Prop 3.6]{PoVa2014}, we can write $t=\sum_{\tau\in\Irr(\C)} t_\tau$, where $t_\tau = (1_\rho\otimes t_{\tau\sigma,\sigma} )(t_{\rho,\rho\tau}\otimes 1_\sigma)$ for $t_{\rho,\rho\tau} \in\Hom(\rho,\rho\tau)$, $t_{\tau\sigma,\sigma} \in \Hom(\tau\sigma,\sigma)$, and we observe that the expansion of each
\begin{align}
  \psi^\ast\iota(t_\tau)\psi
  &=\sum_{\rho',r} \psi_{\rho',r}^\ast E(\psi_{\rho',r} \psi^\ast\iota(t_\tau)\psi)
\end{align}
only involves the sector $[\rho'] = [\taubar]$. Indeed
\begin{align}
  E(\psi_{\rho',r} \psi^\ast\iota(t_\tau)\psi) = E(\psi_{\rho',r} \psi^\ast\iota(\rho(t_{\tau\sigma,\sigma}) t_{\rho,\rho\tau})\psi)
  = \iota(\rho'(t_{\tau\sigma,\sigma})) E(\psi_{\rho',r} \psi^\ast\iota(t_{\rho,\rho\tau})\psi)
\end{align}
and $E(\psi_{\rho',r} \psi^\ast\iota(t_{\rho,\rho\tau})\psi) \in \Hom(\id,\rho'\tau) \cong \Hom(\taubar,\rho')$
which vanishes if $[\rho'] \neq [\bar\tau]$. Thus we conclude $\phi_\mu (\psi^\ast\iota(t_\tau)\psi) = \mu(\tau) \psi^\ast\iota(t_\tau)\psi = \psi^\ast\iota(\mu_{\rho,\sigma}(t_\tau))\psi$ and (2) follows by linearity.
\end{proof}

The next proposition should be compared with \cite[\Prop 7.11, \Cor 7.13]{JoPe2019}.

\begin{prop}\label{prop:braidedsubfisstrongamenab}
Let $\N\subset \M$, $\Omega$ and $\langle\theta\rangle_0$ be as above. Then
\begin{enumerate}
  \item For every finite support ucp-multiplier $\mu$ of $\langle\theta\rangle_0$ 
  equation \eqref{eq:JPucp} defines an $\N$-bimodular $\Omega$-preserving ucp map $\phi_\mu\colon\M\to\M$ with finite support;
  \item If $\langle\theta\rangle_0$ is amenable, then $\N\subset \M$ is strongly relatively amenable.
\end{enumerate}
\end{prop}

\begin{proof}
To show (1), we have to check that $\phi_\mu$ is ucp, $\N$-bimodular, $\Omega$-preserving. 

For every $m\in\M$, $n\in\N$, one can immediately check from the definition that $\phi_\mu(m\iota(n))=\phi_\mu(m)\iota(n)$ and $\phi_\mu(\iota(n) m) = \iota(n) \phi_\mu(m)$, \ie $\phi_\mu$ is $\N$-bimodular.

By the very definition of Pimsner--Popa basis and by equation \eqref{eq:normalizvacuumfield}, we have that $\psi_{\id,1} = 1$ and $E(\psi_{\rho,r}) = 0$ for every $[\rho]\neq\id$, thus $E\circ \phi_\mu(m) = \mu(\id) E(m)$ for every $m\in\M$. Moreover, $\mu(\id) = 1$ because $\mu_{\rho,\sigma}(1_{\rho\sigma}) = 1_{\rho\sigma}$ and $1_{\rho\sigma} = (1_{\rho\sigma})_{\id} = 1_{\rho}\otimes 1_{\sigma}$ via the isomorphism \eqref{eq:Hisom}, thus $E\circ \phi_\mu = E$. Hence we can argue that $(\Omega,\phi_\mu(m)\Omega)=(\Omega,E(m)\Omega)=(\Omega,m\Omega)$ for every $m\in\M$, \ie $\phi_\mu$ is $\Omega$-preserving.

Unitality follows immediately from the previous discussion. For complete positivity, denote by $\M_0$ be the unital *-subalgebra of $\M$ generated by $\N$ and by the Pimsner--Popa basis made of charged fields $\{\psi_{\rho,r}\}$ (Notation \ref{not:genQsysint}). By Proposition \ref{prop:PiPoexp} we can expand products and adjoints of fields as finite $\N$-linear combinations of the basis elements.
 
Let $\{x_1,\ldots,x_k\}\subset \M_0$, then $x_i =  \sum_{\rho,r} \psi_{\rho,r}^\ast \iota(n_{i,\rho,r})$
where $n_{i,\rho,r}\in\N$ and the sum is finite for every $i=1,\ldots,k$. By taking a family of isometries $\{v_{\rho',r'}\}$, one for each field appearing in the expansions of $x_1,\ldots,x_k$ counted with multiplicity, such that $v_{\rho',r'}^\ast v_{\rho,r} = \delta_{\rho,\rho'}\delta_{r,r'}1$ and $\sum_{\rho',r'}v_{\rho',r'}v_{\rho',r'}^\ast = 1$, we can write $x_i = \psi_\sigma^\ast \iota(n_{i})$ for every $i=1,\ldots,k$ for a single (reducible) $\sigma\in\langle\theta\rangle_0$ where $\psi_\sigma = \sum_{\rho',r'} v_{\rho',r'}\psi_{\rho',r'}\in H_\sigma$ and $n_i\in\N$. 
By choosing a pair $r$, $\rbar$ of solutions of the conjugate equations \eqref{eq:conjeqns} for $\sigma$ and $\sigmabar$, we can rewrite $x_i = r^\ast \psi_{\sigmabar} \iota(n_{i})$ where $\psi_{\sigmabar} = \psi_\sigma^\ast \rbar\in H_{\sigmabar}$. 
Let $\{y_1,\ldots,y_k\}\subset \M$, then
\begin{align}
\sum_{i,j=1}^k y_i^\ast\phi_\mu(x_i^\ast x_j) y_j  & = \sum_{i,j=1}^k y_i^\ast\iota(n^\ast_i) \phi_\mu(\psi_{\sigmabar}^\ast r r^\ast \psi_{\sigmabar}) \iota(n_j)y_j \\
& = \sum_{i,j=1}^k  y_i^\ast\iota(n^\ast_i)\psi_{\sigmabar}^\ast \iota(\mu_{\sigmabar,\sigma}(r r^\ast))\psi_{\sigmabar}\iota(n_j)y_j \\
& = y^\ast \iota(\mu_{\bar\sigma,\sigma}(rr^\ast)) y \geq 0
\end{align}
by Lemma \ref{lem:CPonPsiStarPsi} and by the complete\footnote{Note that in \cite[\Lem 3.7]{PoVa2014} it is shown that a positive multiplier on a rigid tensor \Cstar-category is automatically completely positive} positivity of $\mu_{\sigmabar,\sigma}$, where we set $y = \sum_{i=1}^k \psi_{\sigmabar}\iota(n_i)y_i$. Now, let $\{x_1,\ldots,x_k\} \subset \M$ and observe that $\M_0$ is weakly dense in $\M$, see \eg \cite[\Lem 3.8]{IzLoPo1998}. By Kaplansky's density theorem \cite[\Thm II.4.8]{Ta1} we can approximate each $x_i$ with bounded sequences 
$\{x_{i,n}\}\subset\M_0$ in the strong* operator topology. 
Hence $\{x_{i,n}^*x_{j,n}\}$
converges to $x_i^\ast x_j$ strongly*, thus also weakly, and by the continuity properties of $\phi_\mu$ stated in Lemma \ref{lem:CPonPsiStarPsi}, we conclude that $\sum_{i,j=1}^k y_i^\ast\phi_\mu(x_i^\ast x_j) y_j \geq 0$, \ie $\phi_\mu$ is completely positive by \cite[\Cor IV.3.4]{Ta1}.

Clearly, $\phi_\mu$ has finite support because $\mu$ has finite support. This concludes the proof of (1).

To show $(2)$, observe first that $\Irr(\langle\theta\rangle_0)$ is countable by the discreteness of the subfactor. By the amenability assumption there is a sequence $\{\mu_\alpha\}$ of ucp-multipliers on $\langle\theta\rangle_0$ converging pointwise to the identity ucp-multiplier. 
For $m\in\M_0$, compute
\begin{align}
\| (\phi_{\mu_\alpha}(m)-m)\Omega \|^2 & = \| \sum_{\rho,r} \psi_{\rho,r}^\ast (\mu_{\alpha}(\rhobar) - 1) E(\psi_{\rho,r}m) \Omega \|^2\\
& = \sum_{\rho,\rho',r,r'} \left|\mu_{\alpha}(\rhobar)-1 \right|^2 (\Omega, E(m^*\psi_{\rho',r'}^*) \psi_{\rho',r'}\psi_{\rho,r}^* E(\psi_{\rho,r}m) \Omega)\\
& = \sum_{\rho,r} \left|\mu_{\alpha}(\rhobar)-1 \right|^2 \|E(\psi_{\rho,r}m) \Omega\|^2
\end{align}
where we used the $E$-invariance of the state $(\Omega,\slot\Omega)$ and the normalization of the Pimsner--Popa basis, see equation \eqref{eq:normalizfields}. Observing that the sums over $\rho,r$ are finite because $m\in\M_0$, we can take the limit over $\alpha$ and the limit is zero
by the pointwise convergence assumption.
For $m\in\M$, we can choose $m_0\in\M_0$ such that $\|(m-m_0)\Omega\|$ is arbitrarily small, because $\M_0$ is weakly/strongly dense in $\M$, as we observed above. Compute
\begin{align}
\|(\phi_{\mu_\alpha}(m)-m)\Omega\|^2 & = \|(\phi_{\mu_\alpha}(m)-\phi_{\mu_\alpha}(m_0)+\phi_{\mu_\alpha}(m_0)-m_0+m_0-m)\Omega\|^2\\
& \leq \|\phi_{\mu_\alpha}(m-m_0)\Omega\|^2 + \|(\phi_{\mu_\alpha}(m_0)-m_0)\Omega\|^2 + \|(m_0-m)\Omega\|^2 
\end{align}
where the first summand can be rewritten as $(\Omega,\phi_{\mu_\alpha}(m-m_0)^*\phi_{\mu_\alpha}(m-m_0)\Omega)$. By the Kadison--Schwarz inequality, see \eg \cite[\App A.1]{Bi2016}, \cite{NiStZs2003} and using the fact that $\phi_{\mu_\alpha}$ is $\Omega$-preserving, we can estimate that summand from above by $\|(m-m_0)\Omega\|^2$. Thus the statement follows from the previous step by an $\epsilon/3$ argument, concluding the proof.
\end{proof}

Combining Theorem \ref{thm:cathalfbraidedisamenable} with Proposition \ref{prop:braidedsubfisstrongamenab}, we have the following:

\begin{cor}\label{cor:braideddiscreteisstrongamenab}
Let $\N\subset \M$ be an irreducible braided discrete type $\III$ subfactor (Definition \ref{def:semi-discr}, \ref{def:braidedsubf}).
Then $\N\subset\M$ is strongly relatively amenable (Definition \ref{def:strongamenab}).
\end{cor}

%%%
\section{Compact hypergroups}\label{sec:compacthyp}
%%%

In the next section, we shall see how a \emph{compact hypergroup} is canonically associated to an irreducible local discrete subfactor.
In this section, we explain our definition of compact hypergroup, which turns out to be a special case of locally compact hypergroup in the sense of \cite{KaPoCh2010} and of compact quantum hypergroup in the sense of \cite{ChVa1999}, and then we compare it with other notions appearing in the literature on hypergroups (and hypercomplex systems) \cite{BlHe1995}, \cite{BeKa1998}. A prominent role in our description is played by the so-called Haar measure, arising in subfactor land from the unique normal faithful conditional expectation.

\begin{defi}\label{def:StarMonoid}
A \textbf{monoid with involution} is a set $M$ equipped with a binary operation $M\times M\to M, (x,y)\mapsto x\ast y$, an involution $M\to M, x\mapsto x^\sharp$, \ie $(x^\sharp)^\sharp = x$, and a distinguished element $e\in M$, the unit, such that for every $x,y,z\in M$ we have
$(x\ast y)\ast z=x\ast(y\ast z)$,
$\idK\ast x=x= x\ast \idK$ and 
$(x\ast y)^\sharp = y^\sharp \ast x^\sharp$.
We call $h = h^\sharp = h \ast h \in M$ a \textbf{Haar element} for $M$ if $h \ast x = h = x \ast h$ for every $x\in M$.
\end{defi}

Let $K$ be a compact Hausdorff space. We denote by $P(K)$ the set of probability Radon measures on $K$. The vector space obtained by complex linear combinations of elements in $P(K)$ will be denoted by $M(K)$, \ie the vector space of complex bounded Radon measures on $K$.
Recall that a measure $\mu\in P(K)$ is faithful if $f\mapsto \mu(f) := \int_K f\dd \mu$ is a faithful state on the \Cstar-algebra of continuous functions $C(K)$, \ie if $f\in C(K)$, $f\geq 0$ with $\mu(f) = 0$  implies $f = 0$.
Recall also that the continuous dual of $C(K)$ is isometrically isomorphic to $M(K)$ by the Riesz--Markov theorem.
For every $x\in K$, denote by $\delta_x \in P(K)$ the Dirac measure concentrated in $x$.

\begin{defi}\label{def:CompactHypergroup}
A \textbf{compact hypergroup} is a compact Hausdorff space $K$ equipped with a biaffine operation called \textbf{convolution} on probability measures $P(K)\times P(K)\to P(K), (\mu,\nu) \mapsto \mu\ast\nu$,  an \textbf{involution} $K \to K, x \mapsto x^\sharp$, and a distinguished \textbf{identity element} $\idK\in K$ with $\idK = \idK^\sharp$, such that:
\begin{itemize}
\item[$(i)$] The convex space $P(K)$ is a monoid with involution with respect to $\ast,\,^\sharp,\delta_e$, where the involution is defined on measures by the pushforward, \ie $\mu^\sharp(E) := \mu(E^\sharp)$ for every Borel set $E\subset K$. By definition, $\delta_x^\sharp = \delta_{x^\sharp}$ for every $x\in K$;
\item[$(ii)$] The involution $x \mapsto x^\sharp$ is continuous and the map 
\begin{align}
(x,y)\in K\times K \mapsto \delta_x\ast\delta_y \in P(K)
\end{align}
is jointly continuous with respect to the weak* topology on measures;
\item[$(iii)$] There exists a faithful probability measure $\mu_K\in P(K)$, called a \textbf{Haar measure} on $K$,
such that for every $f,g\in C(K)$ and $y\in K$ it holds
\begin{align}\label{eq:haarproperty1}
\int_Kf(y\ast x) g(x)\dd\mu_K(x) = \int_Kf(x)g(y^\sharp \ast x) \dd\mu_K(x),
\end{align}
\begin{align}\label{eq:haarproperty2}
\int_Kf(x\ast y) g(x)\dd\mu_K(x) = \int_Kf(x)g(x \ast y^\sharp) \dd\mu_K(x),
\end{align}
where
\begin{align}
f(x\ast y ) := (\delta_x\ast\delta_y)(f) = \int_Kf(z)\dd(\delta_x \ast \delta_y)(z).
\end{align}
\end{itemize}
\end{defi}

This notion of compact hypergroup coincides with the one of \emph{locally compact hypergroup} studied in \cite[\Def 2.1]{KaPoCh2010}, restricted to the compact case and dualizing the description from $M(K)$ to $C(K)$, if we assume the existence of a (unique normalized) left and right Haar measure \cite[\Eq (H$_4$)(4), (5)]{KaPoCh2010}.

In the case of compact metrizable $K$, a compact hypergroup in the sense of Definition \ref{def:CompactHypergroup} endows the commutative separable \Cstar-algebra $C(K)$ with a \emph{compact quantum hypergroup structure} as in \cite[\Def 4.1]{ChVa1999}.
Indeed, property (QH$_1$) of \cite[\Def 4.1]{ChVa1999} holds 
by \cite[\Thm 8.13]{KaPoCh2010}.

\begin{rmk}
The convolution on Dirac measures can be uniquely extended to a bilinear 
associative convolution on $M(K)$ via the formula
\begin{align}\label{eq:convondirac}
\mu \ast \nu = \int_K \int_K \delta_x \ast \delta_y \dd\mu(x) \dd\nu(y)
\end{align}
for $\mu,\nu\in M(K)$, see \cite[Lem 2.4B]{Je1975}. 
The vector space $M(K)$ becomes a unital involutive Banach algebra with respect to the antilinear adjoint map $\mu\mapsto\mu^*:=\overline{\mu^\sharp}$ and the norm given by the total variation, see \cite[\Sec 1.1.2]{BlHe1995}, \cite[\Sec 1.2.3]{BeKa1998}. 
\end{rmk}

\begin{prop}\label{prop:Haarabsorbs}
The Haar measure is indeed a Haar element for $P(K)$, namely $\mu_K \ast \nu = \mu_K = \nu \ast \mu_K$ for every $\nu\in P(K)$. In particular, it is unique and $\mu_K = \mu_K^\sharp = \mu_K \ast \mu_K$.
\end{prop}

\begin{proof}
It follows from \cite[\Lem 4.1]{KaPoCh2010} and \cite[\Prop 1.3.17]{BlHe1995}.
\end{proof}

\begin{example}[Compact groups]
If $K$ is a compact group, a hypergroup structure 
is given by $\delta_x \ast \delta_y = \delta_{xy}$ and $x^\sharp = x^{-1}$, where $xy\in K$ is the group operation and $x^{-1}\in K$ the group inverse, $e$ is the identity element and $\mu_K$ is the Haar measure.
\end{example}

\begin{example}[Finite hypergroups]\label{ex:finitehyp}
A finite hypergroup in the sense of Sunder--Wildberger \cite[Definition 1.1]{SuWi2003}, see also \cite{Bi2016,BiRe2019}, gives canonically a hypergroup in the sense of Definition \ref{def:CompactHypergroup} by identifying $K$ with $\{\delta_k\}_{k\in K}$.
Conversely, if $K$ is a compact hypergroup in the sense of Definition \ref{def:CompactHypergroup} and $K$ is a finite set then (iii) implies that  $\delta_e \leq \delta_x\ast\delta_y$ if and only if $x=y^\sharp$ and thus $K$ is canonically a Sunder--Wildberger hypergroup. In this case, we will only use the term \emph{finite hypergroup}.

For example, let $K=\{e,k\}$. 
Then for every $w\in [1,\infty]$ the multiplication law $\delta_k\ast\delta_k = w^{-1}\cdot \delta_e+ (1-w^{-1})\cdot \delta_k$ defines a finite hypergroup in the above sense and every finite hypergroup with two elements is necessarily of this type.
\end{example}

\begin{rmk}\label{rmk:DJS}
A compact hypergroup in the sense of Definition \ref{def:CompactHypergroup} is \emph{not} necessarily a \emph{DJS-hypergroup} after Dunkl, Jewett and Spector, see \cite[\Sec 1.1.2 (HG1-7)]{BlHe1995}, unless \eg $K$ is a finite hypergroup.
In Definition \ref{def:CompactHypergroup} we do not assume the continuity property (HG4) which guarantees \eg that the left and right convolution maps are open. The stronger inversion property (HG7), namely $e$ is in the support of the measure $\delta_x \ast \delta_{y}$ if and only if $y = x^\sharp$, is replaced by $(iii)$. Vice versa, every compact DJS-hypergroup fulfills the requirements of Definition \ref{def:CompactHypergroup} by \cite[\Thm 1.3.21, 1.3.28]{BlHe1995}.
\end{rmk}

We conclude the section with two general results for recognizing $x^\sharp$.

\begin{prop}\label{prop:uniqueinverseinK}
  Let $K$ be a compact hypergroup in the sense of Definition \ref{def:CompactHypergroup}. 
  The involution $x\mapsto x^\sharp$ is uniquely determined by the Haar measure via either of the two equations in $(iii)$.
\end{prop}

\begin{proof}
Assume that $x',x'' \in K$ fulfill \eg the first equation in $(iii)$, namely for every $f,g\in C(K)$
\begin{align}
\int_Kf(x\ast y) g(y)\dd\mu_K(y) &= \int_Kf(y)g(x' \ast y) \dd\mu_K(y)\label{eq:x'unique}\\ 
&= \int_Kf(y)g(x'' \ast y) \dd\mu_K(y)\label{eq:x''unique}.
\end{align}
By \cite[\Thm 1.2.2.]{BlHe1995}, which applies since the convolution is jointly continuous by $(ii)$ and $K$ is compact, we have that for every $\epsilon > 0$ there exists a compact neighbourhood $V$ of $e$ such that $|g(x\ast y) - g(x)| < \epsilon$ for every $y \in V$, $x \in K$. Note that $V$ depends on $\epsilon$ and $g$. Choose $f\geq 0$, $\int_K f \dd\mu_K = 1$ and supported in $V$, then \eqref{eq:x'unique} equals
\begin{align}
\int_Kf(y)(g(x' \ast y) - g(x')) \dd\mu_K(y) + g(x')
\end{align}
and $|\int_Kf(y)(g(x' \ast y) - g(x')) \dd\mu_K(y)| \leq \sup_{y\in V} |g(x' \ast y) - g(x')| < \epsilon$. With the same choice of $f$, an analogous statement holds for \eqref{eq:x''unique} because the bound on $g$ is uniform in $x$. Thus $g(x') = g(x'')$ for every $g$ and $x' = x''$, showing the uniqueness of the involution.
\end{proof}

\begin{defi}\label{def:dominatedposmeas}
Let $\mu_1, \mu_2$ be positive measures on $K$. We say that $\mu_1$ is \textbf{dominated} by $\mu_2$, written as $\mu_1 \leq d \mu_2$ or simply $\mu_1 \leq \mu_2$, if there exists a constant $d>0$ such that $d\mu_2-\mu_1$ is a positive measure.
\end{defi}

\begin{prop}\label{prop:DominatedImpliesConjugate}
Let $K$ be a compact hypergroup in the sense of Definition \ref{def:CompactHypergroup}.
Assume that $K$ is metrizable and let $x,x'\in K$. If $\delta_{e}\leq \delta_x\ast\delta_{x'}$ then $x'=x^\sharp$.
\end{prop}

\begin{proof}
Let $\{g_{n}\}\subset C(K)$ be a $\delta_{x'}$-approximating sequence by positive continuous functions supported in balls of radius $1/n$ centered at $x'$, \cf equation \eqref{eq:approxdiracdeltas}.
By assumption, there is a $d > 0$ such that $d\delta_x\ast\delta_{x'} -\delta_{e}$ is a positive measure.
Thus for every $f\in C(K)$, $0\leq f\leq 1$ with $f(e)=1$ and support contained in a compact neighbourhood $V$ of $e$
we have
\begin{align}
  d^{-1}=d^{-1}f(e) \leq
f(x\ast x') =& \lim_{n\to \infty}\int_K g_n(z) f(x\ast y)\dd\mu_K(y)\\
=& \lim_{n\to \infty}\int_V g_n(x^\sharp \ast y) f(y)\dd\mu_K(y).
\end{align}
  Assume that $x'\neq x^\sharp$. 
  Then for $n$ sufficiently big we may assume that $g_n(x^\sharp)=0$.
  Let $d^{-1}/2>\delta>0$ then by \cite[Theorem 1.2.2]{BlHe1995} we can choose $V$ 
  such that $|g_n(x^\sharp \ast y)|=|g_n(x^\sharp \ast y) - g_n(x^\sharp)| <\delta$ for every $y\in V$.
  Note that $\mu_K(V)\leq 1$. Then for $n$ big enough we have
\begin{align}
  \left|\int_V g_n(x^\sharp \ast y) f(y)\dd\mu_K(y)\right|&\leq
  \int_V|g_n(x^\sharp \ast y)|\dd\mu_K(y)<
  \delta\,\mu_K(V)\leq d^{-1}/2
\end{align}
which is a contradiction.
\end{proof}

\begin{rmk}
The analogous of Proposition \ref{prop:DominatedImpliesConjugate} does not hold, with the same proof, if we replace $\delta_{e}\leq \delta_x\ast\delta_{x'}$ with the weaker assumption that $e$ is in the support of the measure $\delta_x \ast \delta_{x'}$ \cite[\Sec 1.1.2 (HG7)]{BlHe1995}.
Furthermore, we could not show that the same weaker assumption is enough to obtain the claim of \cite[\Thm 1.3.21]{BlHe1995}, from which Proposition \ref{prop:DominatedImpliesConjugate} would follow from Proposition \ref{prop:uniqueinverseinK}, without the continuity assumption \cite[\Sec 1.1.2 (HG4)]{BlHe1995}.
\end{rmk}

%%%
\section{Compact hypergroups from local discrete subfactors}\label{sec:compacthypfromsubf}
%%%

In this section we show that from an irreducible local discrete subfactor $\N\subset\M$ one can construct a compact hypergroup $K(\N\subset\M)$ (Definition \ref{def:CompactHypergroupUCP}) in the sense of Definition \ref{def:CompactHypergroup}. The hypergroup acts on $\M$ (Definition \ref{def:action}) and it yields $\N$ as the fixed point subalgebra (Theorem \ref{thm:genorbi}).

$K(\N\subset\M)$ is defined as the set of extreme points of ${\UCP}^{\sharp}_\N(\M,\Omega)$ (Definition \ref{def:UCPbim}), equipped with the topology of $L^2(\M,\Omega)$-convergence (Definition \ref{def:L2conv}). The convolution is given by the composition of ucp maps and the involution by the $\Omega$-adjoint. In order to check this, in particular the compactness of extreme points and the measure theoretical interpretation of the convolution, we construct $K(\N\subset\M)$ in a dual way.
Namely, we first associate to the subfactor a commutative unital involutive algebra $\Trig(\N\subset \M)$ (Definition \ref{def:trig}) which will later turn out to be the algebra of trigonometric polynomials of the hypergroup $\Trig(K)$.
Taking the \Cstar-completion $\Cred(\N\subset\M)$ of $\Trig(\N\subset \M)$ in the correct Hilbert space representation, namely the one induced by the Haar measure of the hypergroup, we obtain an algebra of continuous functions $C(K)$ and we show that $P(K) \cong {\UCP}^{\sharp}_\N(\M,\Omega)$, in particular on the extreme points $K\cong\K(\N\subset\M)$ (Theorem \ref{thm:ucpstatesduality}). 

Note that the property of $\Trig(K)$ being an algebra of functions is not automatic in general \cite[\Sec 2]{Vr1979}, but it holds in our case for $K(\N\subset\M)$. In particular, positive functions in $C(K)$ can be approximated with positive elements in $\Trig(K)$.
This property is crucial for the proof of \eqref{eq:haarproperty1} and \eqref{eq:haarproperty2} in Theorem \ref{thm:Kishypergroup} and for the Choquet-type decomposition of the conditional expectation $E$ (Corollary \ref{cor:choquetdecomp}).
Moreover, $K(\N\subset\M)$ is compact metrizable because $\Cred(\N\subset\M)$ is separable.

%%%
\subsection{Trigonometric polynomials}\label{sec:trig}
%%%

Let $\N\subset \M$ be an irreducible local discrete type $\III$ subfactor (Section \ref{sec:localsubf}), with dual canonical endomorphism $\theta\in\End(\N)$, and let $(\theta,w,\{m_{\rho,r}\})$ be a generalized Q-system of intertwiners as in Notation \ref{not:genQsysint}. 
Consider the formal sum $m=\sum_{\rho,r} m_{\rho,r}$ as in Lemma \ref{lem:genQsysprops} and Section \ref{sec:dualfields}.

\begin{defi}\label{def:trig}
We define
\begin{align}
\Trig(\N\subset \M) := \{a \in \Hom(\theta,\theta) :
a=p a p \text{ for some } p\in \Proj_0(\theta) \}
\end{align}
with a binary operation
\begin{align}\label{eq:aperb}
a\ast b := m^\ast ( a\otimes b) m = \sum_{\rho,\sigma,r,s} m_{\rho,r}^\ast ( a\otimes b) m_{\sigma,s}
\end{align}
and an antilinear involution
\begin{align}\label{eq:snakes}
a^\bullet :=&\; (1_\theta \otimes (w^*m^*))(1_\theta \otimes a^* \otimes 1_\theta)((mw)\otimes 1_\theta) \\ 
=&\; ((w^*m^*)\otimes 1_\theta)(1_\theta \otimes a^* \otimes 1_\theta)(1_\theta \otimes (mw)).\label{eq:snakes2}
\end{align}
We refer to $\Trig(\N\subset\M)$ as the \textbf{algebra of trigonometric polynomials} of $\N\subset\M$.
\end{defi}

\begin{lem}\label{lem:finitesums}
For every $a,b\in\Trig(\N\subset\M)$, all sums in the definition of $a*b$ and $a^\bullet$ are finite.
\end{lem}

\begin{proof}
The sum in the definition of $a*b$ is finite by Lemma \ref{lem:genQsysprops}, and the same is true for the one in $(1_\theta \otimes a^*)mw$. To show that the sum in $(a^* \otimes 1_\theta)mw$ is finite as well, it is enough to compute $(p_{\rho,r} \otimes 1_\theta)m_{\sigma,s}w = w_{\rho,r}w_{\rho,r}^* \iotabar(\iota(w_{\sigma,s})\psi_{\sigma,s}) w =w_{\rho,r} \rho(w_{\sigma,s}) w^* \iotabar(\psi_{\rho,r}\psi_{\sigma,s})w$ and observe that it can be non-zero only if $[\sigma] = [\rhobar]$.
\end{proof}

One can show that these operations are well-defined, \ie they map to $\Trig(\N\subset\M)$, and that the equality claimed in \eqref{eq:snakes}-\eqref{eq:snakes2} holds. Instead of giving a direct proof of these statements, in the remainder of this section we give an equivalent, computationally easier, description of $\Trig(\N\subset\M)$ in terms of \emph{matrix units}. We also show the non-trivial fact that these operations endow it with the structure of an \emph{associative} \emph{involutive} algebra with \emph{unit} $\one := ww^\ast$. This algebra is easily seen to be \emph{commutative} due to the locality assumption on the subfactor.

Recall that $\Hom(\theta,\theta)$ is a direct sum of finite matrix algebras by the discreteness assumption, namely
\begin{align}\label{eq:thetathetafullmats}
\Hom(\theta,\theta) \cong \bigoplus_{[\rho]} M_{n_\rho}(\CC)
\end{align}
where the sum runs over inequivalent irreducible subendomorphisms $\rho\prec\theta$. A complete set of matrix units is given by $w_{\rho,r}w_{\rho,s}^* \in \Trig(\N\subset\M)$, where $w_{\rho,r}\in\Hom(\rho,\theta)$ are the isometries defined in equation \eqref{eq:defwM}.

Let $\rho\in\End_0(\N)$, not necessarily irreducible nor $\rho\prec\theta$. Denote as before the corresponding space of charged fields by $H_\rho = \Hom(\iota, \iota \rho)$. 
We have a map
\begin{align}
  T\colon H_\rho \times H_\rho
  &\to \Trig(\N\subset \M)\\
  (\psi_1,\psi_2)&\mapsto 
  (1_{\bar\iota}\otimes \psi_1^\ast)(ww^\ast \otimes 1_\rho)
  (1_{\bar\iota}\otimes \psi_2)
  \\&\qquad = \bar\iota(\psi_1^\ast)ww^\ast\bar\iota(\psi_2)
  \\
 \left(
 \tikzmatht[.6]{
		  \fill[\colM,rounded corners] (-.5,1) rectangle (1,-.5);
		  \fill[\colN,rounded corners] (0,1) rectangle (1,-.5);
		  \fill[\colN] (0,1) rectangle (.5,-.5);
      \draw (0,1)--(0,-.5);
      \draw[thick] (0,0) arc (-90:0:.5)--(.5,1);
    }
,
 \tikzmatht[.6]{
		  \fill[\colM,rounded corners] (-.5,1) rectangle (1,-.5);
		  \fill[\colN,rounded corners] (0,1) rectangle (1,-.5);
		  \fill[\colN] (0,1) rectangle (.5,-.5);
      \draw (0,1)--(0,-.5);
      \draw[thick] (0,0) arc (-90:0:.5)--(.5,1);
    }
    \right)
  &\mapsto  
    \tikzmatht[.6]{
      \fill[\colN,rounded corners] (-1,1) rectangle (1,-1);
      \fill[\colM] (-.5,1)--(-.5,.4) arc (180:360:.25)--(0,1)
      (-.5,-1)--(-.5,-.4) arc (180:0:.25)--(0,-1);
      \draw (-.5,1)--(-.5,.4) arc (180:360:.25)--(0,1)
      (-.5,-1)--(-.5,-.4) arc (180:0:.25)--(0,-1);
      \draw[thick] (0,.5) arc (90:-90:.5);
    }.
\end{align}

\begin{rmk}\label{rmk:catgenthetaonly}
Note that $\iotabar(\psi^*)w = 0$ if $\Hom(\rho,\theta) = \langle 0 \rangle$, thus also $\Hom(\iota,\iota\rho) = \langle 0 \rangle$ because $w^* \iotabar(\psi\psi^*)w = \iota^{-1}E(\psi\psi^*) = 0$ implies $\psi=0$. Vice versa also holds, see \cite[\Lem 6.15]{DeGi2018}.
Thus we may consider only $\rho\in\End_0(\N)$ in the rigid tensor \Cstar-category generated by $\theta$, namely finite direct sums of irreducible subendomorphisms of $\theta$.
\end{rmk}

Applied to the charged fields in the chosen Pimsner--Popa basis (Notation \ref{not:genQsysint}), with $\rho,\sigma\prec \theta$ irreducible, we get $T(\psi_{\rho,r},\psi_{\rho,s}) = w_{\rho,r}w_{\rho,s}^*$. The multiplication of $\Trig(\N\subset\M)$ (Definition \ref{def:trig}), which is not the matrix/operator multiplication in $\Hom(\theta,\theta)$, acts as follows on matrix units:

\begin{prop}\label{prop:multTirred}
Let $\rho,\sigma\prec\theta$ be irreducible and $\psi_{\rho,r},\psi_{\rho,s}\in H_\rho$, $\psi_{\sigma,u},\psi_{\sigma,v}\in H_\sigma$, then 
\begin{align}
T(\psi_{\rho,r},\psi_{\rho,s}) \ast T(\psi_{\sigma,u},\psi_{\sigma,v}) = T(\psi_{\rho,r}\psi_{\sigma,u},\psi_{\rho,s}\psi_{\sigma,v}).
\end{align}
\end{prop}

\begin{proof}
It is enough to observe that $\theta(w_{\sigma,v}^*)m = \iotabar(\psi_{\sigma,v})$, and compute $T(\psi_{\rho,r},\psi_{\rho,s}) \ast T(\psi_{\sigma,u},\psi_{\sigma,v}) = \iotabar(\psi_{\sigma,u}^*) w_{\rho,r}w_{\rho,s}^* \iotabar(\psi_{\sigma,v}) = \iotabar(\psi_{\sigma,u}^*\psi_{\rho,r}^*) ww^* \iotabar(\psi_{\rho,s}\psi_{\sigma,v})$.
\end{proof}

For every Pimsner--Popa basis of charged fields $\{\psi_{\rho,r}\}$ one can produce another Pimsner--Popa basis as follows. For every $\rho\in\End_0(\N)$, not necessarily irreducible nor $\rho\prec\theta$, and $\psi\in H_\rho$, let
\begin{align}
\psi^\bullet := \psi^* \iota(\rbar_{\rho})
\end{align}
where $\rbar_{\rho}$ is part of a (once and for all chosen) standard solution of the conjugate equations \eqref{eq:conjeqns} for $\rho$ and $\rhobar$. Observe that $\psi^\bullet\in H_{\rhobar}$. The new Pimsner--Popa basis is then given by $\{\psi^\bullet_{\rho,r}\}$. Indeed, by equation \eqref{eq:twoinnerprods} and $a_\rho = 1_{H_\rho}$ (Proposition \ref{prop:arho=1}), we have that $E(\psi_{\rho,r}^\bullet\psi_{\sigma,s}^{\bullet *}) = E(\psi_{\sigma,s}\psi_{\rho,r}^{*}) = \delta_{\rho,\sigma} \delta_{r,s} 1$, \cf the discussion before \cite[\Thm 3.3]{IzLoPo1998} and see \cite[\Lem 6.15]{DeGi2018}. The normalization is the same as the one chosen in \eqref{eq:normalizfields}. 

\begin{lem}\label{lem:bulletPiPoexp}
$\psi_{\sigmabar,s} = \sum_{s'} \lambda_{\sigmabar,\sigma,s,s'} \psi_{\sigma,s'}^{\bullet}$, where $s'=1,\ldots,n_\sigma$ and $\lambda_{\sigmabar,\sigma,s,s'}\in\CC$ are defined by $\iota^{-1} E(\psi_{\sigmabar,s}\psi_{\sigma,s'}) = \lambda_{\sigmabar,\sigma,s,s'} r_{\sigma}$. Similarly, interchanging $\sigma$ with $\sigmabar$ and $r_{\sigma}$ with $\rbar_{\sigma}$.
\end{lem}

\begin{proof}
By the Pimsner--Popa expansion (Proposition \ref{prop:PiPoexp}) 
\begin{align}
\psi_{\sigmabar,s}^* = \sum_{\rho,r} \psi_{\rho,r}^{\bullet *} E(\psi_{\rho,r}^{\bullet} \psi_{\sigmabar,s}^*) = \sum_{s'} \psi_{\sigma,s'}^{\bullet *} E(\psi_{\sigma,s'}^{\bullet} \psi_{\sigmabar,s}^*)
\end{align}
thus $E(\psi_{\sigma,s'}^{\bullet} \psi_{\sigmabar,s}^*) = E(\psi_{\sigma,s'}^{*} \iota(\rbar_{\sigma}) \psi_{\sigmabar,s}^*) = E(\psi_{\sigma,s'}^{*} \psi_{\sigmabar,s}^*) \iota(\sigmabar (\rbar_{\sigma})) = \bar \lambda_{\sigmabar,\sigma,s,s'} \iota(r_{\sigma}^* \sigmabar(\rbar_{\sigma})) = \bar \lambda_{\sigmabar,\sigma,s,s'} 1$ by the conjugate equations.
\end{proof}

\begin{rmk}\label{rmk:lambdaprops}
The matrices $(\Lambda_{\sigmabar,\sigma})_{s,s'} := \lambda_{\sigmabar,\sigma,s,s'}$ are clearly unitary and by the locality assumption they fulfill $(\Lambda_{\sigmabar,\sigma})^t = (\Lambda_{\sigma,\sigmabar})$. Indeed, by the commutation relations \eqref{eq:commutfields} and $\kappa(\sigma) = 1$ (Corollary \ref{cor:spin=1}) we have $\lambda_{\sigmabar,\sigma,s,s'} r_{\sigma} = \iota^{-1} E(\psi_{\sigmabar,s} \psi_{\sigma,s'}) = \iota^{-1} E(\iota(\varepsilon^{\pm}_{\sigma,\sigmabar}) \psi_{\sigma,s'} \psi_{\sigmabar,s}) = \lambda_{\sigma,\sigmabar,s',s} \varepsilon^{\pm}_{\sigma,\sigmabar} \rbar_\sigma = \lambda_{\sigma,\sigmabar,s',s} r_\sigma$ by arguing as in the proof of Proposition \ref{prop:arho=1}, thus $\lambda_{\sigmabar,\sigma,s,s'} = \lambda_{\sigma,\sigmabar,s',s}$.
\end{rmk}

This allows to compute the involution of $\Trig(\N\subset\M)$ (Definition \ref{def:trig}) on matrix units.

\begin{prop}\label{prop:bulletTirred}
Let $\sigma\prec\theta$ be irreducible and $\psi_{\sigma,r},\psi_{\sigma,s}\in H_\sigma$, then 
\begin{align}
T(\psi_{\sigma,s},\psi_{\sigma,r})^\bullet = T(\psi_{\sigma,s}^\bullet,\psi_{\sigma,r}^\bullet).
\end{align}
\end{prop}

\begin{proof}
We have to compute $(w_{\sigma,s} w_{\sigma,r}^*)^\bullet$. Let us use \eqref{eq:snakes}-\eqref{eq:snakes2} as a definition of involution and compute $\theta(w_{\sigma,s}^*) mw = \theta(w_{\sigma,s}^*) \sum_{\rho,r}\iotabar(\iota(w_{\rho,r}) \psi_{\rho,r}) w = \iotabar(\psi_{\sigma,s}) w$ and $w_{\sigmabar,s'}^*\theta(w_{\sigma,s}^*) mw = w_{\sigmabar,s'}^*\iotabar(\psi_{\sigma,s}) w = w^* \iotabar(\psi_{\sigmabar,s'}\psi_{\sigma,s}) w = \iota^{-1} E(\psi_{\sigmabar,s'}\psi_{\sigma,s}) = \lambda_{\sigmabar,\sigma,s',s} r_\sigma$. Similarly, $w^*m^* w_{\sigma,r} = \sum_t \bar\lambda_{\sigma,\sigmabar,r,t} \rbar_{\sigma}^* \sigma(w_{\sigmabar,t}^*)$ and $w^*m^* w_{\sigma,r} \sigma(w_{\sigmabar,r'}) = \sum_t \bar\lambda_{\sigma,\sigmabar,r,t} \rbar_{\sigma}^* \sigma(w_{\sigmabar,t}^*w_{\sigmabar,r'}) = \bar\lambda_{\sigma,\sigmabar,r,r'} \rbar_{\sigma}^*$. By using $1_\theta = \sum_{\tau,t} p_{\tau,t} = \sum_{\tau,t} w_{\tau,t}w_{\tau,t}^*$ we can write
\begin{align}\label{eq:bullet1}
(1_\theta \otimes (w^*m^*))(1_\theta \otimes (w_{\sigma,s} w_{\sigma,r}^*)^* \otimes 1_\theta)((mw)\otimes 1_\theta) = \sum_{s',r'} \lambda_{\sigmabar,\sigma,s',s} \bar \lambda_{\sigma,\sigmabar,r,r'} w_{\sigmabar,s'} w_{\sigmabar,r'}^*
\end{align}
where we used $\sigmabar(\rbar_{\sigma}^*) r_{\sigma} = 1_{\sigmabar}$ from the conjugate equations. With analogous computations, using the second equality in \eqref{eq:snakes}-\eqref{eq:snakes2}, one can also write
\begin{align}
((w^*m^*)\otimes 1_\theta)(1_\theta \otimes (w_{\sigma,s} w_{\sigma,r}^*)^* \otimes 1_\theta)(1_\theta \otimes (mw)) = \sum_{s',r'} \bar \lambda_{\sigmabar,\sigma,r',r} \lambda_{\sigma,\sigmabar,s,s'} w_{\sigmabar,s'} w_{\sigmabar,r'}^*.
\end{align}
The two expressions are clearly the same by Remark \ref{rmk:lambdaprops} above. Now, by Lemma \ref{lem:bulletPiPoexp} we have that $w_{\sigmabar,s'} = \iotabar(\psi_{\sigmabar,s'}^*)w = \sum_{s''} \bar \lambda_{\sigmabar,\sigma,s',s''} \iotabar(\psi_{\sigma,s''}^{\bullet *})w$. Plugging  this expression, \eg in \eqref{eq:bullet1} and summing over $s'$, by the unitarity of $\Lambda_{\sigmabar,\sigma}$ we obtain a contribution of the form $\iotabar(\psi_{\sigma,s}^{\bullet *})w$. Arguing analogously on $w_{\sigmabar,r'}^*$ we obtain $w^*\iotabar(\psi_{\sigma,r}^{\bullet})$ and the proof is complete.
\end{proof}

\begin{prop}\label{prop:triglin}
Let $\psi_1,\psi_2 \in H_\rho$ with $\rho\in\End_0(\N)$, not necessarily irreducible nor $\rho\prec\theta$, then
\begin{align}
T(\psi_1,\psi_2) = \sum_{\sigma,r,s} \mu_{\sigma,r,s}(\psi_1,\psi_2)\, T(\psi_{\sigma,r},\psi_{\sigma,s}),
\end{align}
where $\sigma$ runs over the finitely many inequivalent irreducible subendomorphisms of $\theta$ such that $\sigma\prec\rho$, $r,s=1,\ldots,n_\sigma$, and $\mu_{\sigma,r,s}(\psi_1,\psi_2)\in\CC$ are such that $\iota^{-1} (E(\psi_{\sigma,r}\psi_1^*) E(\psi_2 \psi_{\sigma,s}^*)) = \mu_{\sigma,r,s}(\psi_1,\psi_2) 1$.
\end{prop}

\begin{proof}
Write the finite Pimsner--Popa expansions (Proposition \ref{prop:PiPoexp}) $\psi_1 = \sum_{\sigma,r} E(\psi_1 \psi_{\sigma,r}^*) \psi_{\sigma,r}$ and $\psi_2 = \sum_{\tau,s} E(\psi_2 \psi_{\tau,s}^*) \psi_{\tau,s}$, and plug them in the definition of $T(\psi_1,\psi_2)$. By using the fact that $ww^*$ commutes with $\theta(\N)$ and that $\iota^{-1} (E(\psi_{\sigma,r}\psi_1^*) E(\psi_2 \psi_{\tau,s}^*)) \in \Hom(\tau,\sigma)$ we have the statement.
\end{proof}

The expression for the multiplication and for the involution just obtained (Proposition \ref{prop:multTirred} and \ref{prop:bulletTirred}) extend to arbitrary $T(\psi_1,\psi_2)$, respectively by bilinearity and by antilinearity of the operations.

\begin{prop}\label{prop:multTany}
Let $\psi_1,\psi_2 \in H_\rho$ and $\psi_3,\psi_4 \in H_\sigma$ with $\rho,\sigma\in\End_0(\N)$, not necessarily irreducible nor $\rho,\sigma\prec\theta$, then
\begin{align}
T(\psi_1,\psi_2) \ast T(\psi_3,\psi_4) = T(\psi_1\psi_3,\psi_2\psi_4).
\end{align}
\end{prop}

\begin{proof}
Expanding each factor as in Proposition \ref{prop:triglin} and using Proposition \ref{prop:multTirred} we obtain a sum of 
$\iotabar(\psi_{\tau',r'}^*\psi_{\tau,r}^*)w \iota^{-1}(E(\psi_{\tau,r}\psi_1^*)E(\psi_{\tau',r'}\psi_3^*) E(\psi_4\psi_{\tau',s'}^*) E(\psi_2\psi_{\tau,s}^*)) w^*\iotabar(\psi_{\tau,s}\psi_{\tau',s'})$ where we used the fact that $E(\psi_{\tau',r'}\psi_3^*) E(\psi_4\psi_{\tau',s'}^*)$ is a multiple of $1$. Using the intertwining property of $w$ and summing over $\tau,r,s$ we obtain $\iotabar(\psi_{\tau',r'}^*\psi_1^*)w \iota^{-1}(E(\psi_{\tau',r'}\psi_3^*) E(\psi_4\psi_{\tau',s'}^*)) w^*\iotabar(\psi_2\psi_{\tau',s'})$. We have the claim by writing $\iota^{-1}(E(\psi_{\tau',r'}\psi_3^*) E(\psi_4\psi_{\tau',s'}^*)) = \rho(\iota^{-1}(E(\psi_{\tau',r'}\psi_3^*) E(\psi_4\psi_{\tau',s'}^*)))$, by the intertwining property of $w$ and $\psi_1$, $\psi_2$, and summing over $\tau',r',s'$.
\end{proof}

\begin{prop}\label{prop:bulletTany}
Let $\psi_1,\psi_2 \in H_\rho$ with $\rho\in\End_0(\N)$, not necessarily irreducible nor $\rho\prec\theta$, then
\begin{align}
T(\psi_1,\psi_2)^\bullet = T(\psi_1^\bullet,\psi_2^\bullet),
\end{align}
where $T(\psi_1^\bullet,\psi_2^\bullet)$ does not depend on the choice of standard solutions of the conjugate equations made in the definition of $\psi_1^{\bullet}$, $\psi_2^{\bullet}$.
\end{prop}

\begin{proof}
Expanding $T(\psi_1,\psi_2)$ as in Proposition \ref{prop:triglin} and using Proposition \ref{prop:bulletTirred}, and by the antilinearity of the involution, we 
write $T(\psi_1,\psi_2)^\bullet$ as a sum of $\iota^{-1}(E(\psi_{\sigma,s}\psi_2^*)E(\psi_1\psi_{\sigma,r}^*)) T(\psi_{\sigma,r}^\bullet,\psi_{\sigma,s}^\bullet)$. On the other hand, one can expand $\psi_1^\bullet$, $\psi_2^\bullet$ with respect to the basis $\{\psi_{\sigma,r}^\bullet\}$ and plug the expressions in $T(\psi_1^\bullet,\psi_2^\bullet)$, thus obtaining a sum of $\iota^{-1}(E(\psi_{\sigma,r}^\bullet\psi_1^{\bullet*})E(\psi_2^\bullet\psi_{\sigma,s}^{\bullet*})) T(\psi_{\sigma,r}^\bullet,\psi_{\sigma,s}^\bullet)$. We have to check that the numerical coefficients in the two sums are the same. 

Compute $\iota^{-1}(E(\psi_{\sigma,r}^\bullet\psi_1^{\bullet*})E(\psi_2^\bullet\psi_{\sigma,s}^{\bullet*})) = \iota^{-1}(E(\psi_{\sigma,r}^* \rbar_\sigma \psi_1^{\bullet*})E(\psi_2^{\bullet} \rbar_\sigma^* \psi_{\sigma,s}))$, by using $d(\sigma) 1 =  r_\sigma^*r_\sigma$ we can rewrite it as $d(\sigma)^{-1} r_\sigma^* \iota^{-1}(E(\psi_{\sigma,r}^* \rbar_\sigma \psi_1^{\bullet*})E(\psi_2^{\bullet} \rbar_\sigma^* \psi_{\sigma,s})) r_\sigma = d(\sigma)^{-1} \iota^{-1}(E(\psi_{\sigma,r}^* \psi_1^{\bullet*})E(\psi_2^{\bullet} \psi_{\sigma,s}))$ by the conjugate equations. Now observe that by Remark \ref{rmk:catgenthetaonly} we can assume that $\rho$ is in the rigid tensor \Cstar-category generated by $\theta$, which is braided by assumption. The commutation relations among charged fields \eqref{eq:commutfields} associated with irreducible subendomorphisms of $\theta$ give also, \eg $\psi_2^\bullet \psi_{\sigma,s} = \varepsilon^{\pm}_{\sigma,\rhobar} \psi_{\sigma,s} \psi_2^\bullet$, by first expanding $\psi_2^\bullet$ and then using naturality of the braiding. 
Thus $d(\sigma)^{-1} \iota^{-1}(E(\psi_{\sigma,r}^* \psi_1^{\bullet*})E(\psi_2^{\bullet} \psi_{\sigma,s})) = d(\sigma)^{-1} \iota^{-1}(E(\psi_1^{\bullet*} \psi_{\sigma,r}^*)E(\psi_{\sigma,s}\psi_2^{\bullet}))$ because the braiding is unitary and it cancels on both sides. Continuing, we get that $d(\sigma)^{-1} \rbar_\rho \iota^{-1}(E(\psi_1 \psi_{\sigma,r}^*)E(\psi_{\sigma,s}\psi_2^*)) \rbar_\rho = d(\sigma)^{-1} \rbar_\sigma \iota^{-1}(E(\psi_{\sigma,s}\psi_2^*)E(\psi_1 \psi_{\sigma,r}^*)) \rbar_\sigma = \iota^{-1}(E(\psi_{\sigma,s}\psi_2^*)E(\psi_1 \psi_{\sigma,r}^*))$ by the trace property of the standard right inverses \cite[\Lem 3.7]{LoRo1997}, \cite[\Prop 2.4]{BiKaLoRe2014}, and the desired equality is proven.

To show that $T(\psi_1^\bullet,\psi_2^\bullet)$ is independent of the choice of standard solutions, recall that the latter are uniquely determined up to unitaries \cite[\Lem 3.3]{LoRo1997}, \cite[\Lem 7.23]{GiLo19}. Namely, they are all of the form $\rho(u) \rbar_\rho$ for some unitary $u \in \Hom(\rhobar,\rhobar)$ and a fixed choice of $\rbar_\rho$, similarly for $r_\rho$. Using the intertwining property of $\psi_1$, $\psi_2$, the fact that $ww^*$ commutes with $\theta(\N)$ and the unitarity of $u$ we have the second claim, thus the proof is complete.
\end{proof}

The algebra $\Trig(\N\subset\M)$ is clearly associative and unital, with unit $\one = T(\psi_{\id,1}, \psi_{\id,1})$, $\psi_{\id,1} = 1$. The involutivity is more tricky to prove, because on charged fields it neither holds $\psi^{\bullet\bullet} = \psi$ nor $(\psi_1 \psi_2)^\bullet = {\psi_2}^\bullet  \psi_1^\bullet$. Indeed if $\psi \in H_\rho$ and $\rho\prec\theta$ is irreducible self-conjugate and pseudoreal \cite[\p 139]{LoRo1997}, we get $\psi^{\bullet\bullet} = -\psi$ because $\psi^{\bullet} \in H_{\rhobar} = H_{\rho}$ and $\rbar_\rho = - r_{\rho}$. Moreover, there is in general no way of choosing $\rbar_\rho$, or $r_\rho$ in a standard solution for every $\rho$ in such a way the choice is strictly compatible with the composition of endomorphisms, \ie with the tensor multiplication. The best that one can hope to get is compatibility up to unitaries, \cf \cite{Yam04}.

\begin{prop}
Let $\rho,\sigma\prec\theta$ be irreducible and $\psi_{\rho,r},\psi_{\rho,s}\in H_\rho$, $\psi_{\sigma,u},\psi_{\sigma,v}\in H_\sigma$, then 
\begin{align}
T(\psi_{\rho,r},\psi_{\rho,s})^{\bullet\bullet} &=\, T(\psi_{\rho,r},\psi_{\rho,s}),\\
(T(\psi_{\rho,r},\psi_{\rho,s}) \ast T(\psi_{\sigma,u},\psi_{\sigma,v}))^{\bullet} &=\, T(\psi_{\sigma,u},\psi_{\sigma,v})^{\bullet} \ast T(\psi_{\rho,r},\psi_{\rho,s})^{\bullet}.
\end{align}
\end{prop}

\begin{proof}
The first equation is immediate by the previous discussion. For the second equation, write $T(\psi_{\rho,r}\psi_{\sigma,u},\psi_{\rho,s}\psi_{\sigma,v})^{\bullet} = T((\psi_{\rho,r}\psi_{\sigma,u})^{\bullet},(\psi_{\rho,s}\psi_{\sigma,v})^{\bullet})$ by Proposition \ref{prop:bulletTany}, and observe that $\psi_{\sigma,u}^{\bullet}\psi_{\rho,r}^{\bullet} = \psi_{\sigma,u}^* \iota(\rbar_\sigma) \psi_{\rho,r}^* \iota(\rbar_\rho) = \psi_{\sigma,u}^* \psi_{\rho,r}^* \iota(\rho(\rbar_\sigma)\rbar_\rho)$. Now, $\rho(\rbar_\sigma)\rbar_\rho$ is indeed a part of a standard solution of the conjugate equations \cite[\Cor 3.10]{LoRo1997}, thus $T((\psi_{\rho,r}\psi_{\sigma,u})^{\bullet},(\psi_{\rho,s}\psi_{\sigma,v})^{\bullet}) = T(\psi_{\sigma,u}^{\bullet}\psi_{\rho,r}^{\bullet},\psi_{\sigma,v}^{\bullet}\psi_{\rho,s}^{\bullet})$ again by Proposition \ref{prop:bulletTany}, and the proof is complete.
\end{proof}

\begin{prop}\label{prop:Tcommut}
Let $\psi_1,\psi_2 \in H_\rho$ and $\psi_3,\psi_4 \in H_\sigma$ with $\rho,\sigma\in\End_0(\N)$, not necessarily irreducible nor $\rho,\sigma\prec\theta$, then
\begin{align}
T(\psi_1,\psi_2) \ast T(\psi_3,\psi_4) = T(\psi_3,\psi_4) \ast T(\psi_1,\psi_2).
\end{align}
\end{prop}

\begin{proof}
Combine Proposition \ref{prop:multTany} with the commutation relations among charged fields $\psi_1$, $\psi_3$ and $\psi_2$, $\psi_4$, deduced from \eqref{eq:commutfields} as in the proof of Proposition \ref{prop:bulletTany}.
\end{proof}

Summing up the results of this section:

\begin{thm}\label{thm:trigcommutativestaralg}
Let $\N\subset \M$ be an irreducible local discrete type $\III$ subfactor (Definition \ref{def:semi-discr}, \ref{def:localsubf}). Then $\Trig(\N\subset\M)$ (Definition \ref{def:trig}) is a unital associative and commutative algebra with involution. 
\end{thm}

%%%
\subsection{Duality pairing}\label{sec:dualitypairing}
%%%

Let $\N\subset\M$ be as in Theorem \ref{thm:trigcommutativestaralg}. Denote by $\Hom_{\N}(\M,\M)$ the set of $\N$-bimodular linear maps $\phi:\M\to\M$. Observe that if $\psi\in H_\rho$, $\rho\in\End_0(\N)$, then $\phi(\psi)\in H_\rho$. 
There is a pairing 
\begin{align}
\langle\slot,\slot\rangle\colon \Hom_{\N}(\M,\M)\times \Trig(\N\subset \M)\to \CC
\end{align}
\begin{align}\label{eq:pairing}
  \langle \phi,a\rangle \cdot 1_\iota&:= \psi_1^\ast\phi(\psi_2) ,
  &a=T(\psi_1,\psi_2),\quad pap=a,\; \psi_1,\psi_2\in H_{\theta_p},\; p\in \Proj_0(\theta),\\
  \langle \phi,a\rangle \cdot 1_\iota&:=
    \tikzmatht[.6]{
		  \fill[\colM,rounded corners] (-1,1) rectangle (1,-1.5);
		  \fill[\colN,rounded corners] (0,1) rectangle (1,-1.5);
		  \fill[\colN] (0,1) rectangle (.5,-1.5);
      \draw[very thick] (0,0) arc (90:270:.5)
      node [left] {${\scriptstyle \phi}$};
      \draw[thick] (0,-.5) arc (270:450:.5) node [right] {${\scriptstyle a}$};
      \draw (0,1)--(0,-1.5);}.
\end{align}

Note that the action of $\langle\phi,\cdot\rangle$ preserves the decomposition in Proposition \ref{prop:triglin}, thus we may think of it as defined on matrix units $T(\psi_{\rho,r},\psi_{\rho,s})$ and extend it by linearity to $\Trig(\N\subset \M)$.

\begin{prop}\label{prop:ucpstates}
Let $\phi\colon \M\to \M$ be a $\N$-bimodular linear unital completely positive map, 
then $\omega_\phi := \langle\phi,\slot\rangle$ is a state on $\Trig(\N\subset \M)$.
\end{prop}

\begin{proof}
  \begin{align}
      \langle\phi,\one\rangle&=\langle\phi,T(1,1)\rangle\\
      &=1\phi(1)\\
      &=1.\\
  \end{align}        
  
  \begin{align}  
      \overline{\langle\phi,T(\psi_{\rho,r},\psi_{\rho,s})\rangle}&=
      \phi(\psi_{\rho,s}^\ast)\psi_{\rho,r} \\
      &=d(\rho) E(\psi_{\rho,r}\phi(\psi_{\rho,s}^\ast))\\
      &=\iota(\rbar_\rho^\ast)E(\psi_{\rho,r}\phi(\psi_{\rho,s}^\ast))\iota(\rbar_\rho)\\
      &=E(\psi_{\rho,r}^{\bullet\ast} \phi(\psi_{\rho,s}^{\bullet}))\\
      &=\psi_{\rho,r}^{\bullet\ast} \phi(\psi_{\rho,s}^{\bullet})\\      
      &=\langle\phi,T(\psi_{\rho,r}^\bullet, \psi_{\rho,s}^\bullet)\rangle\\
      &=\langle\phi,T(\psi_{\rho,r},\psi_{\rho,s})^\bullet\rangle. 
   \end{align}
  
   Let $a=\sum_{i=(\rho,r,s)} a_i T(\psi_{\rho,r},\psi_{\rho,s})$, for $a_i\in\CC$
   and multi-indices $i=(\rho,r,s)$, then
   \begin{align}
      \left\langle\phi,a^\bullet \ast a \right\rangle &=
      \sum_{i,i'}\bar a_i a_{i'} \left\langle\phi,T(\psi_{\rho,r},\psi_{\rho,s})^\bullet\ast T(\psi_{\rho',r'},\psi_{\rho',s'}) \right\rangle\\
      &=\sum_{i,i'}\bar a_i a_{i'}
      \left\langle \phi,T(\psi_{\rho,r}^{\bullet}\psi_{\rho',r'},\psi_{\rho,s}^\bullet\psi_{\rho',s'})\right\rangle\\
      &=\sum_{i,i'}\bar a_i a_{i'}
      \psi_{\rho',r'}^\ast\psi_{\rho,r}^{\bullet\ast} \phi(\psi_{\rho,s}^\bullet\psi_{\rho',s'})\\
      &=\sum_{i,i'}\bar a_i a_{i'}
      d(\rho') E(\psi_{\rho,r}^{\bullet\ast} \phi(\psi_{\rho,s}^\bullet\psi_{\rho',s'})\psi_{\rho',r'}^\ast)\\ 
      &=\sum_{i,i'}\bar a_i a_{i'}\iota(\rbar_{\rho'}^\ast) E(\psi_{\rho,r}^{\bullet\ast} \phi(\psi_{\rho,s}^\bullet\psi_{\rho',s'})\psi_{\rho',r'}^\ast) 
      \iota(\rbar_{\rho'})\\ 
      &=\sum_{i,i'}\bar a_i a_{i'}E(\psi_{\rho,r}^{\bullet\ast} \phi(\psi_{\rho,s}^\bullet \iota(\rbar_{\rho'}^\ast)\psi_{\rho',s'})\psi_{\rho',r'}^\ast
      \iota(\rbar_{\rho'}))\\ 
      &=\sum_{i,i'}\bar a_i a_{i'}E(\psi_{\rho,r}^{\bullet\ast} \phi(\psi_{\rho,s}^\bullet\psi_{\rho',s'}^{\bullet\ast})\psi_{\rho',r'}^{\bullet})\\ 
      &=\sum_{i,i'}\bar a_i \psi_{\rho,r}^{\bullet\ast}\phi(\psi_{\rho,s}^\bullet\psi_{\rho',s'}^{\bullet\ast})\psi_{\rho',r'}^{\bullet}a_{i'}
      \\
      &=\sum_{i,i'}y_i^\ast\phi(x_i^\ast x_{i'})y_{i'}\geq 0,
    \end{align}
    where $y_i := a_i\psi^\bullet_{\rho,r}$, $x_i := \psi^{\bullet\ast}_{\rho,s}$, by the complete positivity of $\phi$, \cite[\Cor IV.3.4]{Ta1}.
\end{proof}

%%%
\subsection{Fourier transform}\label{sec:fourier}
%%%

Let $\N\subset \M$ be an irreducible discrete type $\III$ subfactor, with canonical and dual canonical endomorphisms denoted by $\gamma$ and $\theta$ (Section \ref{sec:genQsys}). Consider the \textbf{Fourier transform} \cite[\Def 7]{NiWi95}, see also \cite[\Sec 2.1]{BiRe2019}:
\begin{align}
\cF\colon \Hom(\gamma,\gamma)&\longrightarrow \Hom(\theta,\theta)\\
a&\longmapsto \theta(w^\ast)\bar\iota(a)w
\end{align}
which is injective\footnote{In \cite{NiWi95} the Fourier transform is defined for irreducible semidiscrete subfactors of infinite type, not necessarily braided, nor discrete. However, discreteness follows from semidiscreteness in the case of depth 2 subfactors by \cite{EnNe1996}.}. 
On some domain, \eg provided that $\cF(a),\cF(b)\in\Trig(\N\subset\M)$, we have 
\begin{align}\label{eq:Fismult}
\cF(ab)=\cF(b)\ast\cF(a).
\end{align}
Indeed
\begin{align}
\cF(ab)&=\theta(w^\ast)\bar\iota(ab)w\\
&=\theta(w^\ast)\bar\iota(a)\bar\iota(b)w\\
&=\sum_{\rho,r}\theta(w^\ast)m_{\rho,r}^\ast \theta(\bar\iota(a)w)\bar\iota(b)w\\
&=\sum_{\rho,r}m_{\rho,r}^\ast \theta(\theta(w^*)\bar\iota(a)w)\bar\iota(b)w\\
&=\sum_{\rho,r}m_{\rho,r}^\ast \theta(\cF(a))\bar\iota(b)w\\
&=\sum_{\rho,\sigma,r,s}m_{\rho,r}^\ast \theta(\cF(a))\bar\iota(\iota(w^\ast) b\iota(m_{\sigma,s}))w\\
&=\sum_{\rho,\sigma,r,s}m_{\rho,r}^\ast \theta(\cF(a))\bar\iota(\iota(w^\ast) b)wm_{\sigma,s}\\
&=\sum_{\rho,\sigma,r,s}m_{\rho,r}^\ast \theta(\cF(a))\cF(b) m_{\sigma,s}\\
&=\cF(b)\ast\cF(a).
\end{align}
On some domain, we also have 
\begin{align}\label{eq:Fisinvol}
\cF(a^\ast)=\cF(a)^\bullet,
\end{align}
as we show below for elements $a\in\Hom(\gamma,\gamma)$ such that $\cF(a) \in \Trig(\N\subset\M)$.

\begin{prop}\label{prop:Fdenserange}
If $\N\subset\M$ is in addition discrete and local, then the range of the Fourier transform $\cF\colon\Hom(\gamma,\gamma)\to\Hom(\theta,\theta)$ contains $\Trig(\N\subset\M)$. More precisely, consider the dual fields constructed in Section \ref{sec:dualfields}, then
\begin{align}
\cF(\psi_{\rho,r}^* \iota(\bar{\psi}_{\rho,s})) = T(\psi_{\rho,r},\psi_{\rho,s}).
\end{align}
for every irreducible $\rho\prec\theta$ and $r,s = 1,\ldots,n_\rho$.
\end{prop}

\begin{proof}
Compute
\begin{align}
\cF(\psi_{\rho,r}^* \iota(\psibar_{\rho,s})) &= \iotabar\iota(w^*) \iotabar(\psi_{\rho,r}^*) \iotabar\iota(\psibar_{\rho,s}) w\\
&= \iotabar(\psi_{\rho,r}^* \iota\rho(w^*) \iota(w_{\rho,s}^*) \iota(m))w\\
&= \iotabar(\psi_{\rho,r}^* \iota(w_{\rho,s}^*) \iota(\theta(w^*)m))w\\
&= \iotabar(\psi_{\rho,r}^*) \iotabar\iota(w_{\rho,s}^*)w\\
&= w_{\rho,r} w_{\rho,s}^*\\
&= T(\psi_{\rho,r},\psi_{\rho,s})
\end{align}
using $\theta(w^*)m = 1_\theta$.
\end{proof}

On trigonometric polynomials we have 
\begin{align}
\cF((\psi_{\rho,r}^* \iota(\bar{\psi}_{\rho,s}))^*) = \cF(\psi_{\rho,r}^* \iota(\bar{\psi}_{\rho,s}))^\bullet,
\end{align}
indeed
\begin{align}
\cF((\psi_{\rho,r}^* \iota(\bar{\psi}_{\rho,s}))^*) &= \cF( \iota(\bar{\psi}_{\rho,s})^* \psi_{\rho,r})\\
&= \theta(w^*) \iotabar(\iota(m^*) \iota\iotabar(\psi_{\rho,s}^*)\iota(w)\psi_{\rho,r})w\\
&= \iotabar(\iota(w^*) \gamma(M^*\psi_{\rho,s}^*) \iota(w) \psi_{\rho,r})w\\
&= \iotabar(E(M^*\psi_{\rho,s}^*)\psi_{\rho,r})w,
\end{align}
where $M = \iotabar^{-1}(m)$ and $m$ is the formal sum defined in Section \ref{sec:genQsys}, \ie $M = \sum_{\sigma,t} \iota(w_{\sigma,t}) \psi_{\sigma,t}$ and $w_{\sigma,t} = \iotabar(\psi_{\sigma,t}^*)w$, thus
\begin{align}
\hspace{5.8cm} &= \sum_{\sigma,t}\iotabar(E(\psi_{\sigma,t}^* \psi_{\rho,s}^*)\psi_{\rho,r})ww^* \iotabar(\psi_{\sigma,t})\\
&= \sum_{\sigma,t}\iotabar(E(\psi_{\sigma,t}^* \iota(\sigma(r_{\sigma}^*))\iota(\bar{r}_{\sigma})\psi_{\rho,s}^*)\psi_{\rho,r})ww^* \iotabar(\psi_{\sigma,t})\\
&= \sum_{\sigma,t}\iotabar(\iota(r_{\sigma}^*) E(\psi_{\sigma,t}^\bullet \psi_{\rho,s}^*)\psi_{\rho,r})ww^* \iotabar(\psi_{\sigma,t})\\
&= \sum_{t}\iotabar((\psi_{\rho,r}^\bullet)^*)ww^* \iotabar(\iota(r_{\rho}^*) \iota(\rhobar(\bar{r}_{\rho}))\psi_{\rhobar,t} E(\psi_{\rhobar,t}^\bullet \psi_{\rho,s}^*))\\
&= \sum_{t}\iotabar((\psi_{\rho,r}^\bullet)^*)ww^* \iotabar((\psi_{\rhobar,t}^\bullet)^* E(\psi_{\rhobar,t}^\bullet \psi_{\rho,s}^*)\iota(\bar{r}_{\rho}))\\
&= \iotabar((\psi_{\rho,r}^\bullet)^*)ww^* \iotabar(\psi_{\rho,s}^\bullet)\\
&= T(\psi_{\rho,r}^\bullet,\psi_{\rho,s}^\bullet)\\
&= T(\psi_{\rho,r},\psi_{\rho,s})^\bullet
\end{align}
using the conjugate equations, the Pimsner--Popa expansion (Proposition \ref{prop:PiPoexp}) and the fact that $E(\psi_{\sigma,t}^\bullet \psi_{\rho,s}^*) = 0$ if $[\sigma] \neq [\rhobar]$, and $E(\psi_{\sigma,t}^\bullet \psi_{\rho,s}^*)\in\CC\one$ if $\sigma = \rhobar$. 

Finally, equation \eqref{eq:Fisinvol} follows from Proposition \ref{prop:triglin} and from the linearity of the Fourier transform.

%%%
\subsection{The reduced \Cstar-algebra of a local discrete type $\III$ subfactor}\label{sec:Cistarred}
%%%

In the same assumptions of Section \ref{sec:trig}, we want to endow the commutative unital involutive algebra $\Trig(\N\subset \M)$ (Theorem \ref{thm:trigcommutativestaralg}) with a \Cstar-norm and then take the norm closure. Let 
\begin{align}\label{eq:stateEtrig}
\omega_E:=\langle E,\slot\rangle
\end{align}
be the state on $\Trig(\N\subset \M)$ given by Proposition \ref{prop:ucpstates}. 

\begin{lem}\label{lem:haarfaith}
The state $\omega_E$ is faithful, namely $\omega_E(a^\bullet * a) > 0$ if $a\in\Trig(\N\subset \M)$, $a\neq 0$.
\end{lem}

\begin{proof}
By the chosen normalization of the charged fields \eqref{eq:normalizfields} and by $a_\rho = 1_{H_\rho}$, which follows from locality (Proposition \ref{prop:arho=1}), arguing as in the proof of Proposition \ref{prop:ucpstates} we have
\begin{align}
\omega_E(T(\psi_{\rho,r},\psi_{\rho,s})^\bullet * T(\psi_{\sigma,r'},\psi_{\sigma,s'}))
&= \psi_{\sigma,r'}^*\psi_{\rho,r}^{\bullet*} E(\psi_{\rho,s}^\bullet\psi_{\sigma,s'})\\
&= E(\psi_{\rho,r}^{\bullet*} E(\psi_{\rho,s}^\bullet\psi_{\sigma,s'}^{\bullet*}) \psi_{\sigma,r'}^\bullet)\\
&= E(\psi_{\rho,r}^{\bullet*} E(\psi_{\sigma,s'}\psi_{\rho,s}^*) \psi_{\sigma,r'}^\bullet)\\
&= E(\psi_{\rho,r}^{\bullet*} \psi_{\rho,r'}^\bullet) \delta_{\rho,\sigma}\delta_{s,s'}\\
&= d(\rho) \delta_{\rho,\sigma}\delta_{s,s'}\delta_{r,r'},
\end{align}
from which faithfulness follows.
\end{proof}

Denote by $(\lambda,\Hil_E,\xi_E)$ the GNS representation of $\Trig(\N\subset\M)$ associated with $\omega_E$, see \eg \cite[\Sec 1.3]{KhMo15}, where $\xi_E\in\Hil_E$ is the cyclic and separating unit vector such that $\omega_E=(\xi_E,\lambda(\slot)\xi_E)$. Denote by $\|\cdot\|_E$ the norm on $\Hil_E$ induced by the GNS inner product. 

\begin{lem}\label{lem:lambdafaithbound}
The representation $\lambda$ is faithful and by bounded linear operators on $\Hil_E$, namely $\lambda\colon \Trig(\N\subset \M)\to \B(\Hil_E)$, indeed
\begin{align}\label{eq:GNSbounded}
\|T(\psi_{\rho,r},\psi_{\rho,s}) * T(\psi_{\sigma,t},\psi_{\sigma,u})\|_E \leq d(\rho)\, \|T(\psi_{\sigma,t},\psi_{\sigma,u})\|_E.
\end{align}
\end{lem}

\begin{proof}
Compute
\begin{align}
&\|T(\psi_{\rho,r},\psi_{\rho,s}) * T(\psi_{\sigma,t},\psi_{\sigma,u})\|_E^2 
\\&= \omega_E(
T(\psi_{\sigma,t},\psi_{\sigma,u})^\bullet \ast T(\psi_{\rho,r},\psi_{\rho,s})^\bullet\ast
T(\psi_{\rho,r},\psi_{\rho,s}) \ast T(\psi_{\sigma,t},\psi_{\sigma,u}))
\\&= \omega_E(
T(\psi_{\sigma,t}^\bullet\psi_{\rho,r}^\bullet\psi_{\rho,r}\psi_{\sigma,t},\psi_{\sigma,u}^\bullet\psi_{\rho,s}^\bullet \psi_{\rho,s}\psi_{\sigma,u}))
\\&=
\psi_{\sigma,t}^\ast\psi_{\rho,r}^\ast\psi_{\rho,r}^{\bullet\ast} \psi_{\sigma,t}^{\bullet\ast}E(\psi_{\sigma,u}^\bullet\psi_{\rho,s}^\bullet \psi_{\rho,s}\psi_{\sigma,u})
\\&=
\psi_{\rho,r}^\ast\psi_{\rho,r}^{\bullet\ast} \psi_{\sigma,t}^{\bullet\ast}E(\psi_{\sigma,u}^\bullet\psi_{\rho,s}^\bullet \psi_{\rho,s}\psi_{\sigma,u}^{\bullet\ast})\psi_{\sigma,t}^\bullet
\\&=
\psi_{\rho,r}^{\bullet\ast} \psi_{\sigma,t}^{\bullet\ast}E(\psi_{\sigma,u}^\bullet\underbrace{\psi_{\rho,s}^\bullet \psi_{\rho,s}^{\bullet\ast}}_{\leq d(\rho)1}\psi_{\sigma,u}^{\bullet\ast})\psi_{\sigma,t}^\bullet\psi_{\rho,r}^\bullet
\\&\leq d(\rho) \psi_{\rho,r}^{\bullet\ast}\underbrace{\psi_{\sigma,t}^{\bullet\ast}E(\psi_{\sigma,u}^\bullet\psi_{\sigma,u}^{\bullet\ast})\psi_{\sigma,t}^\bullet}_{\in\CC 1} \psi_{\rho,r}^\bullet
\\&=d(\rho)^2 \psi_{\sigma,t}^{\bullet\ast}E(\psi_{\sigma,u}^\bullet\psi_{\sigma,u}^{\bullet\ast})\psi_{\sigma,t}^\bullet
\\&=d(\rho)^2 \psi_{\sigma,t}^*\psi_{\sigma,t}^{\bullet\ast}E(\psi_{\sigma,u}^\bullet\psi_{\sigma,u})
\\&=d(\rho)^2 \|T(\psi_{\sigma,t},\psi_{\sigma,u})\|_E^2.
\end{align}
\end{proof}

\begin{rmk}
Note that the same computation leading to \eqref{eq:GNSbounded} holds replacing the expectation $E:\M\to\N\subset\M$ with any other $\N$-bimodular ucp map $\phi:\M\to\M$.
\end{rmk}

\begin{defi}\label{def:Cred}
We denote by $\Cred(\N\subset \M)$ be the norm closure of $\lambda(\Trig(\N\subset \M))$ in $\B(\Hil_E)$, and we refer to it as the \textbf{reduced \Cstar-algebra} of $\N\subset\M$. 

Thus 
\begin{align}
\Cred(\N\subset \M) \cong C(K)
\end{align}
by Gelfand duality, where $K$ is a compact Hausdorff space, the spectrum of the reduced \Cstar-algebra, and $C(K)$ the algebra of continuous functions on $K$. Note that $\Cred(\N\subset \M)$ is separable by construction, thus $K$ is metrizable.
\end{defi}

\begin{lem}\label{lem:Eisfaithful}
The state $\omega_E$ defined in \eqref{eq:stateEtrig} extends by the formula $\omega_E = (\xi_E,\slot\xi_E)$ to a faithful state on $\Cred(\N\subset \M)$. 
\end{lem}

\begin{proof}
Observe that $\xi_E$ is cyclic for $\lambda(\Trig(\N\subset\M))$ thus separating for $\lambda(\Trig(\N\subset\M))'$ and that $\Cred(\Trig(\N\subset\M))\subset\lambda(\Trig(\N\subset\M))'$ by the commutativity of $\Trig(\N\subset\M)$.
\end{proof}

%%%
\subsection{A duality theorem: UCP/states correspondence}\label{sec:UCP/states}
%%%

The goal of this section is to show that the duality pairing introduced in Section \ref{sec:dualitypairing} gives a bijective correspondence, in fact a homeomorphism, between the \emph{states on $\Cred(\N\subset \M)$} and the set of \emph{$\N$-bimodular ucp maps on $\M$} (Corollary \ref{cor:localNbimoducp}, Theorem \ref{thm:ucpstatesduality}). These are the main technical results of the paper.

Let $\N\subset \M$ be an irreducible semidiscrete subfactor (Definition \ref{def:semi-discr}). Let $\Omega\in\Hil$ be a standard unit vector for $\M\subset \B(\Hil)$ such that the associated state is $E$-invariant, where $E:\M\to\N\subset\M$ is the unique normal faithful conditional expectation for the subfactor. Then $E$ is $\Omega$-adjointable by Takesaki's theorem \cite[\Sec 10]{Str81}, in symbols $E\in{\UCP}^{\sharp}(\M,\Omega)$. 
In the notation of Section \ref{sec:Omegaadjmaps}, $e := V_E$ is the Jones projection of $E$ with respect to $\Omega$, namely $e m \Omega = E(m)\Omega$ for every $m\in\M$, and $E^{\sharp} = E$ because $(m_1 \Omega, E(m_2) \Omega) = (m_1 \Omega, e m_2 \Omega) = (E(m_1) \Omega, m_2 \Omega)$ for every $m_1,m_2\in\M$.

\begin{defi}\label{def:UCPbim}
We denote by ${\UCP}_\N(\M,\Omega)$ the set of $\N$-bimodular $\Omega$-preserving ucp maps $\phi:\M\to\M$, and by ${\UCP}^{\sharp}_\N(\M,\Omega)$ the subset of $\Omega$-adjointable ones.
\end{defi}

Clearly, $E\in{\UCP}^{\sharp}_\N(\M,\Omega)$. Observe also that if $\phi$ is $\N$-bimodular and $\Omega$-adjointable, then the adjoint $\phi^\sharp$ is also $\N$-bimodular. A remarkable consequence of discreteness and locality of the subfactor (Definition \ref{def:semi-discr} and \ref{def:localsubf}), or better of the property $a_\rho = 1_{H_\rho}$ implied by locality (Proposition \ref{prop:arho=1}), is the following:

\begin{lem}\label{lem:OmegapresisOmegaadj}
Let $\N\subset\M$ be an irreducible discrete type $\III$ subfactor such that $a_\rho = 1_{H_\rho}$ for every irreducible $\rho\prec\theta$ and let $\Omega$ be as above. Then every $\N$-bimodular $\Omega$-preserving ucp map $\phi:\M\to\M$, \ie such that $(\Omega, \phi(\slot)\Omega) = (\Omega, \slot\Omega)$, is automatically $\Omega$-adjointable. 
\end{lem}

\begin{proof}
By the equivalent conditions stated in Proposition \ref{prop:adjointiffmodular}, it is enough to show that $\phi$ commutes with the modular group $\sigma_t$ of $(\M,\Omega)$, namely that $\phi\circ \sigma_t = \sigma_t \circ \phi$ on $\M$, $t\in\RR$. By \cite[\Lem 3.8]{IzLoPo1998}, $\M$ is weakly spanned by $\N$ and by the spaces of charged fields $H_\rho$, with $\rho\prec\theta$ irreducible. Now, by the proof of \cite[\Lem 3.8]{IzLoPo1998} and using the fact that $a_\rho = 1_{H_\rho}$ for every irreducible $\rho\prec\theta$ by assumption, we have that $\sigma_t(\psi) = \iota(u_{\rho,t}) \psi$ for every $\psi\in H_\rho$, where $u_{\rho,t}$ is a unitary in $\N$ given by Connes cocycles.
Thus $\phi(\sigma_t(\psi)) = \iota(u_{\rho,t}) \phi(\psi) = \sigma_t(\phi(\psi))$, by $\N$-bimodularity of $\phi$ and observing that $\phi$ maps each $H_\rho$ into itself. Moreover, $\phi(\sigma_t(\iota(n))) = \sigma_t(\iota(n)) = \sigma_t(\phi(\iota(n)))$, $n\in\N$, because $\phi(\iota(n))=\iota(n)$ and $\sigma_t(\iota(n))\in\iota(\N)$ for every $t\in\RR$ by Takesaki's theorem, concluding the proof. 
\end{proof}

More generally, for irreducible discrete but not necessarily local nor braided subfactors:

\begin{lem}\label{lem:NfixingisOmegapres}
Let $\N\subset\M$ be an irreducible discrete type $\III$ subfactor and let $\Omega$ be as above.
If $\phi:\M\to\M$ is a ucp map acting trivially on $\N$, \ie $\phi_{\restriction\iota(\N)} = \id$, then $\phi$ is $\N$-bimodular and it is absorbed by $E$ on both sides, namely 
\begin{align}
\phi \circ E = E, \quad E \circ \phi = E.
\end{align}
Moreover, $\phi$ is automatically $\Omega$-preserving.
\end{lem}

\begin{proof}
The $\N$-bimodularity follows from Choi's theorem \cite{Ch1974}, namely $\phi(\iota(n)^*\iota(n)) = \iota(n^*n) = \phi(\iota(n)^*)\phi(\iota(n))$ for some $n\in\N$ implies $\phi(m\iota(n)) = \phi(m)\phi(\iota(n)) = \phi(m)\iota(n)$ and $\phi(\iota(n)m) = \phi(\iota(n))\phi(m) = \iota(n)\phi(m)$ for every $m\in\M$, \cf \cite[\Prop 3.4]{Bi2016}.

The first absorption property is immediate. To show the second one we observe that $E(\phi(\psi)) = 0 = E(\psi)$ for every $\psi\in H_\rho$, $[\rho] \neq [\id]$, and $\psi = \lambda 1$, $\lambda\in\CC$, if $[\rho] = [\id]$ (indeed $\psi_{\id,1} = 1$, by equation \eqref{eq:normalizvacuumfield}, thus $E(\psi_{\rho,r}) = 0$ by the very definition of Pimsner--Popa basis) and we conclude by density of $\N$ and $H_\rho$, $\rho\prec\theta$, in $\M$, as in the proof of the previous lemma.

The second absorption property together with the fact that $E$ is $\Omega$-preserving imply that $\phi$ is $\Omega$-preserving.
\end{proof}

\begin{cor}\label{cor:indepOmega}
If $\N\subset\M$ is an irreducible discrete type $\III$ subfactor, then ${\UCP}_\N(\M,\Omega)$ is independent of the choice of $\Omega$ as above. If in addition $a_\rho = 1_{H_\rho}$ for every irreducible $\rho\prec\theta$ (\eg for local or finite index subfactors), then ${\UCP}_\N(\M,\Omega) = {\UCP}^{\sharp}_\N(\M,\Omega)$.
\end{cor}

Summing up the results of this section and those mentioned in Section \ref{sec:Omegaadjmaps}, we obtain a very easy description of the elements of ${\UCP}^{\sharp}_\N(\M,\Omega)$.

\begin{cor}\label{cor:localNbimoducp}
If $\N\subset\M$ is an irreducible local discrete type $\III$ subfactor, then a ucp map $\phi:\M\to\M$ belongs to ${\UCP}^{\sharp}_\N(\M,\Omega)$ if and only if $\phi_{\restriction\iota(\N)} = \id$.
\end{cor}

Now we define a topology on ucp maps.

\begin{defi}\label{def:L2conv}
A net $\{\phi_\alpha\}\subset{\UCP}^{\sharp}_\N(\M,\Omega)$ \textbf{converges in $L^2(\M,\Omega)$} to $\phi\in{\UCP}^{\sharp}_\N(\M,\Omega)$, if and only if $\|\phi_\alpha(m)\Omega-\phi(m)\Omega\|\to 0$ for all $m\in \M$. We write $\phi_\alpha\to\phi$ in $L^2(\M,\Omega)$.
\end{defi}

The convergence $\phi_\alpha\to\phi$ in $L^2(\M,\Omega)$, already encountered in Section \ref{sec:amenability} for ${\UCP}^{\sharp}(\M,\Omega)$, is clearly equivalent to $V_\alpha\to V$ in the strong operator topology of $\B(\Hil)$, where $V_\alpha\in \B(\Hil)$ and $V \in \B(\Hil)$ are the linear contractions respectively associated with $\phi_\alpha$ and $\phi$ (Section \ref{sec:Omegaadjmaps}). 
We note that $V \in \N'$, if $\phi$ is $\N$-bimodular. If $\phi$ is $\Omega$-adjointable, then $[V,J_{\M,\Omega}]=0$ thus $V \in J_{\M,\Omega}\N'J_{\M,\Omega}=\M_1$. Consequently, $\phi_\alpha\to\phi$ in $L^2(\M,\Omega)$ corresponds to the strong operator convergence $V_\alpha\to V$ in $\N'\cap \M_1$.

\begin{rmk}\label{rmk:Topologies}
Since $\N'\cap \M_1$ ($\cong \theta(\N)' \cap \N = \Hom(\theta,\theta)$) is isomorphic to a direct sum of finite matrix algebras by discreteness and irreducibility of the subfactor, any norm bounded net $V_\alpha\to V$ converges in the strong operator topology of $\N'\cap\M_1$, if and only if it converges in the weak operator topology. 

Indeed, assume that $V_\alpha\to V$ converges in the weak operator topology and choose a sequence of minimal central projections $p_m\in\N'\cap \M_1$ such that $\sum_m p_m = 1$ in the strong operator topology. Hence $p_m V_\alpha \to p_m V$ in norm, thus in the strong operator topology, for every fixed $m$. For every $\xi\in\Hil$, denote $\xi_m := p_m \xi$ thus $\|\xi\| = \sum_m \|\xi_m\|$. We have
\begin{align}
\|V_\alpha\xi - V\xi\| = \sum_m \|V_\alpha\xi_m - V\xi_m\| \leq \sum_{m=1}^{N} \|V_\alpha\xi_m - V\xi_m\| + C \sum_{m=N}^{\infty} \|\xi_m\| ,
\end{align}
where $C>0$ is some constant given by the norm boundedness of $V_\alpha$, and by an $\epsilon/2$ argument we conclude $\|V_\alpha\xi - V\xi\| \to 0$.

As a consequence, $\phi_\alpha\to\phi$ in $L^2(\M,\Omega)$ if and only if $\phisharp_\alpha\to\phisharp$ in $L^2(\M,\Omega)$.
\end{rmk}

From now on we assume, without further mention, that $\N\subset\M$ is irreducible, discrete and local. Recall the map defined by duality in equation \eqref{eq:pairing}:
\begin{align}
\phi \mapsto \omega_\phi = \langle\phi,\slot\rangle
\end{align}
from $\N$-bimodular ucp maps on $\M$ to states on $\Trig(\N\subset\M)$, or equivalently on its image under the GNS representation $(\lambda,\Hil_E,\xi_E)$ (Lemma \ref{lem:lambdafaithbound}).

\begin{lem} 
  \label{lem:FiniteRankReducedStates}
  If $\phi\in{\UCP}^{\sharp}_\N(\M,\Omega)$ has finite support (Definition \ref{def:finitesuppucp}), then $\omega_\phi$ extends to a state on $\Cred(\N\subset \M)$ (Definition \ref{def:Cred}).
\end{lem}
\begin{proof}
  A computation shows that $\omega_{\phi}=\omega_E(t\ast \slot)$
  for the operator 
\begin{align}
  t=\sum_{\rho,r} T(\phi(\psi_{\rho,r}),\psi_{\rho,r}) \in \Trig(\N\subset \M).
\end{align}
Thus we can write $\omega_{\phi}(a)=\omega_E(t\ast a) = (\xi_E,\lambda(t)\lambda(a)\xi_E)$ for $a\in\Trig(\N\subset \M)$.
By Proposition \ref{prop:ucpstates}, it defines a state $x\in\Cred(\N\subset \M) \mapsto (\xi_E,\lambda(t)x\xi_E)$
which extends $\omega_\phi$.
\end{proof}

\begin{prop}\label{prop:finitesuppapprox}
  If $\phi\in{\UCP}^{\sharp}_\N(\M,\Omega)$ can be approximated in $L^2(\M,\Omega)$ by finite support maps in ${\UCP}^{\sharp}_\N(\M,\Omega)$, 
  then $\omega_\phi$ extends to a state on $\Cred(\N\subset \M)$.
\end{prop}

\begin{proof}
  Let $\phi_\alpha\to \phi$ in $L^2(\M,\Omega)$ where each $\phi_\alpha$ has finite support. 
  By Lemma \ref{lem:FiniteRankReducedStates}, each $\omega_{\phi_\alpha}$ extends to a state on $\Cred(\N\subset \M)$.
  For $a=T(\psi_{\rho,r},\psi_{\sigma,s})\in \Trig(\N\subset \M)$ we have
  \begin{align}
    |\omega_\phi(a)-\omega_{\phi_\alpha}(a)| &= \left|\left(\Omega, \psi_{\rho,r}^*(\phi-\phi_\alpha)(\psi_{\sigma,s})\Omega\right)\right|
    \\&=|(\psi_{\rho,r}\Omega,(V-V_\alpha)\psi_{\sigma,s}\Omega)|\to 0
  \end{align}
and the same holds for an arbitrary $a \in \Trig(\N\subset \M)$. But this implies
  \begin{align}
    |\omega_\phi(a)| \leq \sup_\alpha |\omega_{\phi_\alpha}(a)| \leq \|\lambda(a)\|
  \end{align}
thus $\omega_\phi$ extends by density to a state on $\Cred(\N\subset \M)$.
\end{proof}

\begin{cor}\label{cor:statesextend}
Every $\phi\in{\UCP}^{\sharp}_\N(\M,\Omega)$ defines by extension a state $\omega_\phi$ on 
$\Cred(\N\subset \M)$.
\end{cor}

\begin{proof}
The subfactor is strongly relatively amenable (Definition \ref{def:strongamenab}) by Corollary \ref{cor:braideddiscreteisstrongamenab}.
Thus by Lemma \ref{lem:OmegapresisOmegaadj}, there is a net $\phi_\alpha\in{\UCP}^{\sharp}_\N(\M,\Omega)$ with finite support such that $\phi_\alpha\to\id$ in $L^2(\M,\Omega)$. The maps $\phi\circ\phi_\alpha\in{\UCP}^{\sharp}_\N(\M,\Omega)$ have finite support and $\phi\circ\phi_\alpha\to\phi$ in $L^2(\M,\Omega)$, and we conclude by Proposition \ref{prop:finitesuppapprox}.
\end{proof}

In the following, we denote again by $\omega_\phi$ the extended state on $\Cred(\N\subset \M)$.

\begin{lem}\label{lem:dualityisinjective}
  The duality map $\phi\mapsto\omega_\phi$ is injective.
\end{lem}

\begin{proof} 
Assume that $\omega_{\phi_1} = \omega_{\phi_2}$ for $\phi_1,\phi_2\in{\UCP}^{\sharp}_\N(\M,\Omega)$.
We first show that $\phisharp_1(\psi_{\rho,r}^*)=\phisharp_2(\psi_{\rho,r}^*)$ for all $\rho,r$, where $\phisharp$ denotes as before the $\Omega$-adjoint of $\phi\in{\UCP}^{\sharp}_\N(\M,\Omega)$. Indeed, writing the Pimsner--Popa expansion (Proposition \ref{prop:PiPoexp}) we obtain
\begin{align}
\phisharp(\psi_{\rho,r}^*) &= \sum_{\sigma,s} \psi_{\sigma,s}^* E(\psi_{\sigma,s}\phisharp(\psi_{\rho,r}^*))\\
&= \sum_{s} \psi_{\rho,s}^* (\Omega,\psi_{\rho,s}\phisharp(\psi_{\rho,r}^*)\Omega)\\
&= \sum_{s} \psi_{\rho,s}^* (\Omega,\phi(\psi_{\rho,s})\psi_{\rho,r}^*\Omega)\\
&= \sum_{s} \psi_{\rho,s}^* E(\phi(\psi_{\rho,s})\psi_{\rho,r}^*)\\
&= \sum_{s} \psi_{\rho,s}^* \frac{1}{d(\rho)} \psi_{\rho,r}^*\phi(\psi_{\rho,s}) = \sum_{s} \psi_{\rho,s}^* \frac{1}{d(\rho)} \omega_\phi(T(\psi_{\rho,r},\psi_{\rho,s}))
\end{align}
by using the fact that $\iota^{-1}E(\psi_{\sigma,s}\phisharp(\psi_{\rho,r}^*)) \in \Hom(\rho,\sigma)$, the $E$-invariance of $(\Omega,\slot\Omega)$ and equation \eqref{eq:twoinnerprods} together with $a_\rho = 1_{H_\rho}$ (Proposition \ref{prop:arho=1}).

Now let $m\in \M$, then $(\psi_{\rho,r}^*\iota(n)\Omega,\phi_1(m)\Omega)=(\psi_{\rho,r}^*\iota(n)\Omega,\phi_2(m)\Omega)$ for every $n\in\N$, by the previous computation and by the $\N$-bimodularity of $\phisharp_1$ and $\phisharp_2$.
Since the vectors $\{\psi_{\rho,r}^*\iota(n)\Omega, n\in\N\}$ are total in $\Hil$, by the very definition of Pimsner--Popa basis (Definition \ref{def:PiPobasis}) we get $\phi_1(m)\Omega=\phi_2(m)\Omega$. Since $\Omega$ is separating for $\M$, we conclude $\phi_1(m)=\phi_2(m)$.
\end{proof}
 
\begin{lem}\label{lem:dualityishomeo}
The duality map $\phi\mapsto \omega_\phi$ is bicontinuous between ${\UCP}^{\sharp}_{\N}(\M,\Omega)$ equipped with the topology of $L^2(\M,\Omega)$-convergence and the state space of $\Cred(\N\subset \M)$ with the weak* topology.
\end{lem}

\begin{proof}
We need to show that $\omega_\alpha := \omega_{\phi_\alpha}\to\omega = \omega_\phi$ in the weak* topology if and only if  $\phi_\alpha\to\phi$ in $L^2(\M,\Omega)$. 
By Remark \ref{rmk:Topologies} it is enough to show the equivalence with the convergence of $V_\alpha \to V$ in the weak operator topology, \ie of $V_\alpha^* = V_{\phisharp_\alpha}\to V^* = V_{\phisharp}$.
By direct computation, as in the proof of Lemma \ref{lem:dualityisinjective}, we have 
\begin{align}
(\psi_{\rho,r}^*\iota(n_1)\Omega,(V_\alpha^* \!-\! V^*) \psi_{\rho,s}^*\iota(n_2)\Omega) 
= \frac{1}{d(\rho)} (\iota(n_1)\Omega,\iota(n_2)\Omega) \left(\omega_\alpha(T(\psi_{\rho,s},\psi_{\rho,r})) \!-\!
\omega(T(\psi_{\rho,s},\psi_{\rho,r}))\right),
\end{align}
where $n_1,n_2\in\N$. The statement follows on the one hand because $\{\psi_{\rho,r}\iota(n)\Omega, n\in\N\}$ are total in $\Hil$, on the other hand by choosing $n_1=n_2=1$ and using that $\{\lambda(T(\psi_{\rho,s},\psi_{\rho,r}))\}$ are total in $\Cred(\N\subset \M)$ by definition.
\end{proof}

\begin{lem}\label{lem:OmegaadjL2cpt}
The topology of $L^2(\M,\Omega)$-convergence makes ${\UCP}^{\sharp}_{\N}(\M,\Omega)$ into a compact Hausdorff topological space. 

Moreover, it coincides on ${\UCP}^{\sharp}_{\N}(\M,\Omega)$ with the pointwise (ultra-)weak operator topology, also called Bounded-Weak (BW) topology.
\end{lem}

\begin{proof}
By the Banach--Alaoglu theorem, one can show that ${\UCP}^{\sharp}_{\N}(\M,\Omega)$ is compact in the BW topology. This is because ${\UCP}^{\sharp}_{\N}(\M,\Omega)$ is BW-closed in the set of bounded linear operators from $\M$ to $\B(\Hil)$ with norm at most 1, as one can directly check, and the latter is BW-compact by \cite[\Thm 7.4]{Pa2002}.

Now, the BW-convergence of $\phi_\alpha \to \phi$ implies $(\xi, \phi_\alpha(m)\Omega) \to (\xi, \phi(m)\Omega)$ for every $m\in\M$, $\xi\in\Hil$, thus $V_\alpha \to V$ in the weak operator topology. By Remark \ref{rmk:Topologies}, this is equivalent to $\phi_\alpha \to \phi$ in $L^2(\M,\Omega)$. Vice versa, $\| \phi_\alpha(m)\Omega - \phi(m)\Omega\| \to 0$, $m\in\M$, implies $\| \phi_\alpha(m)m'\Omega - \phi(m)m'\Omega\| \to 0$ for every $m'\in\M'$, thus $\phi_\alpha \to \phi$ in the pointwise (ultra-)strong operator topology by the cyclicity of $\Omega$ for $\M'$, \cf Remark \ref{rmk:strongamenabrepnindep}.

The Hausdorff property follows immediately from the separating property of $\Omega$ for $\M$.
\end{proof}

\begin{thm}\label{thm:ucpstatesduality}
Let $\N\subset \M$ be an irreducible local discrete type $\III$ subfactor (Definition \ref{def:semi-discr}, \ref{def:localsubf}). Let $\Omega\in\Hil$ be a standard vector for $\M\subset\B(\Hil)$ such that the associated state is invariant with respect to the unique normal faithful conditional expectation $E:\M\to\N\subset\M$. 
Let $\Cred(\N\subset \M)$ be the commutative unital separable \Cstar-algebra obtained from $\Trig(\N\subset\M)$ in the GNS representation $(\lambda,\Hil_E,\xi_E)$ associated with $\omega_E = \langle E, \slot \rangle$ (Definition \ref{def:Cred}).

Then the duality map
\begin{align}
\phi \mapsto \omega_\phi
\end{align}
is a bijection (thus a homeomorphism) between:
\begin{itemize}
\item the set of maps $\phi\colon \M \to \M$ that are normal faithful unital completely positive $\N$-bimodular and $\Omega$-adjointable, 
in symbols $\phi\in{\UCP}^{\sharp}_\N(\M,\Omega)$,
\item the set of states $\omega$ on $\Cred(\N\subset \M)$.
\end{itemize}
In particular, $\phi$ is extreme if and only if $\omega_\phi$ is a pure state, \ie a character of $\Cred(\N\subset \M)$.
\end{thm}

\begin{proof}
By Lemma \ref{lem:dualityisinjective} and Corollary \ref{cor:statesextend}, we only have to show that the map $\phi \mapsto \omega_\phi$ is surjective. To do so, by the continuity of the duality map (Lemma \ref{lem:dualityishomeo}) and by the compactness of ${\UCP}^{\sharp}_{\N}(\M,\Omega)$ (Lemma \ref{lem:OmegaadjL2cpt})
it is enough to show that the image is dense in the state space of $\Cred(\N\subset \M) \cong C(K)$, \ie the set of probability Radon measures $P(K)$. The latter is the weak*-closed convex hull of the pure states, \ie the Dirac measures on the spectrum $K$, 
and these can be approximated in the weak* topology
by states of the form $\omega_f := f \omega_E$
where $f\in C(K)$, $f\geq 0$, $\int_K f \dd\omega_E = 1$ by standard arguments using the metrizability of $K$ and the faithfulness of the measure $\omega_E$ on $K$ associated with $E$ (Lemma \ref{lem:haarfaith}). 
Explicitly, for every $x\in K$ we can approximate $\delta_x$ (the Dirac measure concentrated in $x$) by measures of the form $f_{(x,\epsilon)} \omega_E$ where $f_{(x,\epsilon)}$ is a positive continuous function with integral one and supported in a ball $B_\epsilon(x)$ of radius $\epsilon > 0$ centered at $x$. For every $g\in C(K)$
\begin{align}\label{eq:approxdiracdeltas}
\hspace{-8mm} \left | \int_K f_{(x,\epsilon)}(y) g(y) \dd \omega_E(y) - g(x)\right | = \left | \int_K f_{(x,\epsilon)}(y) (g(y) - g(x)) \dd \omega_E(y) \right | \leq \! \sup_{y\in B_\epsilon(x)} \! \left | g(y) - g(x) \right |
\end{align}
which tends to zero as $\epsilon\to 0$ by uniform continuity. 

Now, every $f$ as before is of the form $f=g^*g$, $g\in C(K)$, and it can be approximated in norm by a sequence of functions $f_n=g_n^* g_n$, $g_n\in \lambda(\Trig(\N\subset \M)) \subset C(K)$, with $\int_K f_n \dd\omega_E = 1$ by the compactness of $K$. The associated sequence of states $\omega_{f_n}$ converges to $\omega_f$ in the weak* topology. 

Thus our task is to show that the image of $\phi \mapsto \omega_\phi$ contains all the states of the form $\omega_{f} =  f \omega_E$, where $f\in \lambda(\Trig(\N\subset \M)) \subset C(K)$, $f\geq 0$, $\int_K  f \dd\omega_E = 1$. 
We shall show this using the Fourier transform from Section \ref{sec:fourier} and a characterization of positive operators in $\Hom(\gamma,\gamma)$, where $\gamma$ is the canonical endomorphism, based on the Connes--Stinespring representation of completely positive maps, see \eg \cite[\App A.2]{Bi2016}. 

By Proposition \ref{prop:Fdenserange}, we know that the Fourier transform is surjective on trigonometric polynomials.
Moreover, it maps the operator product and adjoint in $\Hom(\gamma,\gamma)$ to the (commutative) $\ast$ product and $^\bullet$ involution in $\Trig(\N\subset\M)$, see equations \eqref{eq:Fismult} and \eqref{eq:Fisinvol}.
Thus positive functions $f\in\lambda(\Trig(\N\subset \M))$ correspond to positive operators $T_f\in\Hom(\gamma,\gamma)$ via the Fourier transform $\cF(T_f) = \lambda^{-1}(f)$.
If $T_f\neq 0$, by \cite[\Prop 5.4]{Pa1973}, \cite[Prop.\ A.5]{Bi2016}, it corresponds to a normal faithful and (up to rescaling unital) completely positive $\N$-bimodular map $\phi_f$ on $\M$ (dominated\footnote{See Definition \ref{def:dominateducp}. The two rescaling constants $c, d\geq 0$ such that $\phi_f(1)=c 1$ and $\phi_f \leq d E$ are related to the $L^1$ and $L^\infty$-norm of $f$, respectively, see below.} by $E$) via the formula $\phi_f := \iota(w)^* T_f \gamma(\slot) \iota(w)$. Recall from \eqref{eq:stineE} that $E = \iota(w)^*\gamma(\slot) \iota(w)$, $w\in\Hom(\id,\theta)$, is the minimal\footnote{Namely, $\gamma(\M)\iota(w)\Hil$ is dense in $\Hil$, which follows \eg from the existence of Pimsner--Popa bases in $\M$.\label{foot:minimalstine}} Connes--Stinespring representation of $E$.

To conclude the proof, we check that for every $f\in \lambda(\Trig(\N\subset \M)) \subset C(K)$, $f\geq 0$, the (not necessarily normalized) positive functional $\omega_{\phi_f}$ on $\Cred(\N\subset \M)$ given by duality coincides with the (not necessarily probability) positive measure $\omega_f$ on $K$. By linearity it suffices to check that 
\begin{align}\label{eq:fouriercommuteswithduality}
\omega_{\phi_f} = \omega_{f}
\end{align}
for all the elementary (non-positive) functions $f = \lambda(T(\psi_{\rho,r},\psi_{\rho,s}))$. In this case, $T_f = \psi_{\rho,r}^* \iota(\psibar_{\rho,s})$ thus $\phi_f := \iota(w)^* \psi_{\rho,r}^* \iota(\psibar_{\rho,s}) \gamma(\cdot) \iota(w)$. It suffices to check the validity of \eqref{eq:fouriercommuteswithduality} on elementary functions $g=\lambda(T(\psi_{\sigma,k},\psi_{\sigma,l}))$ because they are total in $\Cred(\N\subset \M)$. Compute
\begin{align}
\omega_{\phi_f}(g) &= \psi_{\sigma,k}^* \phi_f(\psi_{\sigma,l})\\
&= \psi_{\sigma,k}^* \iota(w)^* \psi_{\rho,r}^* \iota(w^* \iotabar(\psi_{\rho,l})m) \gamma(\psi_{\sigma,l}) \iota(w)\\
&= \psi_{\sigma,k}^* \psi_{\rho,r}^* \iota(\rho(w)^* w^* \iotabar(\psi_{\rho,s})m) \gamma(\psi_{\sigma,l}) \iota(w)\\ 
&= \psi_{\sigma,k}^* \psi_{\rho,r}^* \iota(w^* \iotabar(\psi_{\rho,s}) \theta(w)^* m) \gamma(\psi_{\sigma,l}) \iota(w)\\ 
&= \psi_{\sigma,k}^* \psi_{\rho,r}^* \iota(w^* \iotabar(\psi_{\rho,s})) \gamma(\psi_{\sigma,l}) \iota(w)\\ 
&= \psi_{\sigma,k}^* \psi_{\rho,r}^* E(\psi_{\rho,s}\psi_{\sigma,l})
\end{align}
using $\theta(w^*)m = \one$ and $E=\iota(w)^*\gamma(\cdot)\iota(w)$. On the other hand
\begin{align}
\omega_f(g) &= \int_K \lambda(T(\psi_{\rho,r}, \psi_{\rho,s}) \ast T(\psi_{\sigma,k}, \psi_{\sigma,l})) \dd\omega_E\\
&= \int_K \lambda(T(\psi_{\rho,r}\psi_{\sigma,k}, \psi_{\rho,s}\psi_{\sigma,l})) \dd\omega_E\\
&= \omega_E(T(\psi_{\rho,r}\psi_{\sigma,k}, \psi_{\rho,s}\psi_{\sigma,l}))\\
&= \psi_{\sigma,k}^*\psi_{\rho,r}^* E(\psi_{\rho,s}\psi_{\sigma,l}).
\end{align}
Note that the previous expressions are non-zero if and only if $\sigma \cong \rhobar$. 
We conclude that $\omega_{\phi_f} = \omega_f$ on $\Cred(\N\subset \M) \cong C(K)$, for every $f\in\lambda(\Trig(\N\subset \M))$. 

By the same computation as above and by observing that $\one = T(\psi_{\id,1},\psi_{\id,1})$ with $\psi_{\id,1} = 1$, we have that
\begin{align}
\phi_f(1) = c 1,\end{align}
with
\begin{align}
c = \int_K f \dd\omega_E = \omega_E(f).
\end{align}
Thus for $f\in\lambda(\Trig(\N\subset \M))$, $f\geq 0$, the normalization condition for the measure $\int_K f \dd\omega_E = 1$ corresponds to the unitality of the associated completely positive map $\phi_f(1) = 1$. 
By \cite[Prop.\ A.5]{Bi2016} and Lemma \ref{lem:OmegapresisOmegaadj}, we get $\phi_f \in{\UCP}^{\sharp}_\N(\M,\Omega)$, and the proof is complete.
\end{proof}

\begin{notation}
Denote by $\Extr({\UCP}^{\sharp}_\N(\M,\Omega))$ the \textbf{extreme points} in ${\UCP}^{\sharp}_\N(\M,\Omega)$.
\end{notation}

By the previous theorem we have that $\Extr({\UCP}^{\sharp}_\N(\M,\Omega)) \cong K$, where $K$ is the spectrum of $\Cred(\N\subset\M)$.

\begin{rmk}
The subset of extreme points of a compact convex set need \emph{not} be closed (thus compact), not even a Borel set in general \cite[\p 5]{Phe01}. 
This is anyway the case for the pure states on a commutative \Cstar-algebras, as the pureness condition is characterized algebraically by the multiplicativity of the state, and the latter is a closed condition in the weak* topology.
\end{rmk}

\begin{cor}\label{cor:cptmetK}
The set $\Extr({\UCP}^{\sharp}_\N(\M,\Omega))$ is compact and metrizable in the topology of $L^2(\M,\Omega)$-convergence.
\end{cor}

By Gelfand duality and because $\UCP^\sharp_\N(\M,\Omega)$ does not depend on the choice of $\Omega$ (Corollary \ref{cor:indepOmega} and \ref{cor:localNbimoducp}) we also get:

\begin{cor}
The \Cstar-algebra $\Cred(\N\subset\M)$ does not depend up to *-isomorphism on the choice of generalized Q-system of intertwiners $(\theta,w,\{m_{\rho,r}\})$ (Notation \ref{not:genQsysint}) made at the beginning of Section \ref{sec:trig}.
\end{cor}

Before elucidating the hypergroup structure of $\Extr({\UCP}^{\sharp}_\N(\M,\Omega)) \cong K$ in Section \ref{sec:UCPcompacthyp}, we deduce some general consequences of Theorem \ref{thm:ucpstatesduality} on the structure of the subfactor.

%%%
\subsection{Commutativity of $\Hom(\gamma,\gamma)$}\label{sec:gammagamma}
%%%

Denote by $\N\subset\M\subset\M_1\subset\M_2$ the (beginning of) the Jones tower and by $\gamma$ the canonical endomorphism of the subfactor (Section \ref{sec:genQsys}). In this section, we show that if $\N\subset\M$ is irreducible discrete and \emph{local} then the von Neumann algebra $\M'\cap\M_2 \cong \Hom(\gamma,\gamma) = \gamma(\M)' \cap \M$ is \emph{commutative}. In the finite index case, the result follows from Fourier duality, the Fourier transform is a bijection onto $\Hom(\theta,\theta)$ in this case, and from the commutativity of ordinary Q-systems \cite[\Def 4.20]{BiKaLoRe2014-2}. In the local discrete case, we show it by exploiting the proof of Theorem \ref{thm:ucpstatesduality} and proving the identification of $\Hom(\gamma,\gamma)$ with $L^\infty(K,\omega_E)$, where $K$ is the spectrum of $\Cred(\N\subset\M)$.

\begin{rmk}\label{rmk:fermionicgammagamma}
As we shall see in Section \ref{sec:nonlocal}, the commutativity of $\Hom(\gamma,\gamma)$ also holds for some non-local subfactors. 
Examples of this are provided by conformal inclusions: take the finite index subfactors associated with intervals in the Virasoro net $\Vir_{1/2}$ embedded into the (non-local) real Fermionic net. In this case, see \cite[\Sec 2]{MaSc1990}, we have non-local (with respect to the given braiding) $\ZZ_2$-fixed point subfactors, thus $\Hom(\gamma,\gamma) \cong L^\infty(\ZZ_2)$. 
\end{rmk}

\begin{rmk}
In the finite index case, the condition of $\Hom(\gamma,\gamma)$ commutative is at the root of the analysis of \cite{Bi2016}. It is also the starting point of another recent work on the categorifiability of fusion rules in unitary fusion categories \cite{LiPaWu2019-arxiv}. 
\end{rmk}

Denote by ${\Span}_{\CC}\{{\UCP}^{\sharp}_{\N}(\M,\Omega)\}$ the subspace generated by ${\UCP}^{\sharp}_{\N}(\M,\Omega)$ in the vector space of bounded linear operators on $\M$. Let $P(K)$, $M(K)$ be as in Section \ref{sec:compacthyp} and identify the state space of $\Cred(\N\subset\M)$ with $P(K)$. The following is analogous to the notion of domination for positive measures on $K$, seen as positive linear functionals on $C(K)$, given in Definition \ref{def:dominatedposmeas}.

\begin{defi}\label{def:dominateducp}
Let $\phi_1, \phi_2$ be completely positive maps on $\M$. We say that $\phi_1$ is \textbf{dominated} by $\phi_2$, written as $\phi_1 \leq d \phi_2$ or simply $\phi_1 \leq \phi_2$, if there exists a constant $d>0$ such that $d\phi_2-\phi_1$ is completely positive.
\end{defi}

\begin{prop}\label{prop:spancext}
The duality map of Theorem \ref{thm:ucpstatesduality}, ${\UCP}^{\sharp}_{\N}(\M,\Omega)\to P(K)$, $\phi\mapsto \omega_\phi$ extends to a linear bijection ${\Span}_{\CC}\{{\UCP}^{\sharp}_{\N}(\M,\Omega)\}\to M(K)$. 

Moreover, the map is norm and domination order preserving between completely positive maps and positive measures.
\end{prop}

\begin{proof}
Observe that the duality map preserves convex combinations. By separating the real and imaginary parts, it is enough to prove the statement for ${\Span}_{\RR}\{{\UCP}^{\sharp}_{\N}(\M,\Omega)\}$ and real Radon measures. For every $\sum_j \alpha_j \phi_j$, where $\alpha_j\in\RR$, $\phi_j\in{\UCP}^{\sharp}_{\N}(\M,\Omega)$, we show that the linear extension of the duality map $\sum_j \alpha_j \omega_{\phi_j}$ is well-defined. 
Let $\sum_j \alpha_j \phi_j = \sum_k \beta_k \psi_k$. We can assume that $\alpha_j,\beta_k\geq 0$ by moving all the summands with a negative coefficient on the other side of the equality. We can assume that the coefficients are convex, because $\phi_j(1) = \psi_k(1) = 1$ implies $\sum_j \alpha_j = \sum_k \beta_k$ and we can normalize on both sides. Thus $\sum_j \alpha_j \omega_{\phi_j} = \sum_k \beta_k \omega_{\psi_k}$. The same argument can be applied to the inverse of the duality map, thus the linear extension is bijective.

To prove the second statement, we only have to observe that if $\phi$ is $\N$-bimodular and completely positive, then $\phi(1) = c1$ for $c \geq 0$ by the irreducibility of the subfactor.
\end{proof}

From the proof of Theorem \ref{thm:ucpstatesduality}, equation \eqref{eq:fouriercommuteswithduality}, we know that there are two equivalent ways of associating to a function $f\in\lambda(\Trig(\N\subset\M))$, an operator $T_f\in\Hom(\gamma,\gamma)$. The first one is the inverse of the Fourier transform $\cF^{-1}\circ\lambda^{-1}$.
We extend the second one to the weak operator closure of $\lambda(\Trig(\N\subset\M))$ in $\B(\Hil_E)$, which is isomorphic to $L^\infty(K,\omega_E)$.

\begin{prop}\label{prop:piisnormcont}
The map defined on $f\in L^\infty(K,\omega_E)$ such that $f\geq 0$ and $\int_K f \dd\omega_E = 1$, \ie $\omega_f := f\omega_E \in P(K)$, by setting
\begin{align}
f \mapsto \omega_f = \omega_{\phi_{\omega_f}} \mapsto \phi_{\omega_f} = \iota(w)^*T_{\omega_f}\gamma(\cdot)\iota(w) \mapsto T_{\omega_f}
\end{align}
where $T_{\omega_f}$ is a positive operator in $\Hom(\gamma,\gamma)$ such that $w^* T_{\omega_f} w = 1$ and $\|T_{\omega_f}\| \leq \|f\|_\infty$,
extends to a linear bijection $\pi: L^{\infty}(K,\omega_E)\to \Hom(\gamma,\gamma)$.

Moreover, if $f_n\to f$ in the weak* topology, $f_n,f\geq 0$ and $\{f_n\}_n$ is bounded in $L^\infty$-norm, then $T_{\omega_{f_n}}\to T_{\omega_f}$ in the weak operator topology.
\end{prop}

\begin{proof}
The first part follows by convex linearity as in the proof of the previous proposition, thanks to \cite[\Prop 5.4]{Pa1973}, \cite[Prop.\ A.5]{Bi2016}. We only have to check that $\phi_{\omega_f} \leq d E$ for some $d \geq 0$. 
But $d E - \phi_{\omega_f}$ is clearly completely positive since $(d-f) \omega_E$ is a positive measure, if we let $d=\|f\|_\infty$. We also get $d^{-1}T_{\omega_f} \leq 1$ as operators, thus $\|T_{\omega_f}\| \leq d$.

To show the second statement we can assume that $f_n$ and $f$ are 
$L^1$-normalized.
By Lemma \ref{lem:OmegaadjL2cpt} and Theorem \ref{thm:ucpstatesduality}, we know that $\phi_{\omega_{f_n}} \to \phi_{\omega_{f}}$ in the BW topology, namely $\iota(w)^*T_{\omega_{f_n}}\gamma(m)\iota(w)\to \iota(w)^*T_{\omega_f}\gamma(m)\iota(w)$ in the weak operator topology for every $m\in\M$. By the density of $\gamma(\M)\iota(w)\Hil$, \cf footnote \ref{foot:minimalstine}, and by the norm boundedness of the $T_{\omega_{f_n}}$ we get the claim.
\end{proof}

\begin{prop}\label{homgammaiscomm}
The map $\pi: L^{\infty}(K,\omega_E)\to \Hom(\gamma,\gamma)$ is a surjective *-isomorphism.
In particular, $\Hom(\gamma,\gamma) = \gamma(\M)' \cap \M$ is commutative and $\pi$ is normal and isometric.
\end{prop}

\begin{proof}
By the properties of the Fourier transform, \cf the proof of Theorem \ref{thm:ucpstatesduality}, we know that $\pi$ is an injective *-homomorphism when restricted to $\lambda(\Trig(\N\subset\M) \subset L^\infty(K,\omega_E)$. 
For $f,g \in L^\infty(K,\omega_E)$, let $f=\sum_i \alpha_i f_i$ and $g=\sum_j \beta_j g_j$ be the decompositions into four summands such that $\alpha_i,\beta_j\in\CC$ and $f_i,g_j\geq 0$. By Kaplansky's density theorem, each $f_i,g_j$ can be approximated in the weak* topology by $L^\infty$-norm bounded and positive sequences $\{f_{i,n}\}_n, \{g_{j,n}\}_n\subset\lambda(\Trig(\N\subset\M))$. By the joint continuity of the multiplication on norm bounded sets and by Proposition \ref{prop:piisnormcont}, we have 
\begin{align}
\pi(fg) = \sum_{i,j} \alpha_i\beta_j \lim_{n} \pi(f_{i,n}g_{j,n}) = \sum_{i,j} \alpha_i\beta_j \lim_{n} \pi(f_{i,n})\pi(g_{j,n}) = \pi(f)\pi(g).
\end{align}
Similarly, $\pi(\bar{f})=\pi(f)^*$, and the proof is complete.
\end{proof}

%%%
\subsection{Choquet decomposition of the conditional expectation}\label{sec:choquet}
%%%

In this section, we show that the duality pairing between the convex set ${\UCP}^{\sharp}_\N(\M,\Omega)$ and the continuous functions on its extreme points given by Theorem \ref{thm:ucpstatesduality} can be cast into a simple form.

Recall that ${\UCP}^{\sharp}_\N(\M,\Omega)$ equipped with the topology of $L^2(\M,\Omega)$-convergence is compact by Lemma \ref{lem:OmegaadjL2cpt}. 
A continuous function $f\in C({\UCP}^{\sharp}_\N(\M,\Omega))$ is said to be \textbf{affine} if it preserves convex combinations $f(\lambda \phi_1 + (1-\lambda) \phi_2)=\lambda f(\phi_1)+(1-\lambda)f(\phi_2)$ for $\lambda\in [0,1]$, $\phi_1,\phi_2\in {\UCP}^{\sharp}_\N(\M,\Omega)$. 

By Theorem \ref{thm:ucpstatesduality}, $K$, the spectrum of $\Cred(\N\subset\M)$, and $\Extr({\UCP}^{\sharp}_\N(\M,\Omega))$ equipped with the induced topology are homeomorphic, thus we identify them as topological spaces.

\begin{prop}\label{prop:affine}
Let $f\in C(K)$. Then $f$ uniquely extends to an affine continuous function on ${\UCP}^{\sharp}_\N(\M,\Omega)$, denoted by $\hat f$ and given by the formula
\begin{align}
\hat f(\phi) := \omega_\phi(f) = \int_K f \dd\omega_\phi
\end{align}
for every $\phi\in{\UCP}^{\sharp}_\N(\M,\Omega)$. Moreover, $\|\hat f\|_\infty = \|f\|_\infty$.
\end{prop}

\begin{proof}
The formula for $\hat f$ extends $f$ by Gelfand duality and Theorem \ref{thm:ucpstatesduality}. 
The extension is affine because the duality map $\phi\mapsto\omega_\phi$ preserves convex combinations and continuous by Lemma \ref{lem:dualityishomeo}. In particular, $|\hat f(\phi)| \leq \|f\|_\infty$, thus $\|\hat f\|_\infty \leq \|f\|_\infty$ and the converse inequality is obvious.
\end{proof}

\begin{cor}
Let $\phi\in {\UCP}^{\sharp}_\N(\M,\Omega)$. 
Then $\omega_\phi$, seen as an element in $P(K)$, is the unique probability Radon measure on $K$, identified with $\Extr ({\UCP}^{\sharp}_\N(\M,\Omega))$, that has $\phi$ as its barycenter, \ie such that $g(\phi) = \int_K g \dd\omega_{\phi}$ for every affine function $g\in C({\UCP}^{\sharp}_\N(\M,\Omega))$. 
\end{cor}

\begin{proof}
The previous proposition implies that $\phi$ is the barycenter of $\omega_\phi$, as $\hat h = g$ if $h = g_{\restriction K}$. To check uniqueness, suppose that $\mu_\phi$ is another measure on $K$ with barycenter $\phi$. Then $\int_K g \dd\mu_\phi = g(\phi) = \int_K g \dd\omega_\phi$ for every affine function $g\in C({\UCP}^{\sharp}_\N(\M,\Omega))$. By the previous proposition, this implies that the equality holds for every $g\in C(K)$, thus $\mu_\phi=\omega_\phi$.
\end{proof}
We can now obtain a Choquet-type extremal decomposition of the unique normal faithful conditional expectation $E:\M\to\N\subset\M$.

\begin{cor}\label{cor:choquetdecomp}
Let $m\in \M$, $\phi\in{\UCP}^{\sharp}_\N(\M,\Omega)$. We have 
\begin{align}
\phi(m)=\int_K \phi'(m) \dd\omega_\phi (\phi')
\end{align}
where the integral is understood in the weak sense.
In particular, $\omega_E$ is the unique probability Radon measure supported on the extreme points of $\UCP_\N(\M,\Omega)$ with barycenter $E$, namely
\begin{align}
E(m)=\int_K \phi(m) \dd\omega_E (\phi).
\end{align}
\end{cor}

\begin{proof}
It is enough to apply the previous corollary to the affine functions in $C({\UCP}^{\sharp}_\N(\M,\Omega))$ of the form $g(\phi) := (\xi_1,\phi(m)\xi_2)$ with $\xi_1,\xi_2\in\Hil$.
\end{proof}

%%%
\subsection{The compact hypergroup of UCP maps}\label{sec:UCPcompacthyp}
%%%

We can now define the canonical compact hypergroup associated with an irreducible local discrete type $\III$ subfactor $\N\subset \M$ (Definition \ref{def:semi-discr}, \ref{def:localsubf}) mentioned at the beginning of this section. Let $\UCP^\sharp_\N(\M,\Omega)$ be as in Definition \ref{def:UCPbim} and set:

\begin{defi}\label{def:CompactHypergroupUCP}
$K(\N\subset\M) := \Extr({\UCP}^{\sharp}_{\N}(\M,\Omega))$.
\end{defi}

By Theorem \ref{thm:ucpstatesduality}, it is a compact metrizable space (Corollary \ref{cor:cptmetK}). Recall also that an element of ${\UCP}^{\sharp}_{\N}(\M,\Omega)$ is simply a ucp map $\phi:\M\to\M$ which acts trivially on $\N$ (Corollary \ref{cor:localNbimoducp}). Indeed, $\phi$ is automatically normal faithful $\Omega$-preserving and $\Omega$-adjointable for \emph{every} choice of standard vector $\Omega$ inducing an $E$-invariant state on $\M$. 

\begin{lem}\label{lem:extreme}
We have that $\Extr(\UCP_\N(\M,\Omega)) \subset \Extr(\UCP(\M,\Omega))$, where $\UCP(\M,\Omega)$ is defined in Section \ref{sec:Omegaadjmaps}.
In particular, $K(\N\subset\M) = {\UCP}^{\sharp}_\N(\M,\Omega) \cap \Extr(\UCP(\M,\Omega))$.
\end{lem}

\begin{proof}
Let $\phi\in\Extr(\UCP_\N(M,\Omega))$ and assume $\phi=\lambda\phi_1+(1-\lambda)\phi_2$ for some $\lambda\in (0,1)$ 
and $\phi_i\in\UCP(\M,\Omega)$, $i=1,2$. Then $\iota=\phi\circ\iota=\lambda\phi_1\circ\iota+(1-\lambda)\phi_2\circ\iota$, with $\iota:\N\to\M$ the inclusion morphism. We observe that $\iota$ is extreme in the convex set of completely positive maps $\N\to \M$ by \cite[\Thm 1.4.6]{Ar1969}. 
Thus we have $\iota=\phi_i\circ\iota$ and $\phi_i$ is $\N$-bimodular and therefore $\phi_i\in\UCP_\N(\M,\Omega)$, $i=1,2$. 
But since $\phi$ is extreme in  $\UCP_\N(\M,\Omega)$ we have $\phi_1=\phi_2=\phi$. The second assertion follows from Lemma \ref{lem:OmegapresisOmegaadj}.
\end{proof}

By the same argument, we have:

\begin{cor}\label{cor:autinextr}
The group of all *-automorphisms of $\M$ acting trivially on $\N$, denoted by $\Aut_\N(\M)$, is contained in $K(\N\subset\M)$.
\end{cor}

\begin{thm}\label{thm:Kishypergroup}
$K(\N\subset\M)$ is a compact hypergroup in the sense of Definition \ref{def:CompactHypergroup}, where the convolution is given by the composition of ucp maps, the involution is given by the $\Omega$-adjoint, the identity element is the identity automorphism $\id$ and the Haar measure is $\omega_E$.
\end{thm}
\begin{proof}
By Theorem \ref{thm:ucpstatesduality}, ${\UCP}^{\sharp}_{\N}(\M,\Omega)$ is identified with $P(K)$, and the extreme ucp maps with the Dirac measures, \ie $K(\N\subset\M)$ is identified with $K$. This identification allows to transport the composition of $\Omega$-adjointable ucp maps and the $\Omega$-adjoint to $P(K)$. We verify $(ii)$ and $(iii)$ of Definition \ref{def:CompactHypergroup} as the other properties are immediately checked.

For property $(ii)$, the joint continuity of the composition holds since $\phi_\alpha\to\phi$ in $L^2(\M,\Omega)$ if and only if $V_\alpha\to V_{\phi}$ in the strong (or equivalently weak) operator topology by Remark \ref{rmk:Topologies}. Indeed, $V_{\phi_1\circ\phi_2}=V_{\phi_1}V_{\phi_2}$ and 
\begin{align}
(\phi_1,\phi_2)\mapsto V_{\phi_1}V_{\phi_2}
\end{align}
is continuous since the $V_\phi$ have norm 1. Analogously, $V_{\phi^\sharp} = V_\phi^*$ and $\phi \mapsto \phi^\sharp$ is continuous.

We now show property $(iii)$ with respect to $\omega_E$, namely that
for every $f,g\in C(K)$ and $\phi^\prime\in K$ it holds
\begin{align}
\int_K f(\phi^{\prime}\ast \phi)g(\phi)\dd\omega_E(\phi) &= \int_K f(\phi)g({\phi^\prime}{}^\sharp\ast\phi)\dd\omega_E(\phi),\\
\int_K f(\phi\ast \phi^{\prime})g(\phi)\dd\omega_E(\phi) &= \int_K f(\phi)g(\phi\ast {\phi^\prime}{}^\sharp)\dd\omega_E(\phi),
\end{align}
where in this case $f(\phi^{\prime}\ast \phi) = \omega_{\phi^{\prime}\circ \phi}(f) = \hat f(\phi^{\prime}\circ \phi)$ in the notation of Proposition \ref{prop:affine}.
We show the first equation, as the second one is done similarly.
It is enough to prove it for $f = \lambda(T(\psi_{\rho,r},\psi_{\rho,s}))$ and $g = \lambda(T(\psi_{\sigma,u},\psi_{\sigma,v}))$ because they are total in $C(K)$. 
Recall that $\lambda(T(\psi_{\rho,r},\psi_{\rho,s}))(\phi) = \omega_\phi(T(\psi_{\rho,r},\psi_{\rho,s})) = \psi_{\rho,r}^*\phi(\psi_{\rho,s})$ for every $\phi\in K$.
Since
\begin{align}
\psi_{\rho,r}^*\phi^\prime\phi(\psi_{\rho,s}) = (\Omega,\psi_{\rho,r}^*\phi^\prime\phi(\psi_{\rho,s})\Omega) = ({\phi^\prime}{}^\sharp(\psi_{\rho,r})\Omega, \phi(\psi_{\rho,s})\Omega) = {\phi^\prime}{}^\sharp(\psi_{\rho,r}^*)\phi(\psi_{\rho,s})
\end{align}
we have
\begin{align}
\int_K f(\phi^\prime\ast\phi) g(\phi)\dd\omega_E(\phi) &= \int_k\psi_{\rho,r}^*\phi^\prime\phi(\psi_{\rho,s}) \psi^*_{\sigma,u}\phi(\psi_{\sigma,v})\dd\omega_E(\phi)\\
&= \omega_E(T({\phi^\prime}{}^\sharp(\psi_{\rho,r}), \psi_{\rho,s})\ast T(\psi_{\sigma,u},\psi_{\sigma,v}))\\
&= \omega_E(T({\phi^\prime}{}^\sharp(\psi_{\rho,r})\psi_{\sigma,u}, \psi_{\rho,s}\psi_{\sigma,v}))\\
&= ({\phi^\prime}{}^\sharp(\psi_{\rho,r})\psi_{\sigma,u})^* E(\psi_{\rho,s}\psi_{\sigma,v}).
\end{align}
Similarly,
\begin{align}
\int_K f(\phi) g({\phi^\prime}{}^\sharp\ast\phi) \dd\omega_E(\phi) &= \int_K \psi_{\rho,r}^*\phi(\psi_{\rho,s}) \psi_{\sigma,u}^* {\phi^\prime}{}^\sharp\phi(\psi_{\sigma,v}) \dd\omega_E(\phi)\\
&= (\psi_{\rho,r}\phi^\prime(\psi_{\sigma,u}))^* E(\psi_{\rho,s}\psi_{\sigma,v})\\
&= (\phi^\prime(\psi_{\sigma,u})\Omega, \psi_{\rho,r}^*E(\psi_{\rho,s}\psi_{\sigma,v})\Omega)\\
&= (\psi_{\sigma,u}\Omega, {\phi^\prime}{}^{\sharp}(\psi_{\rho,r}^*E(\psi_{\rho,s}\psi_{\sigma,v}))\Omega)\\
&= \psi_{\sigma,u}^*{\phi^\prime}{}^\sharp(\psi_{\rho,r}^*) E(\psi_{\rho,s}\psi_{\sigma,v})
\end{align}
which yields the claim.
\end{proof}

%%%
\section{Hypergroup actions and generalized orbifolds}\label{sec:actionK}
%%%

\begin{defi}\label{def:action}
Let $K$ be a compact hypergroup as in Definition \ref{def:CompactHypergroup} and let $\M\subset\B(\Hil)$ be a von Neumann algebra with a standard unit vector $\Omega\in\Hil$ as in Section \ref{sec:Omegaadjmaps}. An \textbf{action} of $K$ on $\M$ by $\Omega$-adjointable ucp maps is a continuous map 
\begin{align}
\alpha:K\to \Extr({\UCP}^{\sharp}(\M,\Omega))
\end{align}
where ${\UCP}^{\sharp}(\M,\Omega)$ is equipped with the pointwise weak operator topology (BW topology), such that the lift to probability Radon measures
$\tilde\alpha : P(K) \to {\UCP}^{\sharp}(\M,\Omega)$
\begin{align}\label{eq:liftaction}
(\tilde{\alpha}(\mu))(m) := \int_K (\alpha(x))(m) \dd\mu(x), \quad \mu\in P(K), m\in\M
\end{align}
where the integral is in the weak sense, is an involutive monoid homomorphism. Namely,
\begin{align}
\tilde{\alpha}(\mu_1) \circ \tilde{\alpha}(\mu_2) = \tilde{\alpha}(\mu_1 \ast \mu_2), \quad \tilde{\alpha}(\mu)^\sharp = \tilde{\alpha}(\mu^\sharp), \quad \tilde{\alpha}(\delta_e) = \id.
\end{align}
 
The action is called \textbf{faithful} if $\alpha$ is injective and it is called \textbf{minimal} if $\M^K := \{m\in\M : (\alpha(x))(m)=m \text{ for all }x\in K \}$ fulfills ${\M^K}' \cap \M = \CC1$.
\end{defi}

\begin{rmk}
Note that if $\alpha$ is continuous, the map $\tilde{\alpha}(\mu)$, $\mu\in P(K)$, defined by \eqref{eq:liftaction} is automatically cp on $\M$ because the positive cone is closed under weak integrals with respect to positive measures. Moreover, it is clearly unital and $\Omega$-preserving. It is $\Omega$-adjointable because the map $x\mapsto \alpha(x)^\sharp$ is continuous as well, thus it can be integrated.

The map $\mu \mapsto \tilde{\alpha}(\mu)$ is affine (it preserves convex combinations) and continuous from the weak* topology on $P(K)$ to the BW topology on ucp maps. It is also linear and continuous when extended to all complex Radon measures $M(K)$ via the same formula \eqref{eq:liftaction}. 
Moreover, $\|\tilde{\alpha}(\mu)\| \leq \|\mu\|$ where $\tilde{\alpha}(\mu)$ is seen as a bounded linear operator on $\M$ and $\|\mu\|$ is the total variation of $\mu\in M(K)$.
\end{rmk}

\begin{prop}\label{prop:Eforaction}
Let $\mu_K$ be the Haar measure on $K$. Then $\tilde{\alpha}(\mu_K)$ is a normal faithful conditional expectation of $\M$ onto $\M^K$. 

In particular, the inclusion $\M^K\subset\M$ is semidiscrete (Definition \ref{def:semi-discr}).
\end{prop}

\begin{proof}
Note that $\tilde{\alpha}(\mu_K)\in {\UCP}^{\sharp}(\M,\Omega)$ by definition. Clearly, $\tilde{\alpha}(\mu_K)$ is a norm 1 projection. Its image, \ie its fixed points, contain the von Neumann algebra $\M^K$. We have to show the converse inclusion, namely that $m\in\M$, $(\tilde{\alpha}(\mu_K))(m)=m$ implies $m\in\M^K$. But this is immediate since $\alpha(x) = \tilde\alpha(\delta_x)$ and
\begin{align}
(\alpha(x))(m) = (\tilde\alpha(\delta_x) \circ \tilde\alpha(\mu_K))(m) = (\tilde\alpha(\delta_x \ast \mu_K))(m) = (\tilde\alpha(\mu_K))(m) = m
\end{align}
by Proposition \ref{prop:Haarabsorbs}. 
$\tilde{\alpha}(\mu_K)$ is normal and faithful because it is $\Omega$-preserving. 
\end{proof}

The ucp maps in the range of the action are automatically $\M^K$-bimodular, \cf the proof of Lemma \ref{lem:NfixingisOmegapres}. Namely, $\alpha(x), \tilde\alpha(\mu) \in \UCP_{\M^K}^\sharp(\M,\Omega)$ for every $x\in K, \mu\in P(K)$. Moreover, $\alpha(x) \in \Extr(\UCP_{\M^K}^\sharp(\M,\Omega))$. 

\begin{prop}\label{prop:uniqueKact}
Let $\alpha$ be a faithful minimal action of $K$ on $\M$, in particular $\M^K \subset \M$ is an irreducible subfactor.
Assume in addition that it is of type $\III$ discrete\footnote{We conjecture that the discreteness of the subfactor follows from the compactness of $K$.} and local. Then $\alpha$ is a homeomorphism of $K$ onto $\Extr(\UCP_{\M^K}^\sharp(\M,\Omega)) \equiv K(\M^K\subset\M)$.

In particular, $K$ is uniquely determined by $\M^K\subset\M$ up to homeomorphism whose lift to measures preserves convolution and involution.
\end{prop}

\begin{proof}
Since $K$ is compact and $\alpha$ is continuous, the image $\alpha(K)$ is closed in $\Extr({\UCP}^{\sharp}_{\M^K}(\M,\Omega))$. If $\alpha(K)\neq\Extr({\UCP}^{\sharp}_{\M^K}(\M,\Omega))$, then there is a non-empty open set $B\subset\Extr({\UCP}^{\sharp}_{\M^K}(\M,\Omega))$ with $B\cap\alpha(K)=\emptyset$. By Proposition \ref{prop:Eforaction}, $\tilde{\alpha}(\mu_K)=E$, where $E$ is the unique conditional expectation for $\M^K\subset\M$. On the one hand, $\mu_K \circ \alpha^{-1}$ is a measure supported on $\alpha(K)$ that gives a Choquet-type extremal decomposition of $\tilde\alpha(\mu_K)$. On the other hand, $\omega_E$ is a faithful measure on $\Extr(\UCP_{\M^K}^\sharp(\M,\Omega))$ by Lemma \ref{lem:Eisfaithful} and Theorem \ref{thm:ucpstatesduality}. By the uniqueness of the Choquet decomposition of $E$, Corollary \ref{cor:choquetdecomp}, we get a contradiction. Since $\alpha$ is a continuous bijection from a compact space to a Hausdorff space, it is a homeomorphism.
\end{proof}

Let $\hat{\alpha}:M(K)\to M(K(\M^K\subset \M))$ be the pushforward on measures given by $\alpha$, namely $\mu \mapsto \hat{\alpha}(\mu) := \mu\circ \alpha^{-1}$. Let $\kappa: M(K(\M^K\subset \M))\to {\Span}_{\CC}\{{\UCP}^{\sharp}_{\M^K}(\M,\Omega)\}$ be the inverse of the linear extension of the duality map considered in Proposition \ref{prop:spancext}.

\begin{prop}\label{prop:alphatildefactorization}
In the assumptions of Proposition \ref{prop:uniqueKact}, if $K$ is identified with $K(\M^K\subset\M)$ via $\alpha$, then the lift of the action $\tilde{\alpha}$ is also identified with $\kappa$, namely $\tilde{\alpha} = \kappa \circ \hat{\alpha}$.

In particular, $\tilde{\alpha} : M(K)\to{\Span}_{\CC}\{{\UCP}^{\sharp}_{\M^K}(\M,\Omega)\}$ is a linear bijection.
\end{prop}

\begin{proof}
By linearity it suffices to check the statement for $\mu \in P(K)$, thus $\kappa \circ \hat{\alpha}(\mu)\in {\UCP}^{\sharp}_{\M^K}(\M,\Omega)$. For every $m\in \M$, compute
\begin{align}
(\kappa \circ \hat{\alpha}(\mu))(m)
= \int_{K(\M^K\subset \M)} \phi(m) \dd(\mu\circ\alpha^{-1})(\phi)
= \int_{K} (\alpha(x))(m) \dd\mu(x)
= (\tilde{\alpha}(\mu))(m)
\end{align}
by Corollary \ref{cor:choquetdecomp}.
\end{proof}

We are now ready to prove our compact hypergroup \emph{generalized orbifold} result for irreducible local discrete subfactors. 
Let $\N\subset \M$ and $\Omega$ be as in Theorem \ref{thm:ucpstatesduality}.
We first show that the composition product and the $\Omega$-adjoint involution on ucp maps are compatible with the Choquet-type extremal decomposition of Corollary \ref{cor:choquetdecomp}.

\begin{lem}\label{lem:compositionchoquet}
Let $m\in \M$, $\phi_1,\phi_2\in{\UCP}^{\sharp}_\N(\M,\Omega)$. Denote for short $K = K(\N\subset\M) \equiv \Extr(\UCP_{\N}^\sharp(\M,\Omega))$.
Then
\begin{align}\label{eq:KNinMacts1}
\phi_1\circ\phi_2(m) = \int_K \phi_1 \circ \phi(m)\dd\omega_{\phi_2}(\phi) = \int_K \phi \circ \phi_2(m)\dd\omega_{\phi_1}(\phi) 
\end{align}
and
\begin{align}\label{eq:KNinMacts2}
\phi_1^\sharp(m) = \int_K \phi^\sharp(m)\dd\omega_{\phi_1}(\phi)
\end{align}
in the weak sense.
\end{lem}

\begin{proof}
For every $m_1,m_2 \in \M$,
\begin{align}
(m_1\Omega, \phi_1\circ\phi_2(m_2)\Omega) &= (\phi_1^\sharp(m_1)\Omega,\phi_2(m_2)\Omega)\\ 
&= \int_K (\phi_1^\sharp(m_1)\Omega,\phi(m_2)\Omega) \dd\omega_{\phi_2}(\phi)\\ 
&= \int_K (m_1\Omega,\phi_1\circ\phi(m_2)\Omega) \dd\omega_{\phi_2}(\phi)
\end{align}
by Corollary \ref{cor:choquetdecomp}. Since $\Omega$ is standard and $\|\phi_1\circ\phi\| = 1$, the same holds by replacing $m_1\Omega$ with an arbitrary vector of the form $m'm\Omega$, $m\in\M$, $m'\in\M'$. Hence we have the first equality in \eqref{eq:KNinMacts1}.
Similarly, one obtains \eqref{eq:KNinMacts2}, and then the second equality in \eqref{eq:KNinMacts1}.
\end{proof}

\begin{thm}\label{thm:genorbi}
Let $\N\subset \M$ be as in Theorem \ref{thm:ucpstatesduality}.
The compact hypergroup $K(\N\subset\M)$ acts faithfully on $\M$ (Definition \ref{def:action}) and it gives $\N$ as the fixed point subalgebra, $\N = \M^{K(\N\subset\M)}$. Furthermore, it is the unique compact hypergroup, up to homeomorphism whose lift to measures preserves convolution and involution, which acts on $\M$ with these properties.
\end{thm}

\begin{proof}
By Lemma \ref{lem:extreme}, $K(\N\subset\M)$ is a subset of $\Extr({\UCP}^{\sharp}(\M,\Omega))$, thus by Corollary \ref{cor:choquetdecomp} and the previous lemma it acts on $\M$ in the sense of Definition \ref{def:action}. We only have to show that the fixed point subalgebra coincides with $\N$. By definition, $\N\subset\M^{K(\N\subset\M)}$. The converse inclusion follows either from Corollary \ref{cor:choquetdecomp} applied to $E$, or from Lemma \ref{lem:OmegaadjL2cpt} and Krein--Milman's theorem. The uniqueness follows from Proposition \ref{prop:uniqueKact}.
\end{proof}

%%%
\section{Representation theory}\label{sec:repK}
%%%

Let $K$ be a compact hypergroup as in Definition \ref{def:CompactHypergroup} and let $M(K)$ be the associated unital involutive Banach algebra of complex Radon measures on $K$ as in Section \ref{sec:compacthyp}. The following definition should be compared with \cite[\Def 2.1.1]{BlHe1995}, \cite{Je1975}, \cite{Vr1979}.

\begin{defi}\label{def:cpthyprep}
A \textbf{representation} of $K$ on a Hilbert space $\Hil_\pi$ is a unital involutive algebra homomorphism $\pi: M(K) \to \B(\Hil_\pi)$, namely $\pi(\mu * \nu) = \pi(\mu)\pi(\nu)$, $\pi(\mu^*) = \pi(\mu)^*$ for $\mu,\nu\in M(K)$ and $\pi(\delta_e) = 1_{\Hil_\pi}$. Note that $\pi$ is automatically norm decreasing, $\|\pi(\mu)\| \leq \|\mu\|$. 

A representation is called \textbf{continuous} if its restriction to positive measures is continuous from the weak* topology of $M(K)$ to the weak operator topology of $\B(\Hil_\pi)$. 
\end{defi}

Let $K(\N\subset\M)$ be the compact hypergroup associated with a subfactor $\N\subset\M$ as in Theorem \ref{thm:Kishypergroup}. 
By Proposition \ref{prop:spancext}, we identify $M(K(\N\subset\M))$ and ${\Span}_{\CC}\{\UCP^\sharp_\N(\M,\Omega)\}$ as unital involutive algebras. Indeed, $\omega_{\phi_1} * \omega_{\phi_2} = \omega_{\phi_1\circ\phi_2}$, $\omega_{\phi_1}^* = \omega_{\phi_1^\sharp}$ and $\delta_e = \omega_{\id}$ by definition.
Denote by $\mu \mapsto \phi_\mu$ the inverse of the duality map $\phi \mapsto \omega_\phi$. 
 
Let $H_\rho$ be the space of charged fields associated with a not necessarily irreducible subendomorphism $\rho\prec\theta$ of the dual canonical endomorphism $\theta$ of $\N\subset\M$. Recall from Section \ref{sec:genQsys} that $H_\rho$ is finite-dimensional 
whenever $\rho$ has finite dimension and equip it with the inner product $\psi_1^*\psi_2 =: (\psi_1,\psi_2) 1$ for $\psi_1,\psi_2\in H_\rho$.

\begin{prop}\label{prop:repcorresp}
Each $H_\rho$ is a continuous representation of $K(\N\subset\M)$, where the action of $\mu \in M(K(\N\subset\M))$ on $\psi\in H_\rho$ is defined by $\pi_\rho(\mu) \psi := \phi_\mu(\psi)$. 
\end{prop}

\begin{proof}
Multiplicativity and involutivity follow from the above discussion. We only observe that 
$(\psi_1,\pi_\rho(\mu)^*\psi_2) = \phi_\mu(\psi_1)^*\psi_2 = (\Omega,\phi_\mu(\psi_1)^*\psi_2\Omega) = (\Omega,\psi_1^*\phi_\mu^\sharp(\psi_2)\Omega) = (\psi_1,\pi_\rho(\mu^*)\psi_2)$.
To check the continuity, let $\mu_n \to \mu$ in the weak* topology, where $\mu_n, \mu$ are positive measures. By renormalizing, we can assume that $\mu_n,\mu\in P(K(\N\subset\M))$. The convergence of $(\psi_1, \pi_\rho(\mu_n)\psi_2) = \omega_{\phi_{\mu_n}}(T(\psi_1,\psi_2))$ to $(\psi_1, \pi_\rho(\mu)\psi_2) = \omega_{\phi_{\mu}}(T(\psi_1,\psi_2))$ for every $\psi_1,\psi_2\in H_\rho$ follows then by Theorem \ref{thm:ucpstatesduality}.
\end{proof}

Note that if a representation $\Hil_\pi$ has an invariant subspace $V\subset\Hil_\pi$, then the orthogonal complement $V^\perp$ is also invariant.

\begin{prop}
Let $\rho\prec\theta$ be irreducible, \ie $\Hom(\rho,\rho) = \CC1$. Then the representation $H_\rho$ is irreducible, \ie it has no non-trivial invariant subspace. Its dimension is $\dim(H_\rho) = n_\rho$, where $n_\rho$ is the multiplicity of $[\rho]$ in $[\theta]$ as in Section \ref{sec:genQsys}.

Moreover, if $\rho,\sigma\prec\theta$ are irreducible, then $H_\rho$ and $H_\sigma$ are unitarily equivalent if and only if $[\rho] = [\sigma]$.
\end{prop}

\begin{proof}
We argue by contradiction. Assume that $H_\rho = V \oplus V^\perp$ for a non-trivial invariant subspace $V\subset H_\rho$. Choose non-trivial $\psi_1\in V$, $\psi_2\in V^\perp$. On the one hand, the matrix element $T(\psi_1,\psi_2) \neq 0$ by construction, since $\Hom(\theta,\theta) \cong \bigoplus_{[\rho]} M_{n_\rho}(\CC)$, where the sum runs over all the inequivalent irreducible subendomorphisms of $\theta$, see \eqref{eq:thetathetafullmats}. On the other hand, $\omega_{\phi_\mu}(T(\psi_1,\psi_2)) = \psi_1^* \phi_\mu(\psi_2) = 0$ for every $\mu \in M(K(\N\subset\M))$. This is a contradiction since $\Trig(\N\subset\M)$ faithfully embeds into $C(K(\N\subset\M))$ via $\lambda$ by Lemma \ref{lem:lambdafaithbound}, and the states $\omega_{\phi_\mu}$ of $C(K(\N\subset\M))$, $\mu\in P(K(\N\subset\M))$, separate points by Theorem \ref{thm:ucpstatesduality}.

For the second statement, if $u\in\Hom(\rho,\sigma)$ is unitary, then $H_\sigma = u H_\rho$ and $u$ intertwines the actions since $u\in\N$ and each $\phi_\mu$ is $\N$-bimodular. Suppose now that $H_\rho$ and $H_\sigma$ are unitarily equivalent as representations of $K(\N\subset\M)$. 
Then they induce the same matrix elements which means, by the same argument as above, that for every $\psi_1,\psi_2\in H_\rho$, $T(\psi_1,\psi_2)=T(\psi_3,\psi_4)$ for some $\psi_3,\psi_4\in H_\sigma$. This is possible only if $[\rho] = [\sigma]$.
\end{proof}

If $\Hil_\pi$ is a continuous representation (Definition \ref{def:cpthyprep}) of an abstract compact hypergroup $K$ (Definition \ref{def:CompactHypergroup}), the associated representative functions $x\in K \mapsto (u,\pi(x)v)$ for $u,v\in \Hil_\pi$, where $\pi(x) := \pi(\delta_x)$, are by definition continuous.
Denote by $\Trig(K) \subset C(K)$ the complex linear span of all the continuous and irreducible representative functions of $K$.

Every continuous irreducible representation $\Hil_\pi$ of $K$ is necessarily finite-dimensional and its dimension, $\dim(\Hil_\pi)$, is bounded from above by another constant, $k_\pi$, called the \textbf{hyperdimension} of $\Hil_\pi$ \cite{Vr1979}, \cite{AmMe2014}. The hyperdimension is computed as follows, $u\in \Hil_\pi$, $\|u\| = 1$,
\begin{align}
\int_K |(u,\pi(x)u)|^2 \dd\mu_K(x) = \frac{1}{k_\pi} ,
\end{align}
where $\mu_K$ is the Haar measure. Then $\dim(\Hil_\pi) \leq k_\pi$. Equality occurs if $K$ is a compact group. Moreover, the following \emph{orthogonality relations} among the representative functions of continuous irreducible representations $\Hil_\pi$ and $\Hil_{\pi'}$ hold:
\begin{align}
\int_K \overline{(u,\pi(x)w)} (v,\pi'(x)z) \dd\mu_K(x) = \frac{1}{k_\pi} \delta_{[\pi],[\pi']}(v,u) (w,z),
\end{align}
where $[\pi]$ denotes the unitary equivalence class of $\Hil_\pi$. The proofs of these facts follow exactly as in \cite[\Thm 2.2, 2.6]{Vr1979} with our definition of abstract compact hypergroup. Alternatively, by the comments after Definition \ref{def:CompactHypergroup}, one can invoke the results on Peter--Weyl theory for compact quantum hypergroups \cite[\Sec 5]{ChVa1999}.

\begin{thm}
Every continuous irreducible representation of $K(\N\subset\M)$ is of the form $H_\rho$ for some irreducible subendomorphism $\rho\prec\theta$ of the dual canonical endomorphism.

In particular, $\Trig(K) = \Trig(\N\subset\M)$ if we let $K = K(\N\subset\M)$.
\end{thm}

\begin{proof}
By construction $\lambda(\Trig(\N\subset\M))$ is dense in $C(K)$ and it contains all the representative functions associated with $H_\rho$ for $\rho\prec\theta$ irreducible. Indeed, $\lambda(T(\psi_1,\psi_2))(x) = \omega_{\phi_{\delta_x}}(T(\psi_1,\psi_2)) = (\psi_1,\pi_\rho(x)\psi_2)$ for $\psi_1,\psi_2\in H_\rho$, $x\in K$. By the orthogonality relations, no continuous irreducible representation can be inequivalent to every $H_\rho$, unless it is zero.
\end{proof}

\begin{thm}\label{thm:hyperdim}
Let $\rho\prec\theta$ be irreducible and $K = K(\N\subset\M)$. Then the hyperdimension of the representation $H_\rho$ of $K$ equals the dimension of the endomorphism $\rho$ in $\End_0(\N)$. 
\end{thm}

\begin{proof}
For $\psi\in H_\rho$, $\|\psi\| = 1$, we have to check that
\begin{align}
\int_K |(\psi, \phi(\psi))|^2 \dd\omega_E(\phi)=\frac{1}{d(\rho)}.
\end{align}
Observe that $|(\psi, \phi(\psi))|^2=\psi^*\phi(\psi)\phi(\psi^*)\psi=\omega_\phi (T(\psi,\psi))\omega_{\phisharp} (T(\psi,\psi))$ for $\phi\in K=K(\N\subset\M)$. Moreover, one can show that $\omega_{\phisharp} (T(\psi,\psi))=\omega_{\phi} (T(\psi^{\bullet},\psi^{\bullet}))$, thus
\begin{align}
\int_K |(\psi,\phi(\psi))|^2 \dd\omega_E(\phi) &= \int_K \omega_\phi (T(\psi,\psi)) \omega_\phi (T(\psi^{\bullet},\psi^{\bullet}))\dd\omega_E(\phi)\\
&= \int_K \omega_{\phi} (T(\psi\psi^\bullet,\psi\psi^\bullet))\dd\omega_E(\phi)\\
&= \psi^{\bullet*}\psi^* E(\psi\psi^\bullet)\\
&= \psi^{\bullet*}\psi^* E(\psi\psi^*)\iota(\rbar_\rho) = \frac{1}{d(\rho)}
\end{align}
because $(\psi,\psi)1 = (\psi^\bullet,\psi^\bullet)1 = d(\rho)E(\psi\psi^*)$ by the irreducibility of $\rho$, equation \eqref{eq:twoinnerprods} and $a_\rho = 1_{H_\rho}$ (Proposition \ref{prop:arho=1}).
\end{proof}

Together with the inequality $\dim(\Hil_\pi) \leq k_\pi$ mentioned above, Theorem \ref{thm:hyperdim} provides another proof of Corollary \ref{cor:multboundnd}.

%%%
\section{The case of depth 2: No compact quantum group orbifolds}\label{sec:depthtwo}
%%%

In this section we provide an application of our construction to local conformal nets: we show that Woronowicz compact quantum groups \cite{Wo1998}, in the von Neumann algebraic setting \cite{Va2001}, \cite{KuVa2003}, see also \cite{IzLoPo1998}, \cite{To2009}, cannot induce conformal inclusions of local conformal nets by orbifold construction, \ie taking fixed points under some action, unless they are classical.
We do this by specializing Theorem \ref{thm:genorbi} to the case of depth 2 subfactors. Note that \emph{no} discreteness assumption is needed in this section, as it is a consequence of semidiscreteness for depth 2 inclusions thanks to a result of Enock and Nest \cite{EnNe1996}. We first recall the definition \cite[\Def 6.1]{EnNe1996}.

\begin{defi}\label{def:depth2subf}
An irreducible subfactor $\N\subset \M$ has \textbf{depth 2} if $\N'\cap\M_2$ is a factor, where $\N\subset\M\subset\M_1\subset\M_2$ is (the beginning of) the Jones tower. 
\end{defi}

\begin{prop}\label{prop:depth2discrete}
Let $\N\subset \M$ be an irreducible semidiscrete type $\III$ subfactor (Definition \ref{def:semi-discr}). If $\N\subset \M$ has depth 2 then it is discrete.
\end{prop}

\begin{proof}
The statement is contained in \cite[\Thm 12.2, \Prop 12.4]{EnNe1996}, after comparing the terminology \cite[\Def 9.5, \Def 11.12]{EnNe1996}.
\end{proof}

Assuming in addition that the subfactor is \emph{local} (Definition \ref{def:localsubf}), as it is for example the case for an arbitrary conformal inclusion of local conformal nets, we shall show below that the hypergroup $K = K(\N\subset\M)$ of Theorem \ref{thm:genorbi} is a compact metrizable \emph{group}. 

We first prove two lemmas. Recall that for arbitrary irreducible discrete subfactors we have the multiplicity bound $n_\rho \leq d(\rho)^2$ for every irreducible $\rho\prec\theta$, where $n_\rho = \dim(H_\rho)$ and $d(\rho)$ is the (necessarily finite) dimension of $\rho$. Moreover, $n_\rho \leq d(\rho)$ if the condition $a_\rho = 1_{H_\rho}$ is fulfilled. See Sections \ref{sec:genQsys} and \ref{sec:localsubf}.

\begin{lem}
Let $\N\subset\M$ be an irreducible semidiscrete type $\III$ subfactor. If $\N\subset\M$ has depth 2 and if it fulfills $a_\rho = 1_{H_\rho}$ for every irreducible $\rho\prec\theta$, then $n_\rho = d(\rho)$.
\end{lem}

\begin{proof}
By a result of Enock and Nest \cite[\Thm 12.6]{EnNe1996}, see also \cite[\Thm 5.1]{Va2001}, the depth 2 condition implies that $\N'\cap\M_1 (\cong \oplus_{[\rho]}M_{n_\rho}(\CC))$ has the structure of a discrete quantum group. Moreover, the condition $a_\rho = 1_{H_\rho}$ implies that the Haar weight, \ie the restriction of the dual operator-valued weight $\hat E$ to $\N'\cap\M_1$ \cite[\Sec 2]{En1998}, is tracial. This fact is well known to experts, we give below an explicit proof for fixing the notation.
Let $\{\psi_{\rho,r}\}$ be a Pimsner--Popa basis of charged fields as in Section \ref{sec:genQsys}. The elements of the form $\psi_{\rho,r}^* e_\N  \psi_{\rho,s}$ weakly span $\N'\cap\M_1$ by discreteness (Proposition \ref{prop:depth2discrete}) and they lie in the domain of $\hat E$ because $\hat E(e_\N) = 1$ by \cite[\Lem 3.1]{Ko1986}. In particular, $\hat E_{\restriction\N'\cap\M_1}$ is semifinite.
By our choice of normalization \eqref{eq:normalizfields}, by equation \eqref{eq:twoinnerprods} and by the condition $a_\rho = 1_{H_\rho}$, we have $\hat E(\psi_{\rho,r}^* e_\N  \psi_{\rho,s}) = \psi_{\rho,r}^* \psi_{\rho,s} = d(\rho) \delta_{r,s} 1$ and
\begin{align}
\hat E(\psi_{\rho,r}^* e_\N  \psi_{\rho,s} \psi_{\sigma,t}^* e_\N  \psi_{\sigma,u}) = \hat E(\psi_{\rho,r}^* E(\psi_{\rho,s} \psi_{\sigma,t}^*)e_\N \psi_{\sigma,u}) = d(\rho) \delta_{[\rho],[\sigma]} \delta_{s,t}\delta_{r,u} 1.
\end{align}
Similarly $\hat E(\psi_{\sigma,t}^* e_\N  \psi_{\sigma,u} \psi_{\rho,r}^* e_\N  \psi_{\rho,s}) = d(\rho) \delta_{[\rho],[\sigma]} \delta_{s,t}\delta_{r,u} 1$, and the trace property follows.
Thus $\N'\cap\M_1$ is a discrete Kac algebra in the sense of \cite[\Sec 6.3]{EnScBook} and by the structure theorem \cite[\Thm 6.3.5]{EnScBook} we know that $\hat E_{\restriction\N'\cap\M_1}$ is identified with $\sum_{[\rho]} n_\rho \Tr_{\rho}$, where $\Tr_{\rho}$ is the non-normalized trace on $M_{n_\rho}(\CC)$. By the computation above on matrix units $\psi_{\rho,r}^* e_\N  \psi_{\rho,s}$, we infer that $n_\rho = d(\rho)$ for every irreducible $\rho\prec\theta$.
\end{proof}

\begin{lem}\label{lem:nodefect}
Let $\N\subset\M$ be an irreducible discrete type $\III$ subfactor. If $n_\rho = d(\rho)$ and $a_\rho = 1_{H_\rho}$ for an irreducible $\rho\prec\theta$, then $\frac{1}{d(\rho)}\sum_r \psi_{\rho,r}\psi_{\rho,r}^*=1$ where $r=1,\ldots,n_\rho$.
\end{lem}

\begin{proof}
With these assumptions, by equations \eqref{eq:normalizfields} and \eqref{eq:twoinnerprods}, we have that $\frac{1}{d(\rho)}\sum_r \psi_{\rho,r}\psi_{\rho,r}^*$ is a projection in $\M$ since 
\begin{align}
\frac{1}{d(\rho)}\sum_r \psi_{\rho,r}\psi_{\rho,r}^*\frac{1}{d(\rho)}\sum_s \psi_{\rho,s}\psi_{\rho,s}^*=\frac{d(\rho)}{d(\rho)^2}\sum_r\psi_{\rho,r}\psi_{\rho,r}^*=\frac{1}{d(\rho)}\sum_r\psi_{\rho,r}\psi_{\rho,r}^*.
\end{align}
Thus
\begin{align}
E\big(1-\frac{1}{d(\rho)}\sum_r\psi_{\rho,r}\psi_{\rho,r}^*\big) = 1-\frac{n_\rho}{d(\rho)} = 0
\end{align}
and by the faithfulness of $E$ we conclude that $\frac{1}{d(\rho)}\sum_r \psi_{\rho,r}\psi_{\rho,r}^*=1$.
\end{proof}

\begin{thm}\label{thm:noquantum}
Let $\N\subset\M$ be an irreducible local semidiscrete type $\III$ subfactor with depth 2. Then $K(\N\subset\M)$ is the compact metrizable group of all *-automorphisms of $\M$ that act trivially on $\N$, in symbols $K(\N\subset\M) = \Aut_\N(\M)$.
\end{thm}

\begin{proof}
By Proposition \ref{prop:depth2discrete} and Proposition \ref{prop:arho=1} the subfactor is discrete and it has the property $a_\rho = 1_{H_\rho}$ for every irreducible $\rho\prec\theta$. In order to prove the statement, we have to show that every extreme ucp map $\phi:\M\to\M$ acting trivially on $\N$, \ie an element of $K(\N\subset\M)$ by Lemma \ref{lem:OmegapresisOmegaadj}, \ref{lem:NfixingisOmegapres} and \ref{lem:extreme}, is in fact an automorphism of $\M$.

By Theorem \ref{thm:ucpstatesduality}, the associated state $\omega_\phi$ is extreme, thus multiplicative:
\begin{align}
\omega_\phi (T(\psi_{\rho,r},\psi_{\rho,s})) \omega_\phi(T(\psi_{\sigma,t},\psi_{\sigma,u}))=\omega_\phi(T(\psi_{\rho,r}\psi_{\sigma,t},\psi_{\rho,s}\psi_{\sigma,u})).
\end{align}
The left hand side of the equation is $\psi_{\rho,r}^*\phi(\psi_{\rho,s})\psi_{\sigma,t}^*\phi(\psi_{\sigma,u}) = \psi_{\sigma,t}^*\psi_{\rho,r}^*\phi(\psi_{\rho,s})\phi(\psi_{\sigma,u})$, because $\psi_{\rho,r}^*\phi(\psi_{\rho,s})$ is a number. On the other hand we have $\psi_{\sigma,t}^*\psi_{\rho,r}^*\phi(\psi_{\rho,s}\psi_{\sigma,u})$.
By Lemma \ref{lem:nodefect}, multiplying and summing charged fields on the left of the two members of the equation we obtain $\phi(\psi_{\rho,s}\psi_{\sigma,u})=\phi(\psi_{\rho,s})\phi(\psi_{\sigma,u})$. By normality, \cf the proof of Lemma \ref{prop:braidedsubfisstrongamenab}, we conclude that $\phi$ is multiplicative on $\M$.

Now we show that $\phi$ is invertible with inverse $\phi^{-1} = \phisharp$, \ie an automorphism of $\M$.
For every $m,m'\in \M$ write $(m'\Omega,\phisharp(\phi(m))\Omega)=(\phi(m')\Omega,\phi(m)\Omega)=(\Omega,\phi(m'^*)\phi(m)\Omega)=(\Omega,\phi(m'^*m)\Omega)=(\Omega,m'^*m\Omega)$ $=(m'\Omega,m\Omega)$, thus $\phisharp(\phi(m))=m$ because $\Omega$ is cyclic and separating for $\M$. Similarly $\phi(\phisharp(m))=x$, and the proof is complete.
\end{proof}

\begin{rmk}
The condition on the subfactor being braided is not enough.
Namely, \cite[\Rmk \p 616]{Iz2001II} gives an example of a braided depth 2 subfactor with index $8$
which corresponds to a fixed point with respect to the Kac--Paljutkin's 8-dimensional Kac algebra \cite{KaPa1966}.
\end{rmk}

\begin{cor}\label{cor:braidingsymm}
If $\N\subset\M$ is as in Theorem \ref{thm:noquantum}, then the braiding $\{\varepsilon_{\rho,\sigma}\}_{\rho,\sigma\prec\theta}$ is a symmetry.
\end{cor}

\begin{proof}
One can directly show that if $\frac{1}{d(\rho)}\sum_r \psi_{\rho,r}\psi_{\rho,r}^*=1$, $\frac{1}{d(\sigma)}\sum_r \psi_{\sigma,r}\psi_{\sigma,r}^*=1$ for a pair of irreducibles $\rho,\sigma\prec\theta$, \cf Lemma \ref{lem:nodefect}, then the locality condition on the subfactor, equation \eqref{eq:commutfields}, implies that $\varepsilon_{\rho,\sigma} = \frac{1}{d(\rho) d(\sigma)} \sum_{r,s} \psi_{\rho,r} \psi_{\sigma,s} \psi_{\rho,r}^* \psi_{\sigma,s}^*$. In particular, $\varepsilon_{\rho,\sigma}^+ = \varepsilon_{\rho,\sigma}^-$. 
\end{proof}

To prove the existence of a compact group $G$ of automorphisms of $\M$ giving $\N=\M^G\subset\M$ as in Theorem \ref{thm:noquantum}, one can alternatively use the fact that $\Hom(\gamma,\gamma)$ is commutative by Proposition \ref{homgammaiscomm} and apply a result of Enock and Nest \cite[\Thm 11.16 (ii)]{EnNe1996}. By Theorem \ref{thm:genorbi}, $G \cong K(\N\subset\M)$, thus we have an alternative proof of Theorem \ref{thm:noquantum}.

%%%
\section{A remark on non-local extensions}\label{sec:nonlocal}
%%%

In this paper we have mainly restricted ourselves to studying irreducible discrete subfactors (Definition \ref{def:semi-discr}) that are local (Definition \ref{def:localsubf}), \eg those arising from discrete conformal inclusions of local conformal nets. In the non-local case, instead, we expect to obtain a compact quantum hypergroup in the sense of \cite{ChVa1999}, as opposed to a classical one (Theorem \ref{thm:Kishypergroup}). In this section, we remark that the main results of the present work hold as well when we replace the locality condition with a slightly more general \emph{graded-locality} condition, as it is the case, \eg for subfactors coming from a relatively local inclusion of a local net inside a $\ZZ_2$-graded-local \emph{Fermionic} net.

\begin{defi}\label{def:gradedlocalsubf}
Let $\N\subset\M$ be an irreducible braided discrete type $\III$ subfactor, with dual canonical endomorphism $\theta$ and braiding $\varepsilon_{\rho,\sigma}$ as in Definition \ref{def:braidedsubf}. We call $\N\subset\M$ \textbf{graded-local} if for every pair of irreducible subsectors $[\rho],[\sigma]$ of $[\theta]$ there is a number $s([\rho],\![\sigma]) \in \{\pm 1\}$, antisymmetric in its entries, \ie $s([\rho],\![\sigma]) = s([\sigma],\![\rho])^{-1}$, such that
\begin{align}
\psi_\rho\psi_\sigma = s([\rho],\![\sigma]) \iota(\varepsilon_{\sigma,\rho}^\pm) \psi_\sigma \psi_\rho
\end{align}
for all $\psi_\rho\in H_\rho, \psi_\sigma\in H_\sigma$.
\end{defi}

We have the following generalizations of Theorem \ref{thm:genorbi} and \ref{thm:noquantum} to the graded-local case.

\begin{thm}
Let $\N\subset \M$ be an irreducible discrete graded-local type $\III$ subfactor. Let $\Omega\in\Hil$ be a standard vector for $\M\subset\B(\Hil)$ such that the associated state is invariant with respect to the unique normal faithful conditional expectation $E:\M\to\N$. 

Then there is a compact metrizable hypergroup $K$ which acts faithfully on $\M$ by $\Omega$-adjointable ucp maps and which gives $\N$ as fixed point subalgebra, namely $\N=\M^K$.
\end{thm}

\begin{thm}
Let $\N\subset \M$ be an irreducible semidiscrete graded-local type $\III$ subfactor with depth 2. 
Then $K$ is a compact metrizable group acting on $\M$ such that $\N=\M^K$.
\end{thm}

\begin{rmk}
It would be tempting to replace the $\ZZ_2$-grading in Definition \ref{def:gradedlocalsubf}, for example with a $\U(1)$-grading, in order to model \emph{Abelian anyonic} extensions of local nets. The main ingredients of the proof of Theorem \ref{thm:ucpstatesduality}, from which the above results follow, would still be there. 
Namely, the condition $a_\rho = 1_{H_\rho}$ (Proposition \ref{prop:arho=1}), the commutativity of the algebra of trigonometric polynomials (Proposition \ref{prop:Tcommut}), the invertibility of the Fourier transform on trigonometric polynomials (Proposition \ref{prop:Fdenserange}), the faithfulness of the state $\omega_E$ (Lemma \ref{lem:haarfaith}), the boundedness of the GNS representation (Lemma \ref{lem:lambdafaithbound}). The only missing step is the good behaviour of the *-structure on trigonometric polynomials (Proposition \ref{prop:bulletTany}).
More specifically, in the notation of Remark \ref{rmk:lambdaprops} one obtains $\lambda_{\sigmabar,\sigma,s,s'} = s([\sigmabar],\![\sigma]) \lambda_{\sigma,\sigmabar,s',s}$. In order to use the same proof for Proposition \ref{prop:bulletTirred}, one needs then the condition $s([\sigmabar],\![\sigma]) \in \RR$. 
\end{rmk}

%%%
\section{Examples}\label{sec:ex}
%%%

%%%
\subsection{Group orbifolds}
%%%

In this section let $G$ be a compact metrizable group.
Let us consider inclusions arising as compact group orbifolds, namely as $\M^G\subset \M$ for a continuous action $\alpha\colon G\to\Aut(\M)$ on a von Neumann algebra $\M\subset\B(\Hil)$, \ie a pointwise weak operator continuous group homomorphism. Denote by $E := \int_G\alpha(g) \dd \mu_G(g)$ the normal faithful conditional expectation onto $\M^G$ given by the group average with respect to the Haar measure $\mu_G$. By definition, $\tilde{\alpha}(\mu_G) = E$, see \eqref{eq:liftaction}.

\begin{prop}\label{prop:actionGisactionK}
Let $\Omega\in\Hil$ to be a vector representing a normal faithful $E$-invariant state. Then $\alpha$ is an action of $G$ on $\M$ by $\Omega$-adjointable ucp maps in the sense of Definition \ref{def:action}.
\end{prop}

\begin{proof}
$\Aut(\M)\subset \Extr(\UCP^\sharp(\M,\Omega))$ as in Corollary \ref{cor:autinextr}. The lift to probability Radon measures $\mu\in P(G) \mapsto \tilde{\alpha}(\mu)\in \UCP^\sharp(\M,\Omega)$ defined by \eqref{eq:liftaction} is an involutive monoid homomorphism, as one can check arguing similarly to the proof of Lemma \ref{lem:compositionchoquet}, but using here that $\alpha$ is a group homomorphism.
\end{proof}

If $\alpha$ is faithful and minimal, \cf Definition \ref{def:action}, and $\M$, $\M^G$ are of type $\III$ (or infinite) factors, then $\M^G \subset \M$ is an irreducible discrete subfactor. See \eg \cite[\Sec 3]{IzLoPo1998}. Note that the type $\III$ assumption is not very restrictive, as one can always tensor with a type $\III$ factor with the trivial action of $G$.

\begin{prop}\label{prop:cptgrouporbi}
Let $\N\subset\M$ be an irreducible discrete type $\III$ subfactor, assume in addition that it is local. 
Then $\N = \M^G$ for some faithful minimal action $\alpha$ of a compact metrizable group $G$ if and only if every $\phi\in\Extr({\UCP}^{\sharp}_{\N}(\M,\Omega))$ is an automorphism.

In this case, $\alpha$ is a homeomorphism of $G$ onto $\Extr({\UCP}^{\sharp}_{\M^G}(\M,\Omega)) \equiv K(\M^G\subset\M)$. 
Moreover, $\tilde{\alpha}(P(G)) = {\UCP}^{\sharp}_{\M^G}(\M,\Omega)$ and $\tilde{\alpha}(M(G)) = {\Span}_\CC\{{\UCP}^{\sharp}_{\M^G}(\M,\Omega)\}$. 
\end{prop}

\begin{proof}
This is a consequence of Proposition \ref{prop:actionGisactionK} and of the uniqueness of the compact hypergroup action, Proposition \ref{prop:uniqueKact} and Theorem \ref{thm:genorbi}. 
\end{proof}

Not every compact group orbifold subfactor $\M^G\subset\M$ is local with respect to a given braiding on $\langle\theta\rangle_0$, where $\theta$ is the dual canonical endomorphism of $\M^G\subset\M$, \cf Remark \ref{rmk:fermionicgammagamma}. But we can show the following:

\begin{prop}\label{prop:Gsubflocal}
Consider the subfactor $\M^G\subset \M$ for a minimal action of $G$ 
and denote the dual canonical endomorphism by $\theta$.

Then there is a braiding on $\langle\theta\rangle_0$
for which $\M^G\subset \M$ is local and  
 $\langle\theta\rangle_0$ is unitarily braided equivalent
to the symmetric rigid tensor \Cstar-category $\Rep(G)$ of unitary continuous finite-dimensional representations of $G$.
\end{prop}

\begin{proof}
Denote as before by $\iota$ the inclusion morphism, by $H_\rho = \Hom(\iota,\iota\rho)$ the space of charged fields for $\rho\prec\theta$.
Consider the functor $F\colon \langle\theta\rangle_0 \to \Rep(G)$ given on finite-dimensional $\tau,\sigma \prec \theta$ and $x\in \Hom(\tau,\sigma)$ by
\begin{align}
  F(\tau) &:= \bigoplus_{\rho\prec\theta} \Hom(\rho,\tau)\otimes H_\rho\\
  F(x) &:= \bigoplus_{\rho\prec\theta}\sum_{s_j,t_i} s_j t_i^\ast \otimes (\psi\in H_\rho\mapsto \iota(s_j^\ast x t_i)\psi) 
\end{align}
where the sums run over inequivalent irreducible subendomorphisms $\rho\prec\theta$ and orthonormal bases $\{s_j\},\{t_i\}$ respectively of $\Hom(\rho,\sigma)$ and $\Hom(\rho,\tau)$. Note that $\Hom(\rho,\tau)\otimes H_\rho$ is a finite-dimensional Hilbert space with inner product $(t_1\otimes\psi_1,t_2\otimes \psi_2) 1 := \psi^\ast_1\iota(t^\ast_1t_2)\psi_2$ and it is non-zero only for finitely many $\rho$. The group acts unitarily on it by $U(g)(t\otimes \psi) := t\otimes \alpha_g(\psi)$, $g\in G$.

We get a unitary tensorator
\begin{align}
  T_{\tau,\sigma}&\colon
  F(\tau)\otimes F(\sigma)\to F(\tau\sigma)\\
  &(t_1\otimes\psi_1)\otimes (t_2\otimes\psi_2) \mapsto 
  \bigoplus_{\rho\prec\theta} \sum_{r_k} r_k\otimes \iota(r_k^\ast)(\iota(t_1)\psi_1)(\iota(t_2)\psi_2).
\end{align}
where $\{r_k\}$ is an orthonormal basis of $\Hom(\rho,\tau\sigma)$.
Let $\tau,\sigma \prec \theta$ be irreducible and note that $\tau = \sum_{i}\Ad\psi_{\tau,i}$ and $\sigma = \sum_{j}\Ad\psi_{\sigma,j}$
for appropriate orthonormal bases of $H_\tau$ and $H_\sigma$, respectively, with respect to the inner product $\psi_1^*\psi_2 =: (\psi_1,\psi_2) 1$ (\cf \cite[\Eq (2.16)]{Re1994coset} and the proof of Corollary \ref{cor:braidingsymm}).
We get that  the formula (\cf \cite[\Eq (3.8)]{Re1994coset}) 
\begin{align}\label{eq:braidingdef}
  \varepsilon_{\tau,\sigma} 
  := \sum_{i,j} \psi_{\sigma,j}\psi_{\tau,i}\psi^\ast_{\sigma,j}\psi^\ast_{\tau,i}\in \Hom(\tau\sigma,\sigma\tau)
\end{align}
defines by natural extension a unitary braiding which makes the functor $F$ a unitary braided tensor functor and gives $\langle \theta\rangle_0$ the structure of a braided (in fact symmetric) rigid tensor \Cstar-category equivalent (via $F$) to $\Rep(G)$.
Namely, we only have to check that the diagram
\begin{equation}
\begin{tikzcd}
F(\tau)\otimes F(\sigma) \arrow[rr, "{c_{F(\tau),F(\sigma)}}"] \arrow[d, "{T_{\tau,\sigma}}"] &  & F(\sigma)\otimes F(\tau) \arrow[d, "{T_{\sigma,\tau}}"] \\
F(\tau\sigma) \arrow[rr, "{F(\varepsilon_{\sigma,\tau})}"]                                    &  & F(\sigma\tau)                                          
\end{tikzcd}
\end{equation}
commutes, where $c_{X,Y}\colon X\otimes Y\to Y\otimes X$ is the canonical flip $c_{X,Y}(x\otimes y)=y\otimes x$ in $\Rep(G)$.
In particular, with this braiding the irreducible discrete subfactor $\M^G\subset \M$ becomes local.
\end{proof}

\begin{cor}
  Let $\N\subset \M$ be an irreducible semidiscrete type $\III$ subfactor with depth 2
  and denote the dual canonical endomorphism by $\theta$.
  Then there is a braiding on $\langle \theta\rangle_0$ for which 
  $\N\subset \M$ is local if and only if $\M=\N^G$ for a minimal action of a compact metrizable group $G$.

  Furthermore, in this case the braiding is unique.
\end{cor}
\begin{proof} 
  The first statement follows from Proposition \ref{prop:Gsubflocal} and Theorem \ref{thm:noquantum}. 
  If $\N\subset \M$  is local the proof of Corollary \ref{cor:braidingsymm} 
  shows that the braiding already
 coincides with the one given in the proof of Proposition \ref{prop:Gsubflocal}.
\end{proof}

%%%
\subsection{Double coset orbifolds}\label{sec:doublecosets}
%%%

Let us consider a closed subgroup $H\subset G$ of a compact metrizable group $G$ acting on $\M$ as in the previous section. 
Then we can consider the intermediate group-subgroup inclusion $\M^G\subset \M^H$. The goal of this section is to compute $K(\M^G\subset \M^H)$ when the inclusion is irreducible and local. 

Denote by $G\CS H$ or $H\backslash G/H$ 
the set of $H$-double cosets $HxH = \{y_1xy_2, y_1,y_2\in H\}$ for $x\in G$ and denote by $P:G\to G\CS H$ the projection map. Then $G\CS H$ is a compact Hausdorff space equipped with the finest topology which makes $P$ continuous. We want to endow $G\CS H$ with a compact hypergroup structure in the sense of Definition \ref{def:CompactHypergroup}. Denote the pushforward of $P$ to complex Radon measures by $\tilde{P}:M(G)\to M(G\CS H)$
\begin{align}
\tilde{P}(\mu)(f) := \mu(f\circ P),\quad \mu\in M(G), f\in C(G\CS H).
\end{align}
Let $Q:C(G)\to C(G\CS H)$ be
\begin{align}
Q(f)(HxH) := \int_{H\times H} f(y_1 x y_2)\dd\mu_H(y_1)\dd\mu_H(y_2), \quad f\in C(G), x\in G 
\end{align}
where $\mu_H$ is the Haar measure on $H$, and the pullback $\tilde{Q}:M(G\CS H)\to M(G)$
\begin{align}
\tilde{Q}(\mu)(f) := \mu(Q(f)), \quad \mu\in M(G\CS H), f\in C(G).
\end{align}

\begin{lem}\label{lem:measuresoncosets}
We have the following:
\begin{enumerate}
\item $Q(f\circ P) = f$ for all $f\in C(G\CS H)$;
\item $\tilde{P}\circ\tilde{Q} = \id_{M(G\CS H)}$;
\item $\tilde{Q}\circ\tilde{P} = \mu_H\ast\cdot\ast\mu_H$, where $\mu_H$ is seen as an element in $M(G)$;
\item $\mu_H\ast\tilde{Q}(\mu) = \tilde{Q}(\mu)\ast\mu_H = \tilde{Q}(\mu)$ for all $\mu\in M(G\CS H)$;
\item $\tilde{Q}(M(G\CS H)) = \{\mu_H\ast \mu\ast\mu_H: \mu\in M(G)\}$.
\end{enumerate}
\end{lem}

\begin{proof}
The first three properties are immediately checked. Property (4) follows from $\tilde{Q}(\mu)=\tilde{Q}\circ\tilde{P}\circ\tilde{Q}(\mu)=\mu_H\ast \tilde{Q}(\mu)\ast \mu_H$ and $\mu_H\ast\mu_H=\mu_H.$ Property (5) follows from (3) and (4).
\end{proof}

The convolution on $M(G\CS H)$ is defined so that $\tilde{Q}$ preserves it, namely
\begin{align}
\mu_1\ast\mu_2:=\tilde{P}(\tilde{Q}(\mu_1)\ast\tilde{Q}(\mu_2)).
\end{align}
Indeed, $\tilde{Q}(\mu_1\ast\mu_2)=\tilde{Q}(\tilde{P}(\tilde{Q}(\mu_1)\ast\tilde{Q}(\mu_2))) = \mu_H\ast(\tilde{Q}(\mu_1)\ast\tilde{Q}(\mu_2))\ast\mu_H = \tilde{Q}(\mu_1)\ast\tilde{Q}(\mu_2)$.
On Dirac measures $\delta_{HxH},\delta_{HyH}$, $x,y\in G$, the convolution reads
\begin{align}\label{eq:convdiraccoset}
\delta_{HxH}\ast\delta_{HyH} = \int_{H}\delta_{HxzyH}\dd\mu_H(z).
\end{align}
Indeed, $\tilde{Q}(\delta_{HxH}) = \mu_H\ast\delta_x\ast\mu_H$ and 
$\delta_{HxH}\ast\delta_{HyH} = \tilde{P}(\mu_H\ast\delta_x\ast\mu_H\ast\delta_y\ast\mu_H) = \tilde{P}\circ\tilde{Q}\circ\tilde{P}(\delta_x\ast\mu_H\ast\delta_y) = \tilde{P}(\delta_x\ast\mu_H\ast\delta_y) = \int_{H}\delta_{HxzyH}\dd\mu_H(z)$.

Similarly the adjoint on $M(G\CS H)$ is defined so that $\tilde{Q}$ preserves it, namely
\begin{align}
\mu^{\ast} := \tilde{P}(\tilde{Q}(\mu)^{\ast}).
\end{align}
Indeed, $\tilde{Q}(\mu^\ast) = \tilde{Q}(\tilde{P}(\tilde{Q}(\mu)^\ast)) = \mu_H\ast \tilde{Q}(\mu)^\ast \ast\mu_H = (\mu_H\ast \tilde{Q}(\mu)\ast \mu_H)^\ast = \tilde{Q}(\mu)^\ast$.
On Dirac measures it reads
\begin{align}
\delta_{HxH}^\ast = \tilde{P}((\mu_H\ast\delta_x\ast\mu_H)^\ast) = \tilde{P}(\mu_H\ast\delta_{x^{-1}}\ast\mu_H) = \delta_{Hx^{-1}H}.
\end{align}

The space $G\CS H$ equipped with this convolution and adjoint, and with identity element $HeH$, is a compact DJS-hypergroup \cite[\Thm 1.1.9]{BlHe1995}. \Cf Remark \ref{rmk:DJS}. It is also a compact hypergroup in the sense of Definition \ref{def:CompactHypergroup} with Haar measure
\begin{align}\label{eq:haarcosets}
\mu_{G\CS H} := \tilde{P}(\mu_G).
\end{align}

If the action $\alpha$ of $G$ on $\M$ is faithful and minimal, then $\M^G\subset\M^H$ is an irreducible semidiscrete subfactor with normal faithful conditional expectation given by $E_H^G := E_{\restriction \M^H}$, where $E^G := E$. 
Observe that $\Omega\in\Hil$ induces an $E_H^G$-invariant state and it is standard for $\M^H$ on the subspace $\overline{\M^H\Omega}$.

If $\M$ and $\M^G$ are type $\III$, then $\M^H$ is also type $\III$.
By \cite[\Thm 2.7]{To2009}, $\M^G\subset\M^H$ is discrete and its dual canonical endomorphisms $\theta\in\End(\M^G)$ denoted by $\theta_{\M^G\subset\M^H}$ is a subendomorphism of $\theta_{\M^G\subset\M}$. 
In particular, $\M^G\subset\M^H$ is local whenever $\M^G\subset\M$ is local, and this is always the case for some choice of braiding on $\langle \theta_{\M^G\subset\M}\rangle_0$ by Proposition \ref{prop:Gsubflocal}. 

\begin{cor}\label{cor:GHsubfactorlocal}
There exists a braiding on $\langle\theta_{\M^G\subset\M^H}\rangle_0$ such that the subfactor $\M^G\subset \M^H$ is local.
\end{cor}

We show below that $K(\M^G\subset\M^H)$ can be identified with the double coset hypergroup $G\CS H$.

\begin{prop}\label{prop:actioncosets}
Let $\alpha$ be faithful and minimal, and let $\M$, $\M^G$ be of type $\III$, assume in addition that $\M^G \subset\M^H$ is local. Then the map 
\begin{align}
\tilde{\beta} := (\tilde{\alpha}\circ\tilde{Q})_{\restriction \M^H} : M(G\CS H) \to {\Span}_{\CC}\{{\UCP}^{\sharp}_{\M^G}(\M^H,\Omega)\}
\end{align}
with $\tilde{\alpha}$ defined in \eqref{eq:liftaction}, is a unital involutive algebra isomorphism. Its restriction to probability measures $\tilde{\beta}_{\restriction} : P(G\CS H) \to {\UCP}^{\sharp}_{\M^G}(\M^H,\Omega)$ is bicontinuous with respect to the weak* topology and the BW topology. Moreover, $\tilde{\beta}(\mu_{G\CS H}) = E_H^G$.

In particular, $\{\tilde{\beta}(\delta_{HxH}): HxH\in G\CS H\} = \Extr({\UCP}^{\sharp}_{\M^G}(\M^H,\Omega)) \equiv K(\M^G\subset\M^H)$.
\end{prop}

\begin{proof}
The maps $\tilde{\alpha}$ and $\tilde{Q}$ are unital involutive algebra homomorphisms, so is $\tilde{\beta}$. By Lemma \ref{lem:measuresoncosets}, $\tilde{Q}$ is injective onto $\{\mu_H\ast \mu\ast\mu_H:\mu\in M(G)\}$, and $\tilde{\alpha}$ is injective by Proposition \ref{prop:alphatildefactorization}. Thus $\tilde{\alpha}\circ\tilde{Q}$ is injective onto ${\Span}_{\CC}\{E^H\circ \phi\circ E^H: \phi\in{\UCP}^{\sharp}_{\M^G}(\M,\Omega)\}$, where $E^H := \tilde{\alpha}(\mu_H)$ is the normal faithful conditional expectation $\M \to \M^H \subset\M$. Now $\tilde{\beta}$ is bijective because the map $\phi\mapsto \phi\circ E^H$ from ${\UCP}^{\sharp}_{\M^G}(\M^H,\Omega)$ to $\{E^H\circ \phi\circ E^H: \phi\in{\UCP}^{\sharp}_{\M^G}(\M,\Omega)\}$ is bijective with inverse $E^H \circ\phi\circ E^H\mapsto E^H \circ\phi\circ E^H|_{\M^H}$.
The bicontinuity of $\tilde{\beta}$ follows because $\tilde{Q}$ and $\tilde{\alpha}$ are continuous, $P(G\CS H)$ is compact and ${\UCP}^{\sharp}_{\M^G}(\M^H,\Omega)$ is Hausdorff.
By Lemma \ref{lem:measuresoncosets}, we have $\tilde{\beta}(\mu_{G\CS H}) = \tilde{\alpha}(\tilde{Q} \circ \tilde{P}(\mu_G))_{\restriction \M^H} = \tilde{\alpha}(\mu_H \ast \mu_G \ast \mu_H)_{\restriction \M^H} = \tilde{\alpha}(\mu_G)_{\restriction \M^H} = E_H^G$.
\end{proof}

Let $\beta: G\CS H\to K(\M^G\subset \M^H)$ be the composition of the homeomorphism $\tilde{\beta}$ with the identification between $G\CS H$ and the Dirac measures $HxH\mapsto \delta_{HxH}$ as topological spaces. 

\begin{prop}\label{prop:doublecosets}
$\beta$ is an action of $G\CS H$ on $\M^H$ by $\Omega$-adjointable ucp maps in the sense of Definition \ref{def:action} and
\begin{align}
\M^G=(\M^H)^{G\CS H}.
\end{align}
\end{prop}

\begin{proof}
The first statement follows because the unique (affine and continuous in the weak* and BW topologies) lift of $\beta$ to $P(G\CS H)$ given by \eqref{eq:liftaction} coincides with $\tilde{\beta}$, hence it is an involutive monoid homomorphism by Proposition \ref{prop:actioncosets}. The second statement follows from the first by Theorem \ref{thm:genorbi}.
\end{proof}

%%%
\subsection{Quantum channels with infinite index}
%%%

Let us  consider an irreducible local discrete type $\III$ subfactor $\N\subset \M$. Denote by $\iota\colon \N\to\M$ the inclusion morphism.
Recall that a ucp map $\phi\colon \M\to \M$ fulfills $\phi\in K(\N\subset \M)$ (Definition \ref{def:CompactHypergroupUCP}) if and only if it is $\N$-bimodular and extreme by Corollary \ref{cor:localNbimoducp} and Lemma \ref{lem:extreme}. 

Longo defined in \cite{Lo2018} the \textbf{index} $\Ind(\phi)$ of $\phi$ to be the minimal index $[\M:\beta(\M)]\in[1,\infty]$ of the subfactor $\beta(\M)\subset \M$ where $\phi=v^*\beta(\slot)v$ is the minimal Connes--Stinespring representation, $v\in\M$, $v^*v = 1$ and $\beta\in\End(\M)$. Let us write $\dim(\phi):=\sqrt{\Ind(\phi)}$. Then $\dim(\phi) = d(\beta)$ by \cite{Lo1989}, \cite{LoRo1997}.
It is natural to ask in our setting whether $\Ind(\phi)$ is finite or infinite. If $\Ind(\phi)<\infty$, Longo calls $\phi$ a \textbf{quantum channel} \cite[\Sec 3]{Lo2018}. From \cite{Bi2016} it follows that for $[\M:\N]<\infty$ we have 
\begin{align}
\sum_{\phi\in K(\N\subset\M)}\dim(\phi)=[\M:\N].
\end{align}

\begin{lem}\label{lem:isometvandbeta}
Let $\phi\colon \M\to\M$ be a normal faithful ucp map, let $\phi=v^*\beta(\slot)v$ as above. Then $\phi$ is $\N$-bimodular if and only if $v\in\Hom(\iota,\beta\iota)$. In this case, $\phi$ is extreme if and only if $\beta$ is irreducible.
\end{lem}

\begin{proof}
Let $\phi$ be $\N$-bimodular. For ever $m\in\M, n\in \N$ and $\xi\in\Hil$ we have $v^*\beta(\iota(n))\beta(m)v\xi = v^*\beta(\iota(n)m)v\xi = \phi(\iota(n)m)\xi = \iota(n)\phi(m)\xi=\iota(n)v^* \beta(m)v\xi$ because $\phi$ is $\N$-bimodular. By minimality, \ie $\beta(\M)v\Hil$ is dense in $\Hil$, we conclude $v^*\beta(\iota(n))=\iota(n)v^*$. The converse implication is immediate.
The second statement follows because $\iota$ is irreducible and by using Lemma \ref{lem:extreme} and \cite[\Prop A.5]{Bi2016}.
\end{proof}

Recall the notion of domination for ucp maps (Definition \ref{def:dominateducp}).

\begin{lem}
Let $\phi,\phi'\in K(\N\subset\M)$. If $\id \leq \phi' \circ \phi$ and $\id \leq \phi \circ \phi'$, then $\Ind(\phi) = \Ind(\phi') < \infty$.
\end{lem}

\begin{proof}
Let $\phi=v^*\beta(\slot)v$ and $\phi'={v'}^*\beta'(\slot)v'$ 
be the respective minimal Connes--Stinespring representations.
Then \cite[Proposition 2.9]{IzLoPo1998} implies $\id \prec \beta\beta'$ and $\id \prec\beta'\beta$. This implies by \cite[\Thm 4.1]{Lo1990} that $\beta'$ is a conjugate of $\beta$
in the sense of \eqref{eq:conjeqns}, in symbols $\beta' \cong \bar\beta$. Thus $d(\beta) = d(\beta') < \infty$ and $\Ind(\phi) = \Ind(\phi') = d(\beta)^2 < \infty$.
\end{proof}

We show below in Lemma \ref{lem:betasharp} and Proposition \ref{prop:phisharpphicontainsatomid} that the converse of the previous lemma holds. Consider the set
\begin{align}
K_0(\N\subset\M) := \{\phi\in K(\N\subset\M): \Ind(\phi) < \infty\}.
\end{align}

Note that $K_0(\N\subset\M)$ is either finite or non-discrete with the induced topology from $K(\N\subset\M)$, because the latter is compact metrizable hence every infinite subset has accumulation points. 

\begin{lem}\label{lem:betasharp}
Let $\phi \in K(\N\subset\M)$ with adjoint $\phi^\sharp \in K(\N\subset\M)$ (Section \ref{sec:Omegaadjmaps}). If $\phi =v^*\beta(\slot)v$ and $\phi^\sharp =v^{\sharp*}\beta^\sharp(\slot)v^\sharp$ are the respective minimal Connes--Stinespring representations, then $\beta^\sharp \cong \bar\beta$. 

In particular, $\Ind(\phi) = \Ind(\phi^\sharp)$ and $\Ind(\phi) < \infty$ if and only if $\Ind(\phi^\sharp) < \infty$.
\end{lem}

\begin{proof}
This is a consequence of the general theory of bimodules associated with ucp maps \cite[\App V.B]{Co1994}, \cite{Po1986}, see also \cite[\App A.2]{Bi2016} and \cite[\Sec 2.2]{Lo2018}. We only have to check that when the adjoint $\phi^\sharp$ exists in the sense of Section \ref{sec:Omegaadjmaps}, it coincides with the transpose of a ucp map as defined in \cite[\Prop 2.14]{Lo2018}, \cite[\Prop 8.3]{OhPe1993}. Indeed, let $m_1, m_2\in\M$ and $J = J_{\M,\Omega}$ the modular conjugation of $(\M,\Omega)$ as in Section \ref{sec:Omegaadjmaps}, then $(\Omega, \phi(m_1)Jm_2J\Omega) = (\phi(m_1^*) \Omega, Jm_2\Omega) = (V_\phi m_1^* \Omega, Jm_2\Omega) = (\Omega, m_1 V^*_\phi Jm_2\Omega)$ thus 
\begin{align}
(\Omega, \phi(m_1)Jm_2J\Omega) = (\Omega, m_1J\phi^\sharp(m_2)J\Omega)
\end{align}
because $V^*_{\phi} = V_{\phi^\sharp}$ and $V_\phi J = J V_\phi$ by Proposition \ref{prop:adjointiffmodular}, which characterizes the adjointability of $\phi$. 
The claim now follows from \cite[\Prop 2.15]{Lo2018}.
\end{proof}

\begin{prop} \label{prop:phisharpphicontainsatomid}
Let $\phi\in K(\N\subset\M)$ and $\phi=v^* \beta(\slot)v$ as above. We have
\begin{enumerate}
\item If $\Ind(\phi)<\infty$, let $r\in \Hom(\id,\bar\beta \beta)$, $\rbar\in \Hom(\id,\beta \bar\beta)$ be a standard solution of the conjugate equations \eqref{eq:conjeqns} for $\beta$ and $\bar\beta$. Then $r^\ast\bar\beta(vv^\ast)r =: b_\phi1$ where $ b_\phi > 0$ is a number which depends only on $\phi$, and we have
\begin{align}
\phi^\sharp = b_\phi^{-1}\, r^*\bar\beta(v\slot v^*) r.
\end{align}
\item In general,
\begin{align}
  \phi^\sharp\circ\phi -\frac{b_\phi}{\dim(\phi)}\id\text{ is completely positive.}
\end{align}
\end{enumerate}
\end{prop}

\begin{proof}
If $\Ind(\phi)=\infty$ the statement in (2) is trivially satisfied. 
Assume $\Ind(\phi)<\infty$ and observe that $\bar\beta(v^*) r \in \Hom(\iota,\bar\beta\iota)$ because $v\in\Hom(\iota,\beta\iota)$ by Lemma \ref{lem:isometvandbeta}. Thus $r^*\bar\beta(vv^*) r = b 1$, $b\geq 0$, because $\Hom(\iota,\iota) = \CC 1$, and $b > 0$ because $\beta((\bar\beta(v^*) r)^*)\rbar = v$. 
The number $b$ does not depend on the choice of $\beta$, $v$ or $r$, $\rbar$ by the uniqueness of the minimal Connes--Stinespring representation, see \eg \cite[\Thm 2.10]{Lo2018}, and by the uniqueness and trace property of the standard left inverse of $\beta$, see \eg \cite[\Prop 2.4]{BiKaLoRe2014}.
Set $\phi' := b^{-1}r^*\bar\beta(v\slot v^*) r$. Then $\phi' \in K(\N\subset\M)$ again by Lemma \ref{lem:isometvandbeta} and Corollary \ref{cor:localNbimoducp}.
Note that  $p=\frac{1}{d(\beta)}r r^*$ and $p^\perp=1-p$ are orthogonal projections in $\Hom(\bar\beta\beta,\bar\beta\beta)$. Then 
\begin{align}
\phi'\circ\phi (\slot)&=  b^{-1}r^*\bar\beta(vv^* \beta(\slot)v v^*) r\\
&=b^{-1} r^*\bar\beta(vv^*) p\bar\beta(\beta(\slot))\bar\beta(v v^*) r + b^{-1} r^*\bar\beta(vv^*) p^\perp\bar\beta(\beta(\slot))\bar\beta(v v^*) r\\
&=\frac{b}{d(\beta)}\id +\, b^{-1}r^*\bar\beta(vv^*) p^\perp\bar\beta(\beta(\slot))p^\perp \bar\beta(v v^*) r.
\end{align}
By Theorem \ref{thm:Kishypergroup}, $K(\N\subset \M)$ is a compact (metrizable) hypergroup in the sense of Definition \ref{def:CompactHypergroup}
and $\delta_{\phi'} \ast \delta_{\phi} \geq \delta_{\id}$. Thus Proposition \ref{prop:DominatedImpliesConjugate}
implies $\phi'=\phi^\sharp$, and the statements (1) and (2) follow. 
\end{proof}
We conjecture that $b_\phi=1$ for any $\phi \in K_0(\N\subset \M)$, as it holds whenever $\phi$ is an automorphism or $[\M:\N] < \infty$. In the first case, $b_\phi = \dim(\phi) = 1$. In the second case, $K_0(\N\subset \M) = K(\N\subset \M)$ is a finite hypergroup, see Example \ref{ex:finitehyp}, and $b_\phi=1$ follows, as a computation using the trace property of the standard left inverses of $\iota$ and $\beta$ shows \cite[\Lem 4.7]{Bi2016}.

Note that for $\phi_i\in K(\N\subset\M)$, $i=1,2,3$, we have 
\begin{align}
\delta_{\phi_1}\ast\delta_{\phi_2}(\{\phi_3\}) = \sup \left\{\lambda\in[0,1] : \phi_1\circ\phi_2-\lambda\phi_3\text{ is completely positive}\right\}
\end{align}
and thus the following proposition.

\begin{prop}\label{prop:dimboundsweight}
For $\phi \in K(\N\subset\M)$, we have 
\begin{align}\label{eq:Ind}
  \dim(\phi) &\geq\frac{ b_\phi}
  {\delta_{\phi^\sharp}\ast\delta_{\phi}(\{\id\})}.
\end{align}
In particular, if $\delta_{\phi^\sharp}\ast\delta_{\phi}(\{\id\})=0$, \ie if the measure $\delta_{\phi^\sharp}\ast\delta_{\phi}$ does not contain $\id$ as an atom, then $\Ind(\phi)=\infty$.
\end{prop}

If $K$ is a hypergroup and $x\in K$, the constant $w_x := (\delta_{x^\sharp}\ast\delta_{x}(\{e\}))^{-1}\in[1,\infty]$ is called the \textbf{weight} of the element $x$. We conjecture that in \eqref{eq:Ind} we always have equality, which together with the conjecture $b_\phi=1$ would imply
that $w_\phi=\dim(\phi)$. 

In the special case that $[\M:\N]<\infty$ this easily follows from \cite[\Prop 4.4]{Bi2016}.
It also holds if $w_\phi=1$ or equivalently if $\phi$ is an automorphism.

\begin{example}\label{ex:SO3SO2}
Consider 
\begin{align}
\SO(2)\cong\left\{\begin{bmatrix} A &0\\0&1\end{bmatrix} :
A\in\SO(2)\right\}\subset \SO(3)
\end{align}
and assume that $\SO(3)$ acts faithfully and minimally on a type $\III$ factor $\M$. 
As in Section \ref{sec:doublecosets}, let us consider the irreducible discrete group-subgroup subfactor 
\begin{align}
\M^{\SO(3)}\subset \M^{\SO(2)}
\end{align}
and assume it is local, see Section \ref{sec:localsubfCFT} for an example of this arising in CFT. By Proposition \ref{prop:actioncosets}, $K=K(\M^{\SO(3)}\subset \M^{\SO(2)})$ can be identified with the double coset hypergroup $\SO(3)\CS\SO(2)$ which is homeomorphic to the closed interval $[-1,1]$ via the map $B\in\SO(3)\mapsto B_{3,3}$, see \cite[\Sec 1.1.17]{BlHe1995}.
Let us denote $K=\{\phi_t\}_{t\in[-1,1]}$ via this identification.
Then on the one hand, it easily follows that $\Ind(\phi_{\pm1})=1$ and that the subhypergroup $\{\phi_{\pm 1}\}$ is isomorphic to the group $\ZZ_2$.
On the other hand, $\Ind(\phi_t)=\infty$ for every $t\in(-1,1)$. Indeed, for $s,t\in(-1,1)$ the convolution of Dirac measures $\delta_{\phi_s}\ast\delta_{\phi_t}$ on $K$ can be computed as follows.
Let
\begin{align}
  B(s)&:=\begin{bmatrix} 1&0&0\\0& s& \sqrt{1-s^2}\\0&-\sqrt{1-s^2}&s
  \end{bmatrix}\in\SO(3)\,,&
  A(\theta)&:=
  \begin{bmatrix} \cos \theta&\sin\theta &0\\-\sin\theta &\cos \theta &0
  \\0&0&1\end{bmatrix}\in\SO(2)
\end{align}
then via the previous identification
\begin{align}
  B(s)\,A(\theta)\,B(t) \mapsto st-\sqrt{(1-s^2)(1-t^2)}\cos\theta.
\end{align}
Let $E\subset [-1,1]$ be a Borel set with characteristic function $\chi_{E}$, then \eqref{eq:convdiraccoset} reads as
\begin{align}
  (\delta_{\phi_s}\ast\delta_{\phi_t})(E) &= \int\limits_{\SO(2)} \chi_E(B(s)\,A\,B(t)) \dd\mu_{\SO(2)}(A)
  = \int\limits_0^{2\pi} \chi_{E}(st-\sqrt{(1-s^2)(1-t^2)}\cos\theta)\frac{\dd\theta}{2\pi}\\
  &= \int\limits_{st-\sqrt{(1-s^2)(1-t^2)}}^{st+\sqrt{(1-s^2)(1-t^2)}} \frac{\chi_{E}(r)}{\pi \sqrt{(1-s^2)(1-t^2) -(st-r)^2}}\dd r
\end{align}
and we see that $K$ is commutative, in particular $(\SO(3), \SO(2))$ is a \emph{Gelfand pair} \cite[\Ch 2.2]{BeKa1998}. The computation also shows that $1$ is in the support of the measure $\delta_{\phi_s}\ast\delta_{\phi_t}$ if and only if $s=t$, thus $\phi_s^\sharp=\phi_s$. Furthermore, we see that $\delta_{\phi_t}\ast\delta_{\phi_t}$ is absolutely continuous with respect to the Lebesgue measure on $[-1,1]$ for every $t\in (-1,1)$. Thus $\Ind(\phi_t)=\infty$ by Proposition \ref{prop:dimboundsweight}, and $K_0(\N\subset\M) = \{\phi_{\pm 1}\} \cong \ZZ_2$ in this example. 

The Haar measure on $K$ given by \eqref{eq:haarcosets} thanks to Proposition \ref{prop:actioncosets}, see \eg \cite[\Sec 1.6]{Na1964} for the Haar measure on $\SO(3)$, happens to coincide with the normalized Lebesgue measure on $[-1,1]$, namely for $E\subset [-1,1]$
\begin{align}
  \mu_K(E) &= \int\limits_0^{\pi} \frac{1}{2}\sin\theta\, \chi_E(\cos\theta)\dd\theta = \frac{1}{2}\int\limits_E\dd r.
\end{align}
\end{example}

%%%
\subsection{Local discrete subfactors in Conformal Field Theory}\label{sec:localsubfCFT}
%%%

In this section we recall the notion of \emph{local conformal net}, see \eg \cite{Lo2}, which is the operator algebraic description of chiral, \ie one-dimensional, conformal field theory (CFT). Discrete conformal inclusions of local conformal nets naturally give rise to examples of local discrete subfactors, 
and their study is the original motivation of this work.

Denote by $\Sc$ the unit circle and by $\cI$ the set of open non-empty non-dense intervals $I\subset\Sc$. Denote by $\Mob = \PSU(1,1) = \SU(1,1)/\{\pm 1\}$ the group of M\"obius transformations that preserve the complex upper half-plane and act on $\Sc$ by complex fractional linear diffeomorphisms.

\begin{defi}  
A \textbf{local conformal net} is a collection of von Neumann algebras parametrized by the intervals in $\cI$, $\A = \{\A(I) : I\in\cI\}$, acting on a common (separable) Hilbert space $\Hil = \Hil_\A$ and subject to the following restrictions. Let $I_1,I_2\in\cI$
\begin{itemize}
\item[$(i)$] \emph{Isotony}. $\A(I_1)\subset\A(I_2)$ if $I_1\subset I_2$.
\item[$(ii)$] \emph{Locality}. $\A(I_1)\subset\A(I_2)'$ if $I_1\cap I_2 = \emptyset$.
\item[$(iii)$] \emph{M\"obius covariance}. There is a strongly continuous unitary representation $U:\Mob \to \cU(\Hil)$ which acts covariantly on the net, \ie $U(g) \A(I_1) U(g)^* = \A(gI_1)$ for every $g\in\Mob$. 
\item[$(iv)$] \emph{Positivity of the conformal Hamiltonian}: the infinitesimal generator $H$ of the rotations subgroup of $\Mob$ has non-negative spectrum.
\item[$(v)$] \emph{Vacuum vector}. There exists a unique, up to scalar multiples, unit vector $\Omega\in\Hil$ such that $U(g)\Omega = \Omega$ for every $g\in\Mob$. The vector $\Omega$ is cyclic for $\{\A(I) : I\in\cI\}''$.
\end{itemize}
\end{defi}

The consequences of these assumptions \cite{GuLo1992, GaFr1993, BrGuLo1993, GuLo1996, FrJr1996, GuLoWi1998} that we need are: the \emph{Reeh-Schlieder property}, \ie $\Omega$ is standard (cyclic and separating) for each $\A(I)$, the \emph{Bisognano-Wichmann property}, \ie the modular group $\sigma_t$, $t\in\RR$, and the modular conjugation $J$ associated with $(\A(I),\Omega)$ are geometric. Namely, the modular group corresponds via $U$ to the dilations subgroup of $\Mob$ which maps $I$ onto itself, the modular conjugation to the reflection of $\Sc$ with respect to the extreme points of $I$ sitting in $\Mobt = \Mob \rtimes \ZZ_2$. 
In particular, $\A(I)' = \A(I')$ where $I' := \Sc \smallsetminus \bar{I} \in \cI$ is the complementary interval of $I$.
The algebras $\A(I)$ are type $\III_1$ factors and they are called the \emph{local algebras} of $\A$. The datum of the local algebras (actually three of them associated with a tripartition of $\Sc$, \ie forming a \emph{half-sided modular factorization} \cite{GuLoWi1998}, are sufficient) together with the vacuum vector, uniquely determines the covariance representation $U$.

\begin{defi} 
An \textbf{inclusion} $\A\subset\B$ of local conformal nets $\A$ and $\B$ is a family of subfactors $\{\A(I)\subset\B(I) : I\in\cI\}$ acting on the same Hilbert space $\Hil$. $\B$ is called an \textbf{extension} of $\A$, and $\A$ a \textbf{subsystem} or \textbf{subnet} of $\B$. 
\end{defi}

We may identify $\Hil = \Hil_\B$ 
but note that $\overline{\A(I)\Omega_\B}\subset \Hil_\B = \overline{\B(I)\Omega_\B} $ is proper unless $\A(I)=\B(I)$.

\begin{defi} 
An inclusion $\A\subset\B$ as above is called a \textbf{conformal inclusion} if the M\"obius representation $U=U_\B$ of $\B$ acts covariantly on $\A$ as well, \ie if $U(g)\A(I)U(g)^* = \A(gI)$ for every $g\in \Mob$, $I\in\cI$, or equivalently if $U_\A$ extends to a M\"obius representation acting covariantly on $\B$.
In this case, $\Omega_\A = \Omega_\B$ and $\Hil_\A = \overline{\A(I)\Omega_\B}$.
\end{defi}

Following \cite{LoRe1995}, an inclusion $\A\subset\B$ is called \textbf{standard} if it admits a standard unit vector $\Omega$ for every $\B(I)$ on $\Hil_\B$ which is also standard for every $\A(I)$ acting on a fixed, independent of $I$, closed subspace $\Kil \subset\Hil_\B$. Denote by $e$ the orthogonal projection onto $\Kil$.
As in \cite{DeGi2018}, we say that the standard inclusion $\A\subset\B$ is \textbf{semidiscrete} if the formula $e m e = E_I(m) e$, $m\in\B(I)$, defines a family of normal faithful conditional expectations $E_I:\B(I)\to\A(I)\subset\B(I)$ for every $I\in\cI$ that are $\Omega$-preserving and compatible with inclusions in the sense that ${E_{I_2}}_{\restriction \B(I_1)} = E_{I_1}$ if $I_1\subset I_2$, $I_1, I_2\in \cI$. 

The following is well known, see \eg \cite[\Lem 14]{Lo2003}.

\begin{lem}
If $\A\subset\B$ is a conformal inclusion, then it is automatically standard and semidiscrete. Moreover, if $\A(I) \subset \B(I)$ is respectively discrete, irreducible, strongly relatively amenable, depth 2 or finite index for some $I\in\cI$, then the same holds for every $I$.
\end{lem}

\begin{proof}
By assumption $U = U_\B = U_\A$. Then $\Omega = \Omega_\B = \Omega_\A$ is a standard vector for $\A\subset\B$ and $\Kil = \Hil_\A$. By the Bisognano-Wichmann property, the modular group of $(\B(I),\Omega)$ leaves globally invariant $\A(I)$, for every $I\in\cI$. Thus by Takesaki's theorem \cite[\Sec 10]{Str81} $e m \Omega = E_I(m) \Omega$, $m\in\B(I)$, defines an expectation with the desired properties. The compatibility condition follows because $e$ does not depend on $I$. The second statement follows because the subfactors $\A(I_1)\subset\B(I_1)$ and $\A(I_2)\subset\B(I_2)$, $I_1,I_2\in\cI$, are isomorphic via $\Ad{U(g)}$ for some $g\in\Mob$ such that $g I_1 = I_2$.
\end{proof}

\begin{defi}\label{def:subfsdefs}
We say that $\A \subset \B$ is respectively \textbf{discrete}, \textbf{irreducible}, \textbf{strongly relatively amenable}, \textbf{depth 2} or \textbf{finite index} if the subfactor $\A(I)\subset\B(I)$ has the property for some $I\in\cI$.
\end{defi}

By \cite[\Prop 16, 17]{Lo2003}, \cite[\Thm 3.2]{LoRe1995}, the dual canonical endomorphism $\theta_I \in \End(\A(I))$ of $\A(I)\subset \B(I)$ as in Section \ref{sec:genQsys} can be extended to a \textbf{representation} of the net $\A$ on $\Hil_\A$, denoted by $\theta$, which is unitarily equivalent to the defining vacuum representation $\Hil_\B$ of $\B$ restricted to $\A$. In particular, the subfactor $\A(I)\subset\B(I)$ is always braided (Definition \ref{def:braidedsubf}) with respect to the \textbf{DHR braiding} $\varepsilon_{\theta_I^n,\theta_I^m}\in\Hom(\theta_I^{n+m},\theta_I^{n+m})$, $n,m\in\NN$, on $\langle\theta_I\rangle \subset \End(\A(I))$ \cite{DoHaRo1971}, \cite{FrReSc1992}.

Recall that a representation of the net $\A$ is a family of normal representations $\pi = \{\pi_I: I\in\cI\}$, $\pi_I : \A(I) \to \B(\Hil_\pi)$ on a fixed (separable) Hilbert space $\Hil_\pi$, such that ${\pi_{I_2}}_{\restriction \A(I_1)} = \pi_{I_1}$ if $I_1\subset I_2$. Two representations $\pi$ and $\sigma$ of $\A$ are unitarily equivalent if there is a unitary $U:\Hil_\pi \to \Hil_\sigma$ such that $U\pi_I(m)U^* = \sigma_I(m)$ for every $m\in\A(I)$, $I\in\cI$. 
Due to the type $\III$ property of local algebras, 
every representation $\pi$ of $\A$ is locally unitarily equivalent to the defining vacuum representation $\Hil_\A$, namely for every $I\in\cI$ there is a unitary $V_I : \Hil_\pi \to \Hil_\A$ such that $V_I \pi_I(m) V_I^* = m$ for every $m\in\A(I)$. Moreover, $V_I \pi_{I'}(\slot) V_I^*$ can be shown to be an endomorphism of $\A(I')$, denoted by $\rho_{I'}\in\End(\A(I'))$. If the representation is M\"obius covariant \cite[\Prop 2.1]{GuLo1996}, the dimension $d(\pi) := d(\rho_{I'})$ is independent of $I$. 

A conformal inclusion fulfilling the hypotheses of the following proposition is called \emph{of compact type}\footnote{The terminology \emph{of discrete type} would be closer to the one used in this paper. On the other hand, we might have used \emph{compact} instead of \emph{discrete} from the very beginning in Definition \ref{def:semi-discr}.} in \cite[\Def 3.2]{Ca2004}.

\begin{prop}\label{prop:localsubfinCFT}
If $\A\subset \B$ is an irreducible conformal inclusion and $\theta = \bigoplus_{[\rho], r} \rho$ where each $\rho$ is an irreducible representation of $\A$ with $d(\rho) < \infty$, then $\A\subset\B$ is discrete and $\A(I)\subset \B(I)$ is an irreducible local discrete subfactor (Definition \ref{def:localsubf}). 
\end{prop}

\begin{proof}
The inclusion is conformal, thus $\theta$ and every $\rho$ is M\"obius covariant. As in the proof of \cite[\Prop 3.3 (a)]{Ca2004}, \cf \cite[\Rmk 3.4]{Ca2004}, every $\A(I)\subset\B(I)$ is an irreducible discrete subfactor. The locality of the subfactor is the characterization of the locality of the extension $\B$ (in the sense of conformal nets) by \cite[\Prop 6.10 (i), 6.16]{DeGi2018} adapted to the chiral conformal setting. 
\end{proof}

Combining Proposition \ref{prop:localsubfinCFT} and Theorem \ref{thm:noquantum}, we get

\begin{cor}
If $\A\subset \B$ is an irreducible depth 2 conformal inclusion and $\theta = \bigoplus_{[\rho], r} \rho$ where each $\rho$ is an irreducible representation of $\A$ with $d(\rho) < \infty$, then $\A(I)\subset \B(I)$ is a compact group orbifold.
\end{cor}

The depth 2 condition is fulfilled whenever $\mathbb{G}$ is a von Neumann algebraic compact quantum group acting faithfully on $\B(I)$ such that $\A(I) = \B(I)^\mathbb{G}$, see the references in Section \ref{sec:depthtwo}.

\begin{cor}\label{cor:noquantumCFT}
If a conformal inclusion in the assumptions above arises as a compact quantum group orbifold, then it is a classical group orbifold.
\end{cor}

\begin{rmk}
The discreteness assumption in Proposition \ref{prop:localsubfinCFT} on the decomposition of $\theta$ as a representation of $\A$ is seemingly stronger than the notion of discreteness for $\A\subset\B$ considered in Definition \ref{def:subfsdefs}, given on every $\A(I)\subset\B(I)$, \cf \eqref{eq:thetadiscrete}. Recall that the second notion is automatically fulfilled for depth 2 conformal inclusions by Proposition \ref{prop:depth2discrete}.
Moreover, the two notions are both fulfilled in the finite index case \cite[\Cor 18, 19]{Lo2003}, and they coincide if $\A$ has the stronger locality property called \emph{strong additivity} \cite[\Lem 1.3]{GuLoWi1998}. 
%We shall come back to this point in \cite{BiDeGi2020-2}.
\end{rmk}

\begin{rmk}
The result of Corollary \ref{cor:noquantumCFT} stating the absence of purely \emph{quantum} global gauge group symmetries in local conformal field theory, extends the one in \cite[\Cor 1.2]{Bi2016} where the inclusions are assumed to be with finite Jones index and thus the corresponding groups are finite. A result of the same type, again in the case of finite-dimensional Hopf algebras and finite groups, was obtained in the framework of vertex operator algebras in \cite[\Thm 5.2]{DoWa2018}.
\end{rmk}

We now come to examples of compact hypergroups arising from conformal inclusions.

\begin{example}\label{ex:vir1ext}
The double coset hypergroups analyzed in Section \ref{sec:doublecosets} occur in conformal inclusions. By \cite[\Sec 3]{Ca2004}, \cite[\Sec 4]{Xu2005}, every local conformal extension of compact type (or equivalently discrete in this case) of the Virasoro net with central charge 1 \cite{Ca1998}, $\Vir_1 \subset\B$, is of this form. More precisely, $\Vir_1 = {L\SU(2)_1}^{\SO(3)}$ and $\B$ is intermediate to the loop group net extension $L\SU(2)_1$ \cite{Wa1998}, and of the form $\B = {L\SU(2)_1}^H$ for a closed subgroup $H\subset\SO(3)$. By Proposition \ref{prop:actioncosets}, we have that $K(\Vir_1(I)\subset\B(I) ) \cong \SO(3)\CS H$. \Cf Example \ref{ex:SO3SO2} where $H=\SO(2)$, in this case the extension of $\Vir_1$ is the $\U(1)$-current net studied in \cite{BuMaTo1988}.

As $\SO(3)$ is a simple group, by Proposition \ref{prop:uniqueKact} and \ref{prop:doublecosets} and by \cite[\Prop 2.1]{Ca1999}, we conclude that there are no local conformal extensions $\Vir_1\subset\B$ for which $\Vir_1 = \B^G$ for some compact group $G$, other than $L\SU(2)_1$ and the trivial one $\Vir_1$.
\end{example}

\begin{example}
Double coset hypergroups naturally arise in Minkowski, \ie 3+1-dimensional, quantum field theory. By \cite[\Sec 4]{CoDoRo2001}, every inclusion of local nets $\A\subset\B$ which is intermediate to the canonical field net extension $\cF$ of $\A$ \cite{DoRo1990}, is given by $\B = \cF^H$ for some closed subgroup $H\subset G$ of the canonical global gauge group $G$ determined by the DHR superselection sectors of $\A$. In the assumptions of \cite{CoDoRo2001}, by the results of Section \ref{sec:nonlocal} on graded-local extensions and by the uniqueness of $G$, or by Proposition \ref{prop:uniqueKact}, we get that $G \cong K(\A(I)\subset\cF(I))$ and $G\CS H \cong K(\A(I)\subset\B(I))$ by Proposition \ref{prop:actioncosets}.
This situation appears to be general in 3+1-dimensional quantum field theory \cite{CaCo2001}, \cite{CaCo2005}, and double coset hypergroups $G\CS H$
describe arbitrary irreducible local extensions, under fairly general assumptions, by \cite[\Thm 5.2]{CaCo2005}.
\end{example}

\begin{example}
Examples of finite hypergroups arising from conformal inclusions which are neither groups nor double cosets of groups, but double cosets of fusion rings, can be found in the finite index case \cite[\Sec 4.6]{Bi2016}. \Eg the inclusion $L\SU(2)_{10} \subset L\Spin(5)_1$ has index $3+\sqrt{3}$ and hypergroup $K(L\SU(2)_{10}(I) \subset L\Spin(5)_1(I)) =: K_{e,x,2+\sqrt{3}}$ consisting of two points $e$, $x$ with $\delta_x^\sharp = \delta_x$ and $\delta_x\ast \delta_x=\frac{1}{2+\sqrt{3}}\delta_e+\frac{1+\sqrt{3}}{2+\sqrt{3}}\delta_x$.
\end{example}

\begin{example}
Examples of infinite compact hypergroups  arising from conformal inclusions which are \emph{not} double coset hypergroups can be obtained by taking tensor products of conformal inclusions, see Section \ref{sec:newdiscretesubf}. \Eg $\SO(3)\CS \SO(2) \times K_{e,x,2+\sqrt{3}}$ arises from a conformal inclusion. 
\end{example}

\begin{rmk}
  Let $G$ be a compact metrizable group and $H\subset G$ be a closed subgroup,
  then the representations of the double coset hypergroup $G\CS H$ have integer hyperdimension (Section \ref{sec:repK}). 
  This follows from Theorem \ref{thm:hyperdim} by taking a minimal action of $G$ on a factor $\M$ 
  and considering the subfactor $\M^G\subseteq\M^H$, and using the results 
  from Section \ref{sec:doublecosets}.\footnote{One can show directly that the hyperdimensions of the representations of $G\CS H$ are integer. We thank Massoud Amini for providing us a proof of this fact.}
\end{rmk}

It is an interesting problem to find subfactors with $K(\N\subset\M)$ not a product of $G\CS H$ with a finite hypergroup. This is for example the case when the dual canonical endomorphism $\theta$ has an irreducible subendomorphism with \emph{non-integer} dimension, or equivalently when the associated hypergroup has a representation with non-integer hyperdimension, and $K(\N\subset\M)$ is either not a product or connected.

\begin{conjecture}
Regarding further occurrences of infinite exotic compact hypergroups arising in CFT, we give two families of possible candidates. The first one comes from considering loop group nets, their compact group orbifold subnets and their finite index extensions, \eg
\begin{align}
{L\SU(2)_{10}}^{H}\subset L\SU(2)_{10} \subset L\Spin(5)_1
\end{align}
for $H\subset \SO(3)$. 
The second family comes from the coset construction of conformal nets, \eg
\begin{equation}
{L\SU(2)_{N+1}}^{H} \otimes \Vir_{c_{N-3}} \subset L\SU(2)_{N+1} \otimes \Vir_{c_{N-3}} \subset L\SU(2)_N\otimes L\SU(2)_1
\end{equation}
for $H\subset\SO(3)$ and $c_{N-3} = 1-\frac{6}{(N-3)(N-2)}$, \cf \cite[\Sec 4.3]{Xu2000coset}.\footnote{We thank Sebastiano Carpi for pointing out this second family of candidates.}

At present we are not able to prove that these inclusions are discrete. Instead, we can show in general that the composition of a discrete subfactor on top of a finite index one, thus in the order opposite to the one we have in the previous examples, is always discrete (Section \ref{sec:newdiscretesubf}). 
\end{conjecture}

%%%
\subsection{Constructions of discrete subfactors}\label{sec:newdiscretesubf}
%%%

In this section we provide some constructions of new discrete subfactors from old ones (Definition \ref{def:semi-discr}), and we draw some conclusions in the local case (Definition \ref{def:localsubf}).

If $K_1$ and $K_2$ are compact hypergroups in the sense of Definition \ref{def:CompactHypergroup}, the direct product compact topological space $K_1\times K_2$ is naturally a hypergroup by setting $\delta_{(x,y)}\ast\delta_{(z,t)} := \delta_{x}\ast\delta_{z} \times \delta_{y}\ast\delta_{t}$, ${\delta_{(x,y)}}^\sharp := \delta_{(x^\sharp,y^\sharp)}$, \cf \cite[\Def 1.5.29]{BlHe1995}. Indeed, the convolution on probability measures is biaffine by definition, thus uniquely determined by the convolution on Dirac measures via \eqref{eq:convondirac}. The identity element and the Haar measure are defined by $e_{K_1\times K_2}:= (e_{K_1},e_{K_2})$ and $\mu_{K_1\times K_2} := \mu_{K_1} \times \mu_{K_2}$, thus \eqref{eq:haarproperty1} and \eqref{eq:haarproperty2} hold by Fubini's theorem.

\begin{prop}
Let $\N_1\subset\M_1$ and $\N_2\subset\M_2$ be discrete type $\III$ subfactors. Then $\N_1\otimes\N_2\subset\M_1\otimes\M_2$ is discrete. If in addition the subfactors are irreducible and local, then the same is true for their tensor product and we have $K(\N_1\otimes\N_2\subset\M_1\otimes\M_2) \cong K(\N_1\subset\M_1)\times K(\N_2\subset\M_2)$. 
\end{prop}

\begin{proof}
The dual canonical endomorphism of the tensor product subfactor is $\theta_{\N_1\subset\M_1} \otimes \theta_{\N_2\subset\M_2}$ in $\End(\N_1\otimes\N_2)$, thus 
discreteness follows from the characterization \eqref{eq:thetadiscrete}. Being irreducible, braided and local clearly passes to tensor products. 
The last statement follows by observing that $K(\N_1\subset\M_1)\times K(\N_2\subset\M_2)$ acts faithfully on $\M_1\otimes\M_2$ in the sense of Definition \ref{def:action} via $(\phi_1,\phi_2) \mapsto \phi_1\otimes\phi_2$, with fixed points $\N_1\otimes\N_2$.
Thus the uniqueness part of Theorem \ref{thm:genorbi} applies.
\end{proof}

More examples come from composing subfactors as follows:

\begin{prop}
Let $\N\subset \cP\subset \M$ be a composition of type $\III$ subfactors with $[\cP:\N] < \infty$ and $\cP\subset \M$ discrete, then $\N\subset \M$ is discrete.
\end{prop}

\begin{proof}
Let $\theta_{\cP\subset \M}=\bigoplus_{[\rho],r} \rho$ with $d(\rho) < \infty$, then $\theta_{\N\subset \M} = \bar\iota_{\N\subset \cP} \, \theta_{\cP\subset \M}\, \iota_{\N\subset \cP}$ and the statement follows from $d(\bar\iota_{\N\subset \cP}\,\rho\,\iota_{\N\subset \cP}) = d(\iota_{\N\subset \cP})^2 d(\rho) = [\cP:\N] d(\rho) < \infty$.
\end{proof}

We do not know whether a statement similar to the above with $\N\subset \cP$ discrete and $\cP\subset \M$ finite index holds, namely if $\N\subset \M$ is necessarily discrete or not at this level of generality. If $\N\subset \cP$ and $\cP\subset \M$ are both discrete (and neither of them with finite index) instead, we know that $\N\subset \M$ need not be discrete. Examples of this also arise from conformal inclusions, \cf \cite[\Thm 4.6]{Xu2005} and \cite[\Thm 3.5]{Ca2004}.

\bigskip
\noindent
{\bf Acknowledgements.}
We would like to thank Massoud Amini, Claire Anantharaman-Delaroche, Dietmar Bisch, Sebastiano Carpi, Corey Jones, Roberto Longo, David Penneys, Stefano Rossi, Reiji Tomatsu and Makoto Yamashita for discussions. We gratefully acknowledge support from the Simons Center for Geometry and Physics, Stony Brook University (during the program \lq\lq Operator Algebras and Quantum Physics\rq\rq) and the Mathematisches Forschungsinstitut Oberwolfach (in the occasion of the workshop 1944 \lq\lq Subfactors and Applications\rq\rq) where some of the research for this paper was performed and presented. We also acknowledge support from the NSF DMS grant 1641020 during the workshop 2018 AMS MRC \lq\lq Quantum Symmetries: Subfactors and Fusion Categories\rq\rq\ and from the \lq\lq MIUR Excellence Department Project awarded to the Department of Mathematics, University of Rome Tor Vergata, CUP E83C18000100006\rq\rq.

%\bibliography{cft.bib}

\def\cprime{$'$}\newcommand{\noopsort}[1]{}

\vspace{0mm}

\address

\end{document}